\documentclass[11pt]{article}
\usepackage{graphicx}\usepackage[USenglish]{babel}
\usepackage[T1]{fontenc}
\usepackage[ansinew]{inputenc}
\usepackage{indentfirst}
\usepackage{graphicx}
\usepackage{amsmath}
\usepackage{amsthm}
\usepackage{amsfonts}
\usepackage{amssymb}
\usepackage{dsfont}
\usepackage{enumitem} 
\usepackage{hyperref}
\usepackage{color}
\usepackage{amsmath}
\usepackage{tikz}
\usetikzlibrary{arrows.meta}
\usepackage{fullpage}
\usepackage{tikz}
\usepackage{float}
\usepackage{subfigure}
\usetikzlibrary{calc} 
\RequirePackage[capitalize,nameinlink,noabbrev]{cleveref}
\usepackage{aliascnt}

\DeclareMathOperator{\Tr}{Tr}
\DeclareMathOperator{\Diag}{Diag}
\DeclareMathOperator{\Id}{Id}
\DeclareMathOperator{\Span}{Span}
\DeclareMathOperator{\Ima}{Im}
\DeclareMathOperator{\sign}{sign}
\DeclareMathOperator{\cc}{cc}
\DeclareMathOperator{\second}{\prime\prime}
\DeclareMathOperator{\Sp}{Sp}

\DeclareMathOperator{\CC}{\mathbb{C}}
\DeclareMathOperator{\NN}{\mathbb{N}}

\newcommand*\norme[1]{\|#1\|}
\newcommand{\xxx}{\mathrm{x}}
\newcommand{\ccc}{\mathrm{c}}
\newcommand{\sss}{\mathrm{s}}
\newcommand{\abs}[1]{\left| #1 \right|}

\theoremstyle{plain}
\newtheorem{thhm}{Theorem}[section]
\newtheorem{thm}[thhm]{Theorem}
\newtheorem{lem}{Lemma}[section]
\newaliascnt{prop}{thhm}
\newtheorem{prop}[prop]{Proposition}
\newaliascnt{conj}{thhm}
\newtheorem{conj}[conj]{Conjecture}
\newtheorem{defi}{Definition}[section]

\theoremstyle{remark}
\newtheorem{rem}{Remark}[section]
\newtheorem{ex}{Example}[section]

\aliascntresetthe{prop}
\aliascntresetthe{conj}

\crefname{thhm}{Theorem}{Theorems}

\crefname{prop}{Proposition}{Propositions}
\crefname{conj}{Conjecture}{Propositions}

\title{Spectral gap of a convex combination of a random permutation and a deterministic matrix}
\author{Sarah Timhadjelt}

\begin{document}
	
	\maketitle
	\begin{abstract}
		We study the spectral gap behavior of an operator obtained by summing a random permutation $M$ and a
		deterministic bistochastic matrix $Q$. We are interested in the asymptotic in terms of dimension. In the case where $(M,Q)$ are asymptotically free with amalgamation over the diagonal, we can compute  limit operators $(u,q)$ which give the weak limit spectral distribution. Therefore we introduce free with amalgamation operators that are suitable for computing the
		spectral gap limit of our operator in high dimensions. We then approximate the spectral radius of the corresponding limit operator and finally give an upper bound for the spectral
		radius of the finite-dimensional operator. In particular, we show that if the deterministic matrix underlying graph is an expander, then the underlying Markov chain associated to the sum with a random permutation is again an expander.
		
	\end{abstract}
	\tableofcontents 
	
	\section{Introduction}\label{presentation}
	\subsection{Notations and conventions}
	
	For all integer $N \ge 1$ we set $[N] := \{1,..,N\}$, $\mathrm{M}_N(\CC)$ the algebra of complex matrices of dimension $N$, $\mathrm{D}_N(\CC)$ its subalgebra of diagonal matrices and $\mathrm{S}_N \subset \mathrm{M}_N(\CC)$ the symmetric group. We denote $\Delta: \mathrm{M}_N( \CC) \rightarrow \mathrm{D}_N(\CC)$ the projection defined by $\Delta(B) = (\delta_{ij}b_{ij})_{i,j \in [N]}$ where $\delta_{ij} = \mathds{1}(i=j)$ and $B = (b_{ij})_{i,j \in [N]} \in \mathrm{M}_{N}(\CC)$. We denote $\Tr(\cdot)$ the classic trace operation on matrices and $\Tr_N(\cdot) := \frac{1}{N} \Tr(\cdot)$ the normalized trace.  We denote by $\abs{\cdot}$ the cardinality of a set and the usual absolute value. We denote by $\mathbb{P}(\cdot)$ and $\mathbb{E}(\cdot)$ the probability and expectation under the uniform measure on $\mathrm{S}_N$. We denote $\mathrm{M}_N(L^\infty)$ the algebra of complex random matrices of dimension $N$ with essentially bounded entries and $\mathrm{D}_N(L^\infty)$ its subalgebra of diagonal matrices. If $M$ is a random permutation matrix taken uniformly we denote $M \sim \mathcal{U}\left(\mathrm{S}_N\right)$.\\
	
	We denote $\mathrm{B}(\mathrm{H})$ the algebra of bounded operator on the Hilbert space $\mathrm{H}$. For a bounded operator $b \in \mathrm{B}\left(\mathrm{H}\right)$ where $\mathrm{H}$ is an Hilbert space, we denote by $\|b\|$ its operator norm associated to the inner scalar product norm of $\mathrm{H}$ and we denote by $b^\ast$ the adjoint operator of $b$. For $\epsilon  \in \{+,-\}$, we set $b^\epsilon = b$ if $\epsilon = +$ and $b^\epsilon = b^\ast$ if $\epsilon = -$. Regarding the spectrum, for $B$ an operator on a Hilbert space of dimension $N$, we consider the eigenvalues with multiplicities. We denote the $k$-th eigenvalue by $\lambda_{k}(B)$ and order them as follows
	\begin{equation*}
		\abs{\lambda_{1}(B)} \ge \abs{\lambda_{2}(B)} \ge \cdots \abs{\lambda_N(B)}.
	\end{equation*}
	We introduce the following sparsity parameters of $B \in \mathrm{M}_N(\mathbb{C})$ 
	
	\begin{align*}
	d= d(B) &:= \underset{x \in [N]}{\max}|\{y \in [N]~;~ B_{xy} \ne 0\}|\vee \underset{y\in [N]}{\max}|\{x\in [N]~;~ B_{xy} \ne 0\}|.
	\end{align*}
For any integer $ L \ge 0$, we introduce the additional sparsity parameter
	\begin{equation}\label{eq:kappa_p}
		\kappa_{L} (B)  :=  \underset{\substack{k \le p \le L,\\ (x_t,y_t) \in [N]^{2k}}}{\max}~\left\{\max_{\substack{n_1 +\cdots +n_k = p \\ n_t \le \ell(B)}}~ \prod_{t=1}^{k}(B^{n_t})_{x_t y_t}\right\}^{\frac{1}{p+k}}, 
	\end{equation}
	where $\ell(B):= \frac{\ln(N)}{7\ln(d(B)}$ and we set
	\begin{equation}\label{eq:kappa}
		\kappa(B):= \sup_{L} \kappa_L(B).
	\end{equation}
	We will omit the $B$ in the notation $\kappa_L$ and $\kappa$ if there is no confusion. 
	
	\subsection{Model}
	In this article we will consider the following model of matrix $ Q \in \mathrm{M}_N (\mathbb{C})$.
	\begin{enumerate}[label=(H:\arabic*)]
		\item \label{H:1} For all $x,y \in [N]$, $Q_{xy} \ge 0$ and the vector $\boldsymbol{1}:= (1,\cdots,1)^\top \in \mathbb{R}^{N}$ is an eigenvector of $Q$ and $Q^{\ast}$, i.e.
		\begin{equation*} 
			Q \boldsymbol{1} = \delta \boldsymbol{1} \quad \text{ and } \quad Q^{\ast} \boldsymbol{1} =  \delta \boldsymbol{1}
		\end{equation*}
		for some $\delta >0$.
		\item \label{H:2} There exists $K > 0$ independent of $N$ such that $\|Q\| \le K$.
		\item \label{H:3} There exists $0 < \alpha < 1$ such that we have 
		$$d:= \underset{x \in [N]}{\max}|\{y \in [N]~;~ Q_{xy} \ne 0\}|\vee \underset{x \in [N]}{\max}|\{y \in [N]~;~ Q^\ast_{xy} \ne 0\}| = e^{o\left(\ln(N)^{\alpha} \right)}.$$
	\end{enumerate}
	
	We denote 
    \begin{equation}\label{eq:defi_Q_perp}
    Q_{\perp}:= Q - \frac{1}{N} \boldsymbol{1}^{\top}\boldsymbol{1}.
    \end{equation}
    From \ref{H:1} we have that $Q$ and $Q_{\perp}$ have a triangular representation in a common basis.
    We see $Q$ as an adjacency matrix for a $\delta$-regular graph. We then take interest in the operators, denoted $A_{\xxx}$ for $\xxx \in \{ \sss,\ccc\}$. The first operator is given by the following sum
	\begin{equation}\label{eq:1origins}
		A_{\sss}= M+Q
	\end{equation}
	where $M\sim \mathcal{U}(\mathrm{S}_N)$. In the case where $\delta =1$, $Q$ is called doubly stochastic. The set of bistochastic matrices being a convex polytope, any convex combination of a permutation matrix and $Q$ remains bistochastic. In this paper we also are interested in the following convex combination
	\begin{equation}\label{eq:2origins}
		A_{\ccc}:= (1-r)M+rQ
	\end{equation}
	where $M \sim \mathcal{U}(\mathrm{S}_N)$ and $0 < r < 1$. If we consider $Q$ and $M$ as transition matrices of Markov chains on the same vertex set $[N]$, $P$ is the sum of these two graphs, rescaled to obtain a Markov chain. According to Perron-Frobenius Theorem, $\lambda_{1}(A_{\mathrm{c}})=1$ (\textit{resp.} $\lambda_{1}(A_{\mathrm{s}})=2$) is then the largest eigenvalue of $A_{\mathrm{c}}$ (\textit{resp.} of $A_{\mathrm s}$). We are interested in the asymptotic behavior in dimension $N$ of the spectral gap $\lambda_{1}(A_{\mathrm{x}}) - |\lambda_{2}(A_{\mathrm{x}})|$  for $\ast\in \{\mathrm{s}, \mathrm{c}\}$. The main result of this article (see \cref{thm:bigthm_simplified} and \cref{bigtheorem}) is a probabilistic upper bounds on $|\lambda_{2}(A_{\mathrm{x}})|$ that depends only on parameters of $Q$. More precisely the bound we are interested in is the spectral radius of an operator that depends on $Q$ and $r$ and is independent with the permutation matrix $M$.
	\subsection{Graph expansion}
	
	The density of edges in relation to the number of vertices of a graph is called the sparsity of the graph. The idea of graph expansion is the fact that we want to construct growing networks in terms of vertices, control the sparsity and maintain a consequent connectivity. We call the study of the asymptotic behavior of the sparsity/connectivity ratio \textit{graph expansion}.\\
	In this paper we study the second largest eigenvalue of the operator defined by \eqref{eq:1origins} that can be seen as an adjacency matrix if $Q$ is an adjacency matrix and the operator given by \eqref{eq:2origins} that can be seen as a Markov operator if $Q$ is a Markov operator. The motivation behind the main result of this article is to prove or improve the expansion of the graph or chain defined by $Q$ by adding the random edges defined by $M$.\\
	
	To explain the link with the second largest eigenvalue of the operators $A_{\sss}$ and $A_{\ccc}$, we first consider a more local approach of expansion. We ask whether, for a given graph, there is a not too small subset of the vertex set that is poorly connected to the rest of the vertex set. In computer science this phenomenon is called a \textit{bottleneck}. The Cheeger constant of $G=(V,E)$ a finite connected undirected graph is a common quantifier of graph expansion and it is defined by the following
	\begin{equation}\label{cheegerconst}
		h(G) := \min \{ \frac{|\partial F|}{|F|}; ~ F \subset V,~ 0 < |F| \le \frac{1}{2}|V|\}
	\end{equation}
	where for $F\subset V$, $\partial F \subset E$ is the set of edges connecting a vertex of $F$ to a vertex of $V\backslash F$. For $G=(V,E)$ a graph and $ \epsilon > 0$ we say that $G$ is a $\epsilon$-\textit{expander} if we have 
	$$h(G) \ge \epsilon.$$ 
	It is the object considered as respectful of the condition of having good connectivity while being sparse. The graph expansion is also of interest when considering the simple walk on the graph. Recall that for any undirected finite graph $G=(V,E)$ one can define its adjacency operator $\mathrm{A} = \left( \mathds{1} \left( \{x,y\} \in E \right) \right)_{x,y \in V} \in \mathrm{M}_{\abs{V}}(\CC)$ but also the Markov operator $\mathrm{P}\in \mathrm{M}_{\abs{V}}(\CC)$ associated to the simple walk on $G$ defined by
	\begin{equation*}
		\mathrm{P}_{xy} := \frac{\mathds{1} \left( \{x,y\} \in E \right) }{\sum_{z \in V} \mathds{1}\left( \{x,z\} \in E \right)}.
	\end{equation*}
	The spectral gap measures the asymptotic mixing rate to equilibrium. More precisely, for $\mathrm{P}$ any transition matrix to an aperiodic and irreducible Markov chain on the set $[N]$, for any probability measure $\pi_0$ on $V$ we have
	$$\lim_{t\to +\infty} \|\pi_0 \mathrm{P}^t - \pi\|_{TV}^{1/t} = |\lambda_{2}(\mathrm{P})|,$$
	where $\pi$ is the invariant measure of $\mathrm{P}$ and for $\nu$ a measure on $[N]$, $\|\nu\|_{TV}:= \frac{1}{2} \sum_{x} \nu(x)$ denotes the total variation norm. The link between the Cheeger constant and spectral analysis is given by the following inequality.
	\begin{thm}[Cheeger Inequality]
		Let $G=(V,E)$ be a connected, $d$-regular undirected graph. We denote by $\mathrm{P}$ the Markov operator of the simple walk on $G$. We have
		$$\frac{1 - \lambda_2(\mathrm{P})}{2} \le \frac{h(G)}{d} \le \sqrt{1-\lambda_2(\mathrm{P})}.$$
	\end{thm}
	Therefore an alternative way to define expansion of graphs is the following.
	\begin{defi}[Spectral-expansion]\label{defi:expander}
		A $d$-regular graph $G=(V,E)$ is a $\epsilon$-\textup{expander for eigenvalues} if
		$$\lambda_2(\mathrm{P}) \le 1-\epsilon.$$
		where $\mathrm{P}$ is the Markov operator of the simple walk on $G$.
	\end{defi}
	Even though we take interest in this article in non-Hermitian operators, for the reader understanding we only gave definitions for undirected graphs in this current section. Nevertheless we can similarly define the Cheeger constant and link it to the spectral expansion of a directed graph (see for instance \cite[Theorem 5.1]{chung2005laplacians}). 
	\subsection{Free probability and spectral analysis of random matrices}\label{section:freepbandspecanal}
	
	Now that we gave some motivation for the study of spectrum in graph analysis, we introduce a tool for asymptotic spectral analysis for possibly random graphs. Roughly speaking, the notion
	of freeness in non-commutative probability is the non-commutative equivalent of independence for
	(classical commutative) random variables and therefore it is the suitable framework to study the asymptotic behavior of matrices spectrum. When considering independent non-commutative variables, such as random matrices, this independence is indeed translated into freeness asymptotically. Voiculescu \cite{voiculescu1991limit} introduced free probability with the idea of constructing spaces of operators imitating basic probability theory and we refer to Anderson, Guionnet and Zeitouni \cite{ander2010intro} for a comprehensive study of random matrices.
	
	\paragraph{Freeness and convergence in distribution.} We have \textit{convergence in distribution} for a couple $\mathbf{Q}=(Q_1,Q_2)$ of random matrices if for all $P$ non-commutative symmetric polynomial the sequence $\left( \Tr_N(P(\mathbf{Q}, \mathbf{Q}^\ast) \right)_{N\in \NN}$ converges.  Assuming some control over the operator norms of the family $\mathbf{Q}$ this convergence implies the existence of a family $\mathbf{q} = (q_1,q_2) \in \mathrm{B}(\mathrm{H})^2$ of self-adjoint operators on a Hilbert space such that for all polynomial $P$ we have $\lim_{N\rightarrow \infty} \Tr_N (P(\mathbf{Q}, \mathbf{Q}^\ast)) = \langle P(\mathbf q, \mathbf{q}^\ast)\xi, \xi\rangle$  where $\xi \in \mathrm{H}$ verifies $\|\xi \|_{2}=1$. The linear form $\tau:= \langle \cdot \xi, \xi \rangle : \mathrm{B}(\mathrm{H}) \rightarrow \CC$ has to be seen as the non-commutative analog of the expectation $\mathbb{E}(\cdot)$ in the classic probability setting. For the previous property of the convergence in distribution of random matrices one can refer to \cite[Corollary 5.2.16]{ander2010intro}. In term of spectral analysis, considering $P(\cdot)$ a non-commutative symmetric polynomial we can apply the spectral theorem to the operator $\mathrm{b} =P(\mathbf q , \mathbf{q}^\ast)$ and denoting $\mathrm{P}= P( \mathbf Q, \mathbf{Q}^{\ast})$ it means that there exists a unique probability measure $\mu_{\mathrm{b}}$ such that
	\begin{equation}\label{eq:spectralasymp}
		\mu_{\mathrm{P}} := \frac{1}{N} \sum_{k=1}^N \delta_{\lambda_k(\mathrm{P})} \underset{N\rightarrow \infty}{\rightharpoonup} \mu_{\mathrm{b}}
	\end{equation}
	where the previous convergence is weak. In particular if one can compute the distribution $\mu_{\mathrm{b}}$ and denoting $\rho(\mathrm{b})$ its spectral radius we have that for all $\delta >0 $ there exists $\epsilon>0$ such that for all $N$ large enough
	\begin{equation}\label{eq:optimal_exp}
		\sigma(\mathrm{P}) \cap [-\rho(\mathrm{b}) + \delta, \rho(\mathrm{b}) - \delta ] \ge \epsilon N.   
	\end{equation}
	In particular it gives the best expansion one can get for a graph with an adjacency matrix given by $\mathrm{P}$.\\
	The \textit{freeness} of the family $\mathbf{q}$ of operators (see \cite[Section 5]{ander2010intro}) has to be seen as the analog of independence in classic probability. Freeness will translate asymptotically the independence and/or invariance of laws of random matrices and allow computability of the measure $\mu_{\mathrm{b}}$ (see \cite[Section 5.3.3]{ander2010intro} and for the non-Hermitian case \cite{biane1999computation}). To see the interest of asymptotic freeness in our setting, recall that we take interest in a random matrix, given by either \eqref{eq:1origins} or \eqref{eq:2origins}, of the form
	$$A = U + Q$$
	where $Q$ is deterministic and $U$ random. We then have the following theorem.
	\begin{thm}[Voiculescu, 1991]\label{thm:asymptfree(Q,U)}
		Let $\mathrm{G}_N$ be either the orthogonal group $\mathrm{O}_N$ or the unitary group $\mathrm{U}_N$. If $Q\in \mathrm{M}_{N}(\CC)$  converges in distribution and is invariant under $\mathrm{G}_N$ conjugation and $U \in \mathrm{G}_N$ is Haar distributed, then the couple $(Q,U)$ is asymptotically free.
	\end{thm}
	If we consider that $Q$ is in the setting of the previous theorem and $U$ is Haar distributed unitary independent of $Q$, then the asymptotic spectral measure of $P=\frac{1}{4}\left(Q+ Q^\ast + U + U^\ast\right)$ is completely characterized by the asymptotic distribution of $Q$. The asymptotic freeness had been proved for a family of Haar distributed unitaries of $\mathrm{U}_N$ by Voiculescu \cite{voiculescu1991limit} in 1991 and for random permutations by Nica \cite{nica1993asymptotically} in 1993. In this article we want to consider random graphs, and the fact that we softened the regularity of the distribution for the random matrix $Q$ by considering random adjacency operators weakens the freeness of the asymptotic operators. 
	\paragraph{Freeness with amalgamation.} This new freeness is called \textit{freeness with amalgamation} and we give here the definition. These notions can be found in the article by Au, C{\'e}bron, Dahlqvist, Gabriel and Male \cite{au2018large} or in Voiculescu, Dykema, and Nica \cite{voiculescu1992free}.
	\begin{defi}\label{def1}
		We give first the setting and definition of freeness with amalgamation.
		\begin{itemize}
			\item An \textup{Operator valued probability space} denoted $(\mathcal{A},\mathcal{B}, \mathrm{E},\ast)$ consists in a unital algebra $\mathcal{A}$, an involution $\ast: \mathcal{A} \rightarrow \mathcal{A}$, a unital sub-algebra $\mathcal{B} \subset \mathcal{A}$ stable under $\ast$, and a \textup{conditional expectation} $\mathrm{E} : \mathcal{A} \to \mathcal{B}$ which is a unit-preserving linear map such that $\mathrm{E}(b_1 a b_2 )= b_1 \mathrm{E}(a) b_2$ for any $a \in \mathcal{A}$ and $b_1,b_2 \in \mathcal{B}$.
			\item We write $\mathcal{B} \langle X,X^\ast \rangle$ for the algebra generated by $\mathcal{B}$  and the freely non-commuting indeterminates $(X,X^\ast)$. For a monomial $P(X, X^\ast)= b_0 X^{\epsilon_1} b_1 X^{\epsilon_2} \cdots X^{\epsilon_n} b_n \in \mathcal{B}\langle X, X^{\ast}\rangle$ the \textup{degree} of $P(X,X^\ast)$ is $n$, its \textup{coefficients} are $b_0,...,b_n \in \mathcal{B}$ and for all $1 \le k \le n$ we have $\epsilon_k\in \{+,-\}$. Then the \textup{operator valued distribution}, or $\mathrm{E}$-distribution, of $a\in \mathcal{A}$ is the map given by
			$$\mathrm{E}_{a} ~:~ P(\cdot) \in \mathcal{B}\langle X,X^\ast \rangle \mapsto \mathrm{E}[P(a,a^\ast)] \in \mathcal{B}.$$
			\item Two elements $a_1,a_2 \in \mathcal{A}^2$ are called \textup{free with amalgamation over $\mathcal{B}$} whenever for any $n \ge 2$, any $\ell_1 \neq \ell_2 \neq \cdots  \neq\ell_n \in \{1,2\}$ and any monomials $P_1, ..., P_n \in \mathcal{B}\langle X, X^\ast\rangle$, one has
			$$\mathrm{E}\left[ \left( P_1(a_{\ell_1},a_{\ell_1}^\ast) - \mathrm{E}\left[P_1 (a_{\ell_1},a_{\ell_1}^\ast)\right] \right) \cdots \left(P_n (a_{\ell_n},a_{\ell_n}^\ast) - \mathrm{E} \left[P_n(a_{\ell_n},a_{\ell_n}^\ast)\right]\right)\right]=0.$$
		\end{itemize}
		We denote $(\mathcal{A},\mathcal{B}, \mathrm{E},\Phi,\ast)$ when the following is verified.
		\begin{itemize}
			\item The $4$-tuple $(\mathcal{A},\mathcal{B}, \mathrm{E},\ast)$ is an operator-valued probability space.
			\item  The triple $(\mathcal{A},\Phi,\ast)$ is a non-commutative probability space with a faithful trace, i.e. we have $\Phi : \mathcal{A} \rightarrow \CC$ linear, verifying $\Phi(1) = 1$ and 
			$$\Phi(aa^\ast) = 0 \Rightarrow a =0.$$
			\item Finally we have $\Phi \circ \mathrm{E} = \Phi$.
		\end{itemize}
		The triplet $(\mathcal{A},\Phi, \ast)$ is called a \textup{non-commutative probability space} and the \textup{$\ast$-distribution} of $a\in \mathcal{A}$ is the map 
		\begin{equation*}
			\Phi_{a} : P(\cdot) \in \CC\langle X,X^{\ast} \rangle \mapsto \Phi[P(a,a^{\ast})] \in \CC.
		\end{equation*}
	\end{defi}
	Note that $(\mathrm{M}_N(\mathbb{C}),\mathrm{D}_N (\mathbb{C}),\Delta,\Tr_N,\ast)$ and $(\mathrm{M}_N(L^\infty),\mathrm{D}_N(L^\infty),\Delta,\mathbb{E} \circ\Tr_N,\ast)$ are operator-valued probability spaces. With no loss of generality here, any operator-valued probability space can be viewed as a subalgebra of a $\mathrm{B}( \mathrm{H})$, set of bounded operators on a Hilbert space $\mathrm{H}$. This freeness has a particular interest in random graph theory because the previous theorem holds again replacing freeness with freeness with amalgamation and invariance under $\mathrm{G}_N$ conjugation by invariance under $\mathrm{S}_N$ conjugation.
	\begin{thm}\cite[Theorem  1.3]{au2018large}\label{thm:permutinvasfreewitham}
		If $Q \in \mathrm{M}_{N}(\CC)$  converges in distribution and is invariant under $\mathrm{S}_N$ conjugation and $M \in \mathrm{S}_N$ is uniformly distributed, then the couple $(Q,M)$ is asymptotically free with amalgamation over the diagonal, i.e the subalgebra of diagonal matrices.
	\end{thm}
	Once again the convergence in distribution and the freeness with amalgamation of the limit operators allows to compute the distribution of the limit family. An explicit algorithm to compute the joint distribution of free with amalgamation operators is given by Belinschi, Mai, Speicher \cite{belinschi2017analytic}. The previous setting is better suited for graph theory applications, since we can compute random graphs which joint entry wise law is not invariant under $\mathrm{O}_N$ or $\mathrm{U}_N$ conjugation but under $\mathrm{S}_N$. Indeed one can consider $Q$ the adjacency matrix (\textit{resp.} the Markov operator of the simple walk) of an Erd\H{o}s-R{\'e}nyi graph, then the couple $(Q,M)$ is asymptotically free with amalgamation and the asymptotic spectral distribution of the operators defined by \eqref{eq:1origins} and \eqref{eq:2origins} can be computed. Again, in an Hermitian case where we consider $Q$ a deterministic hermitian matrix invariant under $\mathrm{S}_N$ conjugation and $M$ a uniformly distributed random permutation, considering 
	$$\mathrm{P} := Q+ M+ M^\ast$$
	one can apply \cref{thm:permutinvasfreewitham} and denoting $(q,u)$ the limit free with amalgamation operators for the joint distribution of $(Q,M)$ we have again
	\begin{equation*}
		|\sigma(\mathrm{P}) \cap [-\|\mathrm{b}\|,\|\mathrm{b}\|]^\mathrm{c}| = o (N)
	\end{equation*}
	where $\mathrm{b}= q + u + u^{\ast}$.\\\
	In our case, not only we do not assume any asymptotic freeness of the couple $(Q,M)$, but we also consider non-Hermitian operators \eqref{eq:1origins} \eqref{eq:2origins}. Nevertheless we show that we have expansion under some assumptions on $Q$ distribution (\ref{H:1}, \ref{H:2}, \ref{H:3}). The fact that we use an approximation of the spectral radius of the sum of free with amalgamation operators will then allow us to say that the expansion we prove here is optimal, up to showing the weak convergence of the empirical spectral measure towards the \textit{brown measure} of the limit operator. Computing the \textit{brown measure} radius for the limit object is possible in some cases \cite{biane1999computation}. Also simulations lead to think that it is indeed this radius that gives the asymptotic second largest eigenvalue in modulus (see \cref{someapplication}, \cref{spectre_M+Q} and \cref{fig1}).

Whether one has asymptotic freeness or freeness with amalgamation of $(Q,M)$, the fact that we are dealing with non-Hermitian operators invalidates the spectral theorem for the operator $p = (1-r)u + rq$. One can then ask if Girko's method applies to show the convergence \eqref{eq:spectralasymp} with $\mu_{\mathrm{b}}$ the brown measure of the operator $\mathrm{b}$ (see for instance \cite{bordenave2012around,rudelson2008littlewood} for $Q$ with independent entries, \cite{basak2018circular} for $P= \frac{1}{d} \sum_{s} M(s)$ with $(\ln(N)^{12}/\ln(\ln(N))^4 \le  d = o (N)$). This convergence would then depends on the distribution chosen for $Q$. In the case where we have the convergence \eqref{eq:spectralasymp} we then can conclude that the expansion given in this article is optimal in regard of \eqref{eq:optimal_exp}.
\paragraph{Strong asymptotic freeness.}
	In any case, the convergence given in \eqref{eq:spectralasymp} is not sufficient to conclude about the expansion of operators. In fact, it is precisely the absence of \textit{outliers} that is needed to satisfy the expander \cref{defi:expander}. One way to obtain these expansions is to show \textit{strong} asymptotic freeness. This strong convergence had been shown for $\mathbf{M}=(M(s))_{s\in [d]}$ a family of iid uniformly distributed random permutation by Bordenave and Collins \cite{CBeigen}. In particular it means that the random regular graph defined by the adjacency operator $A= \sum_{s=1}^{d} M(s)+ M(s)^{\ast}$ has the best expansion possible, i.e it verifies that for all $\epsilon >0$ we have that for $N$ large enough $\lambda_{2}(A) \le 2 \sqrt{2d-1} + \epsilon$. The fact that the best expansion possible is given by $2\sqrt{2d-1}$ can be seen as a consequence of the asymptotic freeness given by Nica \cite{nica1993asymptotically} as the spectral radius of the limit operator $\mathrm{b}$ in the convergence \eqref{eq:spectralasymp}, or directly by the Alon-Boppana bound of regular graph (see for instance \cite{Mohar_2010}). 
	We now give a simplified version of \cref{bigtheorem}. 
	
	\begin{thm}\label{thm:bigthm_simplified}
		Let $Q$ be a bistochastic matrix and $M$ an independent uniformly distributed random permutation.  There exists $(q,u)$ a pair of free with amalgamation operator such that $u$ is unitary and $q$ and $Q$ have same $\ast$-distribution asymptotically. Besides for any $0 < r < 1/3$ and $0 < c < 1$ and for all $\epsilon >0$ we have 
		\begin{equation*}
			\mathbb{P} ( \abs{\lambda_2((1-r)M + rQ)} \le (1+ \epsilon)\rho((1-r)u + rq)) \ge N^{-c}
		\end{equation*}
		for $N$ large enough.
	\end{thm}
	The previous simplified version is the case discussed in \cref{section:convcomb}. If one can compute $\rho=\rho((1-r) u + rq)$ in the previous theorem and if $\rho < 1$ this gives us expansion for the random walk defined by $(1-r)M + rQ$. To prove then that this expansion is optimal, one way would be to prove the convergence of the empirical spectral measure towards a measure which support is countained in the complex disc of radius $\rho$. In regards of \cite[Lemma 4.3]{bordenave2012around} in the case where $(M,Q)$ is asymptotically free with amalgamation, the only thing left to show that the empirical spectral converges towards the brown measure of the operator $(1-r)u + rq$ is the $\log$-\textit{integrability}, which often comes down to controlling singular values of $Q+M$. 
	Finally, an interest of looking directly at the operator $A=Q+M$ is that its structure and the computations of its moments are closer to the more general case where we study $A^\prime = a_{0} \otimes \Id_N + a_{1}\otimes Q + a_{2}\otimes M$ with $a_{i}$ fixed matrices of dimension $k$. Indeed, in many cases the strong convergence in the free case is showed using a \textit{linearization trick} saying that if the spectrum of $A^\prime$ for any choices of $(a_i)_i$ such that $A^\prime$ is Hermitian converges then we also have convergence of the spectrum of $P(M,Q)$ for any symmetric polynomial $P(\cdot)$ (see for instance \cite[Section 6]{CBeigen}). The main issue in adapting this linearization trick to our case is the fact that the matrix $a_{1}\otimes Q$ does not verify \ref{H:1}.
	
	\section{Main result}\label{section:freenesswamalgamation}
	\subsection{Existence of the free with amalgamation operators}
	We need the following lemma to enunciate our main theorem. 
	\begin{lem}\label{lem:existandconstructofqu}
		Let $Q \in \mathrm{M}_N(\mathbb{C})$ be a deterministic matrix. There exists $(\mathcal{A}, \mathcal{B}, \mathrm{E}, \Phi)$ an operator-valued probability space and $(q,u) \in \mathcal{A}$ with $u$ Haar unit operator such that
		\begin{itemize}
            \item asymptotically, the operators ${q}$ and $Q$ have the same $\ast$-distribution, 
			\item The operators $q$ and $u$ are free with amalgamation over $\mathcal{B}$.
		\end{itemize}
	\end{lem}
	The proof of the previous lemma and useful characteristics of of the couple $(q,u)$ are given in \cref{intermediateconstruct}. For any matrix $Q \in \mathrm{M}_N(\CC)$ we define the operators $\{\mathrm{b}_{\xxx}\}_ {\xxx \in \{ \sss, \ccc\}}$ as follows
	\begin{align}
		&\mathrm{b}_{\mathrm{s}}:= u+q,~~  \mathrm{b}_{\mathrm{c}}:= (1-r)u + rq\label{eq:defaandp}
	\end{align}
	where $(q,u)$ are the operators given by \cref{lem:existandconstructofqu}. The previous operators
	asymptotically bounds the spectral radius of the matrices defined by \eqref{eq:1origins} and \eqref{eq:2origins}. 
	The main result of the present article consists in upper bounding $\rho(A_{\xxx}|_{\boldsymbol{1}^\perp})$ for $\xxx \in \{\sss,\ccc\}$ defined by \eqref{eq:1origins} and \eqref{eq:2origins} by an approximation of the spectral radius of $\mathrm{b}_{\xxx}$ defined by \eqref{eq:defaandp}. More precisely, owing to Gelfand's Theorem  we have for such $\mathrm{b} \in \mathrm{B}(\mathrm{H})$
	$$\rho(\mathrm b) = \lim_{n \to \infty} \|\mathrm{b}^{n}\|^{1/n}$$
	where one has to recall here that with no loss of generality we can assume that the algebra $\mathcal{A}$ given by \cref{lem:existandconstructofqu} verifies $\mathcal{A}=\mathrm{B}(\mathrm{H})$ for some Hilbert space $\mathrm{H}$. Therefore the approximations of spectral radius we will use are given by
	\begin{equation}\label{defNell0}
		\mathrm{N}_{n,\mathrm{x}}(Q) = \|\mathrm{b}_{\xxx}^{n}\|^{1/n}
	\end{equation}
	for $\xxx \in \{ \sss, \ccc\}$ and a fixed $n \in \NN$. We will omit the dependence on $Q$ and simply denote by $\mathrm{N}_{n, \xxx}$ when there is no confusion. The fact that this norm depends only on the matrix $Q$ will be a consequence of the construction of the operators $(q,u)$ in \cref{intermediateconstruct}.
	
	\subsection{Main result}
	
	We fix $0 < r < 1$ and we set 
	\begin{equation}\label{eq:rho(ell_0)}
		\rho_{n,\sss}= \max\{\mathrm{N}_{n,\sss},2\kappa(Q),\delta ,2 \|Q_{\perp}\|\}~\text{and}~ 
		\rho_{n,\ccc}= \max\{\mathrm{N}_{n,\ccc}, 2r\kappa(Q), r\delta ,2r\|Q_{\perp}\|\}
	\end{equation}
	where $\kappa(Q)$ is given by \eqref{eq:kappa} and $ \mathrm{N}_{n,\xxx} $ by \eqref{defNell0} for $\xxx \in \{ \sss, \ccc \}$. The main theorem is the following one.
	\begin{thm}\label{bigtheorem}
		We consider $Q$ a matrix of dimension $N$ that verifies \ref{H:1}, \ref{H:2} and \ref{H:3}. Then for any $n$ integer and $0<c<1$ there exists $C>0$ such that setting
		$$\epsilon:= C\frac{ \ln(d)}{\ln(N)^{\alpha}},$$
		we have
		$$\mathbb{P}( \abs{\lambda_{2}(A_{\xxx})} \ge (1+ \epsilon) \rho_{n,\xxx}) \le N^{-c},$$
		for $M\sim \mathcal{U}(\mathrm{S}_N)$, $\xxx \in \{\sss, \ccc\}$ and $N$ large enough.
	\end{thm}
	
\section{Needed improvements}
The strategy of the proof of the main technical statement of this paper, given by \cref{big2theorem} is largely inspired by previous bounding of norms and/or spectral radius using \textit{tangle-freeness} of paths in the moment method \cite{bordenave2015new,CBeigen,Brito_Dumitriu_Harris_2022}. A good improvement of \cref{bigtheorem} would be to get rid of $\delta = \sum_y Q_{x,y}$ in the upper bound given by \cref{boundingsumq(gamma)}, which would then give the following theorem. 
\begin{conj}\label{conj:better}
\cref{bigtheorem} holds if one replaces $\rho_{n,\xxx}$ with 
\begin{equation*}
		\tilde{\rho}_{n,\sss}= \max\{\mathrm{N}_{n,\sss},2\kappa(Q),2 \|Q_{\perp}\|\}~\text{and}~ 
		\tilde{\rho}_{n,\ccc}= \max\{\mathrm{N}_{n,\ccc}, 2r\kappa(Q),2r\|Q_{\perp}\|\}.
\end{equation*}
\end{conj}

An important application of the previous statement would be for $A$ an adjacency matrix of a $d$-regular independent of $M$ a uniformly distributed random matrix. Then we would have that for $N$ large enough, $|\lambda_{2} (A+M)| \le C\sqrt{d}$ for some universal constant $C$ and therefore we have a good expansion. We give more detail about a possible strategy to obtain this conjecture in \cref{rem:conj}.
	\section{Some applications}\label{someapplication}
	In this section we give examples where \cref{bigtheorem} holds, and where the upper bound $\rho_{n,\xxx}$ defined by \eqref{eq:rho(ell_0)} is equal to the approximation of the spectral radius of the deterministic operator, i.e ${\mathrm N}_{n,\xxx}$ defined by \eqref{defNell0}. For these applications it is important to notice that owing to \cref{lem:normeplusque1} we have 
    \begin{equation*}
    \mathrm{{N}}_{n,\xxx} \ge \mathds{1}(\xxx=\sss) + (1-r)\mathds{1}(\xxx= \ccc).
    \end{equation*}
 We also give an example where an improvement of $\rho_{n,\sss}$ would be appreciated to conclude. 
	\subsection{Convex combination}\label{section:convcomb}
	In the setting of \cref{bigtheorem} and for any $Q \in \mathrm{M}_N(\mathbb{C})$ that verifies \ref{H:1}, it is sufficient to consider $0 < r < 1$ small enough to have $\rho_{n,\ccc} = {\mathrm{N}}_{n,\ccc}$ where $\rho_{n,\ccc}$ is defined by \eqref{eq:rho(ell_0)}. Indeed we have then $Q_{xy} \le \delta $ for all $x,y \in [N]$ and therefore for any integers $1\le k\le p$ and $(n_t)_{t \le k}$ such that $\sum_{t=1}^k n_t=p$ and for any $(x_t,y_t)_{t} \in [N]^{2k}$ we have
	\begin{equation}\label{majkappa}
		\prod_{t=1}^{k}(Q^{n_t})_{x_t y_t}\le \delta^{p},
	\end{equation}
 	which implies $\kappa(Q) \vee \|Q_{\perp}\| \le 1 \vee \delta$ where $\kappa(\cdot)$ is defined by \eqref{eq:kappa}. In particular for $Q$ bistochastic the upper bound is the approximation $\rho_{n,\ccc} = \mathrm N_{n,\ccc}$ for any $n \ge 0$ and for any $0 < r < \frac{1}{3}$. To then show that the spectral gap is strictly positive, one would then need to provide an upper bound for the norm  $\norme{[(1-r)u+rq]^{n}}^{1/n}$ that is strictly less than $1$. In the following application we give an explicit example where we can give such an upper bound. In the formulation of \cref{thm:bigthm_simplified} we used the fact that owing to Gelfand theorem we have that for all $\epsilon >0$ we have $\|[(1-r) u + rq]^n\|^{1/n} \ge (1- \epsilon) \rho((1-r) u + rq)$ for $n$ large enough.
	\subsection{The asymptotically free case}
	We consider $Q$ the Markov operator of a simple walk obtain with a perfect matching. One way to compute $Q$ is to take any $S \in \mathrm{S}_{N}$ permutation matrix and consider $Q=SUS^\ast$ where $U=\Diag \left(\begin{pmatrix} 0&1 \\ 1&0 \end{pmatrix} \cdots \begin{pmatrix} 0&1 \\ 1&0 \end{pmatrix}\right)$. The spectral gap for the simple walk defined by $Q$ is trivially equal to $0$ since $1$ is an eigenvalue with an eigenspace of dimension at least $2$ for $N \ge 4$. Computing the result of Biane and Lehner \cite{biane1999computation} we obtain that the spectral radius of $a=u+q$ where $q$ has same distribution as $Q$ and $u$ is free with $Q$ is $\rho =  \sqrt{3}$ (given by the equation $|\lambda|^2+1 = |\lambda^2 -1|^2$). For any $\epsilon >0$, owing to Gelfand theorem, there exists $n_{0}>0$ such that for all $n \ge n_0$ we have ${\mathrm{N}}_{n,\sss} \le \left(1 + \epsilon \right) \rho$. A realization of the spectrum of $A=M+Q$ encircled by the circle of radius $\rho$ is given in \cref{spectre_M+Q}.
	\begin{figure}[h]
		\centering
		\includegraphics[scale=0.7]{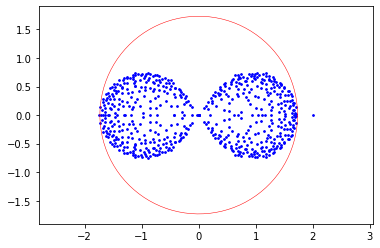}
		\caption{Plot of the eigenvalues of $A=M+Q$ with $Q$ perfect matching when $N=600$. In red the circle of radius equal to $\rho=\sqrt{3}$ the spectral radius of $u+q$ computed from the formula given by Biane and Lehner in \cite{biane1999computation}.}
		\label{spectre_M+Q} 
	\end{figure}
	However the \cref{bigtheorem} does not allow to directly conclude about the spectral gap for $A$. Indeed in this case we have $\kappa =1$ and therefore $\rho_{n,\sss} =2\kappa = 2\|Q_{\perp}\| =\lambda_1(A)$. It shows that an improvement on the factor $2$ of $\kappa$ in the definition of $\rho_n$ would be beneficial to our purposes. This factor appears in \eqref{eq:appear2kappa} in the proof of \cref{boundingsumq(gamma)}. \\
    
	However, one can consider again the convex combination case and compute the explicit asymptotical radius applying Biane and Lehner method. We rewrite 
	$$\mathrm{b}_{\ccc} = (1-r)u + rq = (1-r)\left( u + \frac{r}{1-r}q \right) = (1-r)(u+ q^\prime)$$
	and the domain of the Brown measure of $u+ q^\prime$ is contained in the curve of equation 
	\begin{equation*}
		\frac{1}{\abs{\lambda -\varphi}^2} +\frac{1}{\abs{\lambda +\varphi}^2} = 2
	\end{equation*}
	where $\varphi = r/(1-r)$ (see \cite[Equation (4.5)]{biane1999computation}). Therefore considering $r \le 1/3$ we have $\rho_{n,\ccc} = {\mathrm{N}}_{n,\ccc}$. Besides, for all $\epsilon > 0$ we have for $n$ large enough $ {{\mathrm N}}_{n,\ccc} \le \rho(p) + \epsilon$ which proves a strictly positive lower bound for the spectral gap for $N$ large enough. 
	\begin{figure}[h]
		\null\hfill
		\subfigure[r=1/3\label{fig1a}]{%
			\includegraphics[scale=0.33]{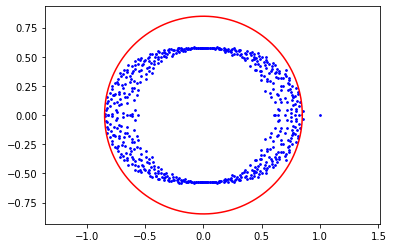}}
		\hfill
		\subfigure[r=1/4\label{fig1b}]{%
			\includegraphics[scale=0.33]{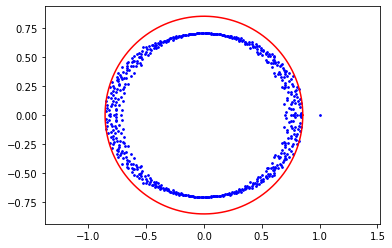}}
		\hfill\null
		\caption{In blue the eigenvalues of $P=(1-r)M+rQ$ with $Q$ perfect matching and $N=600$. In red the spectral radius of $(1-r)u + rq$.}\label{fig1}
	\end{figure}
	
	\subsection{Application to $d$-regular Ramanujan graphs}
	Let $G=(V,E)$ be a finite, connected, non-oriented $d$-regular graph with $V=[N]$. Let $B$ be its adjacency matrix and we set $Q:= \frac{1}{d}B$. In this case we have $\delta=1$. For all $n \ge 0$ integer we have
	\begin{equation*}
		Q^n = \frac{1}{N} \mathbf{1}^{\top} \mathbf{1} + Q^n|_{\boldsymbol{1}^\perp}
	\end{equation*}
	where $\mathbf{1}^{\top}  \mathbf{1}$ is the $N \times N$ matrix which entries are all equal to $1$. Setting
	$$\rho := \max\{|\mu|;~ \mu \in \Sp(Q);~ |\mu|\neq 1\}$$
	we have that for all $x,y \in [N]$
	\begin{equation}\label{eq:powersQ}
		(Q^n)_{xy} \le \frac{1}{N} + \rho^{n}.
	\end{equation}
	Besides for all $d$, there exists $c_d >0$ such that we have the following Alon-Boppana lower bound \cite{Mohar_2010}
	$$\rho \ge \frac{2\sqrt{d-1}}{d} -\frac{c_d}{\ln(N)^2}.$$
	 From the previous bound we have that $\rho \ge c\frac{2 \sqrt{d-1}}{d}$ for all integer $d$, $0 < c < 1$ and $N$ large enough. Therefore for all $n \le \frac{\ln(N)}{7\ln(d)}$ we have $\rho^{n} \ge N^{- 1/38}$. Owing to \eqref{eq:powersQ} and \eqref{eq:kappa_p} we have
	\begin{equation}\label{eq:boundkappa}
		\kappa(Q) \le \left(1 + o_{N}(1)\right)\rho.
	\end{equation}
	In particular, for any undirected $d$-regular graph verifying that $\rho < \frac{1}{2}$, for $N$ large enough we have $\rho_{n;\xxx} = \mathrm{N}_{n,\xxx}$ ans the spectral radius of $A_{\mathrm{x}}$ is upper bounded by the spectral radius of $\mathrm{b}_{\xxx}$. As an example, we can consider Ramanujan graphs which by definition verifies
	$$\rho \le \frac{2\sqrt{d-1}}{d}.$$
    In that case, for all $d\ge 15$ we have $\|Q_{\perp}\|=\rho<\frac{1}{2}$, for $N$ large enough and  owing to \eqref{eq:boundkappa} we have $\kappa(Q) < 1$ and therefore the upper bound for $|\lambda_2(A_{\sss})|$ given by \cref{bigtheorem} is ${\mathrm{N}}_{n,\sss} \ge 1$. An explicit example of such graphs are the Ramanujan graphs $X^{p,q}$ of degree $p+1$ and cardinality given by $\frac{q(q^2-1)}{2}$ where $(p,q)$ are prime numbers. Their construction is given in  \cite[Chapter 4]{davidoff2003elementary}. 
    \section{Overview of the proofs}
   
	The proof of \cref{lem:existandconstructofqu} is independent of the rest of this paper. We give an explicit constructive definition of the free with amalgamation operators $(q,u)\in (\mathcal A, \mathcal B, \mathrm E, ,\Phi, \ast)$ in \cref{intermediateconstruct}. This construction is inspired by the existing construction given by Hermon, Sly and Sousi \cite{hermon2020universality} (see \cref{quasi_tree_1}, \cref{quasi_tree_2} and \cref{quasi_tree}). This construction allows us to compute the moments of the limit operators given by \eqref{eq:defaandp} (see \cref{formuletrace}) and therefore the approximation of radius of the limit operators. \\
	In \cref{section:pathdecom} we start by using that for any matrix $A \in \mathrm M_N(\CC)$ and integer $\ell \ge 1$ we have 
	$$\rho(A) \le \| A^{\ell}\|^{1/\ell}$$
	to obtain the bound \eqref{eq:bound_gelfand} and we then state the main technical statement of this paper given by \cref{big2theorem}. In this statement we consider any power $\ell \ge 1$ and denoting $\mathbb P := \mathbb{P}\left(|\lambda_{2}(A_{\mathrm x})| \ge(1+ \epsilon) \rho_{n,\xxx}\right) $ we rewrite the probability we want to bound as follows
	\begin{equation*}
		\mathbb P\le \mathbb{P}\left(|\lambda_{2}(A_{\mathrm x})| \ge(1+ \epsilon) \rho_{n,\xxx}, (M,Q) \text{ is $\ell$-tangle free }\right) + \mathbb{P}\left( (M,Q) \text{ is not $\ell$-tangle free }\right).
	\end{equation*}
	The notion of a $(M,Q)$ being $\ell$-\textit{tangle-free} is introduced by \cref{deficoincid}. The couple $(M,Q)$ being $\ell$-tangle free tells us that looking at the powers of $A_{\mathrm x}$, the second probability in the sum above tends to zero with speed $N^{-c/2}$ for some $0<c<1$ provided that we have 
	$$\ell < \frac{1-c}{3 \ln(d)} \ln(N).$$ 
	The previous statement is the object of \cref{section:comput_rd_permut} and \cref{pr1}. Once we showed that our couple $(M,Q)$ is $\ell$-tangle free with high probability, we rewrite the entries of $A^{\ell}_{\mathrm x}$ on this event (see \eqref{6}).
	 Finally, we want to apply the property of random permutation given by \cite[Proposition 11]{bordenave2015new}, recalled in \cref{estimconsistent}. Owing to this proposition, we can define a new matrix $\underline{A}_{\mathrm x}^{(\ell)}$ where now all the appearance of $M$ entries are centered (see \eqref{eq:defPveclubarr}). Indeed, introducing $\underline{A}_{\mathrm x}^{(\ell)}$, \textit{rest} terms appear (see \cref{MQltf}), and 
	 the normalized sum of these terms can be neglected.
	 
In regard of \cref{MQltf} and assuming the operator norm of the rest terms are suitably controlled, it remains to upper bound the operator norm of $\underline{A}^{(\ell)}_{\mathrm x}$. This upper bound is given by \cref{boundingnormP} and its proof, constituting the main technical challenge, is the object of \cref{section:high_tr_meth}. In that section, the goal is to show that for an explicit power $\ell \sim \ln(N) $ we have \begin{equation*}
    \|\underline{A}_{\mathrm x}^{(\ell)}\| \le \left( 1 + o(1) \right)^{\ell} \rho_{n,\mathrm x}^\ell
\end{equation*}
with high probability and where $\rho_{n,\mathrm x}$ is introduced in \eqref{eq:rho(ell_0)}. We will apply Markov inequality, i.e
\begin{align*}
    \mathbb{P} \left(     \|\underline{A}_{\mathrm x}^{(\ell)}\| \le \left( 1 + o(1) \right)^{\ell} \rho_{n,\mathrm x}^\ell \right) =     \mathbb{P} \left(     \|\underline{A}_{\mathrm x}^{(\ell)} {\underline{A}_{\mathrm x}^{(\ell)}}^{\top}\|^m \le \left( 1 + o(1) \right)^{2\ell m} \rho_{n,\mathrm x}^{2m\ell} \right) \le \frac{ \mathbb{E} \left(\|\underline{A}_{\mathrm x}^{(\ell)} {\underline{A}_{\mathrm x}^{(\ell)}}^{\top}\|^m \right) }{ \left( 1 + o(1) \right)^{2\ell m} \rho_{n,\mathrm x}^{2m\ell}}
\end{align*}
and bound the expectation in the left hand side of the inequality above using trace method (see the computation in \cref{subsectionboundofP}). The fact that we restrict to the event where $(M,Q)$ is $\ell$-tangle free allows us to rewrite the trace $\Tr\left\{\left(\underline{A}^{(\ell)}{\underline{A}^{(\ell)}}^{\top}\right)^m\right\}$ as a sum over all concatenations of $2m$ $\ell$-tangle free paths with specific bound properties (see \eqref{eq:condilim}). We then define equivalent classes on the set of such $\ell$-tangle free paths and obtain the upper bound given by \eqref{eq:htmethod}. We bound separately the number of equivalent classes and, for a given equivalent class, the contribution of the permutation matrix $M$ given by $\mu(\gamma)$ and the contribution of $Q$. \cref{section:encodingeqclass} and \cref{section:contribQ} are dedicated to the proof of the bounds on respectively the cardinal of equivalent classes given by \cref{boundingWlTsap} and the contribution of the matrix $Q$ given by \cref{boundingsumq(gamma)}. The contribution of the permutation given by \cref{boundingmu(gamma)} is handled applying \cref{estimconsistent}. The strategy to encode and enumerate the number of equivalent class is analog to the various other high trace method proofs (see for instance \cite{bordenave2015new,bord2018spectral,CBeigen,Brito_Dumitriu_Harris_2022}). We also conclude the proof of \cref{boundingnormP} in \cref{section:contribQ}. \\

Finally, in \cref{section:op_norm_rest} we bound the rest terms appearing in \cref{MQltf} using the same strategy used to bound the operator norm of $\underline{A}_{\mathrm x}^{(\ell)}$. Indeed bounding the operator norms with traces we obtain \eqref{defSR_kl}, analog of \eqref{eq:htmethod} with a new definition for equivalent classes for paths. \cref{boundingnormR} and \cref{boundWlTR_k} are the analog of \cref{boundingnormP} and \cref{boundingWlTsap}. The normalized sum in \cref{MQltf} makes the upper bound given in \cref{boundingnormR} less challenging (see \eqref{eq:normalization}). We gather all these results in \cref{section:proofbigth} to conclude the proof of \cref{big2theorem}.
	\section{Intermediate construction}\label{intermediateconstruct}
	
	We give here a constructive definition of the operator valued probability space $(\mathcal{A},\mathcal{B},\mathrm{E},\Phi)$ and the operators $(q,u) \in \mathcal{A}$ which verifies the first statement of \cref{bigtheorem}, i.e.
	\begin{itemize}
		\item $q$ has, asymptotically, same $\ast$-distribution as $Q$,
		\item The operator $u$ is unitary (i.e. $uu^\ast = u^\ast u = 1$) and $q$ and $u$ are free with amalgamation over $\mathcal{B}$.
	\end{itemize}
	In a second time, we study the $\ast$-distribution of $\mathrm{b}_{\sss}=u+q$ and $\mathrm{b}_{\ccc} = (1-r)u +rq$ and give an approximation of the spectral radius of $\mathrm{b}_{\xxx}$ for $\xxx \in \{ \sss, \ccc\}$.
	\subsection{Construction of $q$ and $u$ as adjacency operators of graphs}\label{graphpov}
	The construction given here is inspired by the quasi-tree construction described by Hermon, Sly and Sousi \cite{hermon2020universality}. We will consider $q$ and $u$ as adjacency operators of explicit graphs, meaning that we will give explicit writings of these operators allowing us to then compute $a$ and $p$ moments. One has to consider the construction of these graphs as follows. First, we take infinite and countable copies of the vertex set $V=[N]$ on which the adjacency operator $Q$ is defined. Then we consider each copy as a vertex and we put a tree structure on the set of copies. The tree is regular and the degree depends on the dimension $N$ of  the operator $Q$. More precisely, we consider one of the copies as the origin (denoted by $H_{\emptyset}$ in the \cref{quasi_tree_1}). We consider that the origin has $N$ descendants which are $N$ other copies of the vertex set $[N]$. To each vertex $x$ of the origin we assign uniformly at random one of the $N$ descendant copies. In that descendant, we take uniformly at random one of its vertex $x^\prime$ to receive the ascending edge. The edge $(x,x^\prime)$ is represented in blue in \cref{quasi_tree_1} and the adjacency operator that will correspond to the graph defined by the blue edges (after constructing them all) is $u$.  We repeat this operation for each descendant, taking at random the vertex set copies among the copies never seen before, and so on. The following \cref{quasi_tree_1} shows an example for $N=3$.
	\begin{figure}[H]
		\centering
		\includegraphics[scale=0.4]{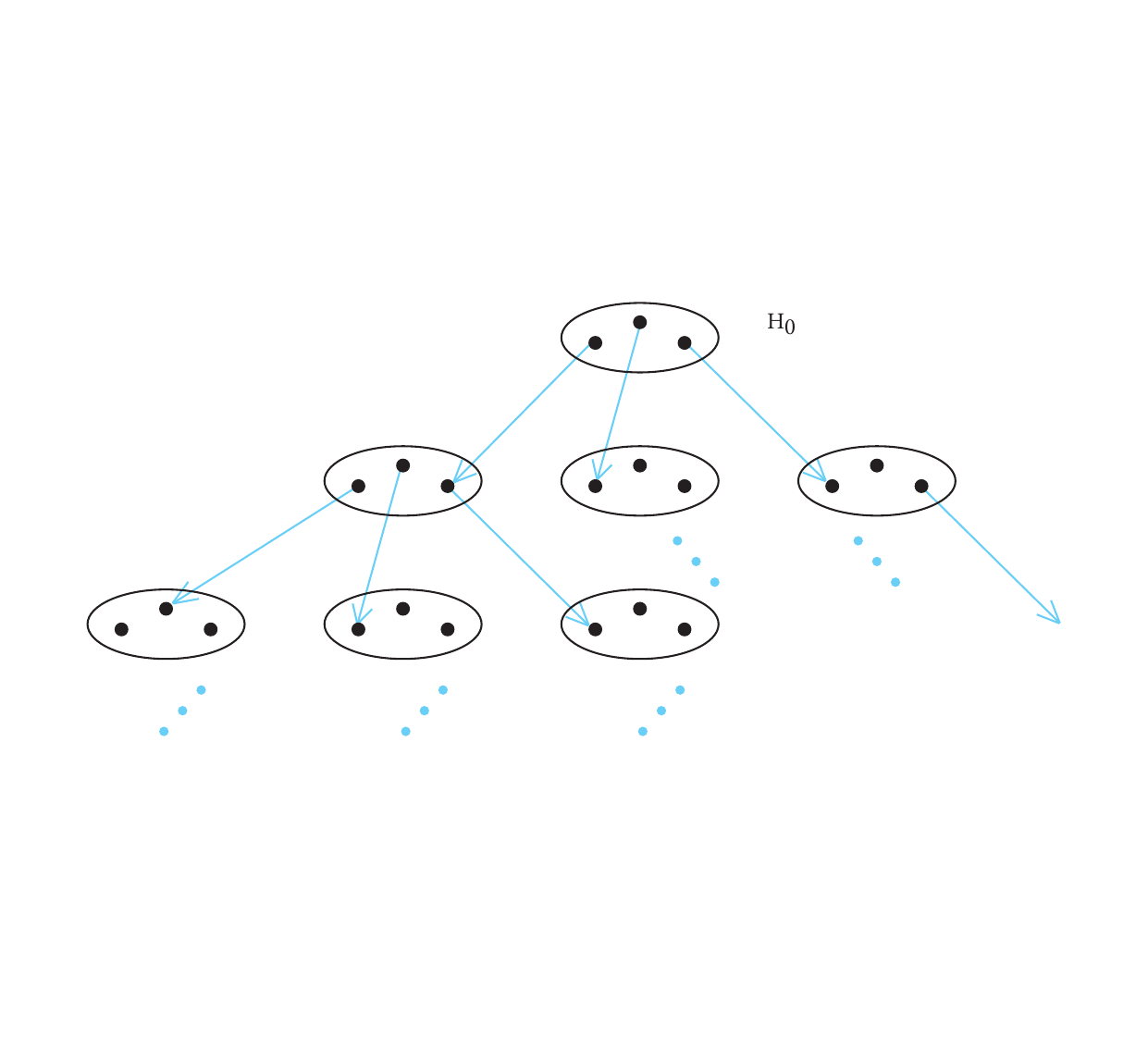}
		\caption{Quasi-tree, first step where $N=3$. In blue the first step of construction of the operator $u$.}
		\label{quasi_tree_1}
	\end{figure}
	We give simultaneously the construction $u^{\ast}$, inverse of the unit operator $u$. To each blue edge we connect an inverse red edge, which connects it to its antecedent. We also construct the symmetric graph above the origin, where now the red edges corresponding to $u^{\ast}$ are drawn first and can be seen as descendants, and the blue ones corresponding to $u$ are their inverse (see \cref{quasi_tree_2}).
	\begin{figure}[H]
		\centering
		\includegraphics[scale=0.4]{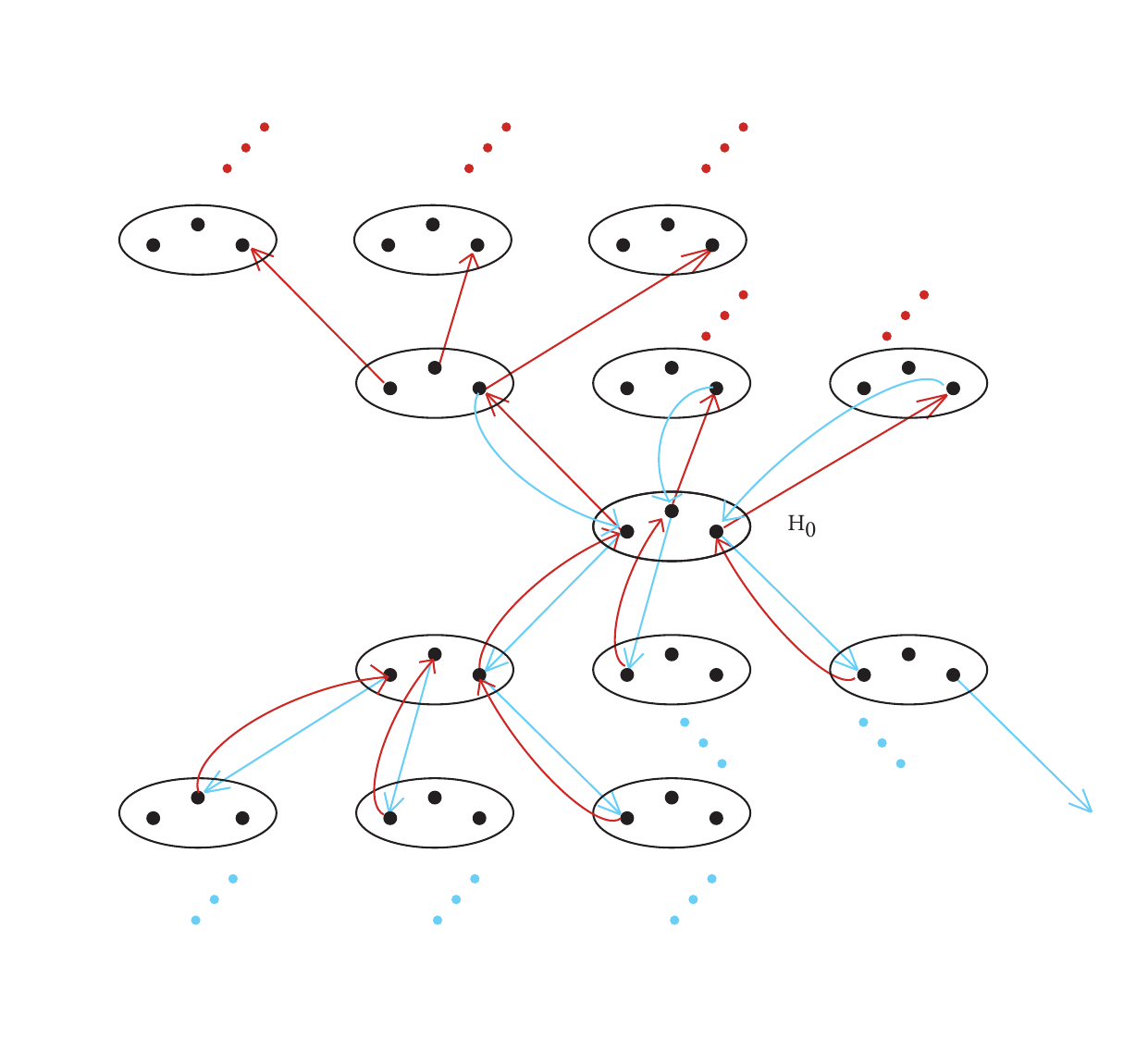}
		\caption{Quasi-tree, second step. Construction of $u^\ast$ in red as the inverse of $u$ and symmetrizing the construction inverting $u$ and $u^\ast$ roles.}
		\label{quasi_tree_2}
	\end{figure}
	Finally it remains to place on each set of vertices the edges corresponding to the graph of $Q$ (in \cref{quasi_tree} we considered $Q= (1,1,1)^{\top}\cdot(1,1,1)$).
	\begin{figure}[H]
		\centering
		\includegraphics[scale=0.45]{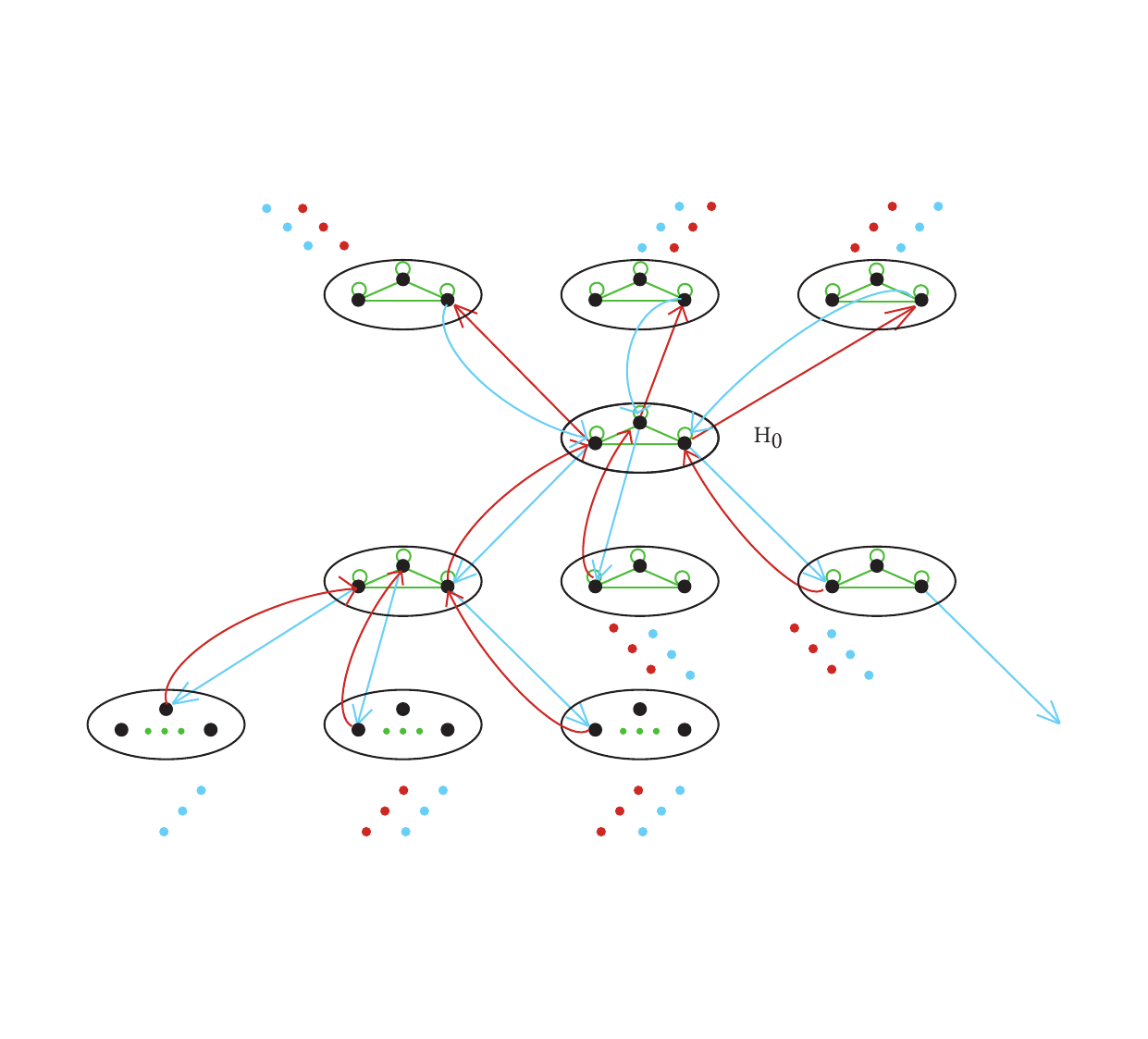}
		\caption{Quasi-tree final step. Adding on each vertex-set of the quasi-tree structure the edges defined by the operator $Q$.}
		\label{quasi_tree}
	\end{figure}
	The operators $q$ and $u$ we are looking for are the adjacency operators of the graphs described above. For instance, the operator $a = u+q$ is the adjacency operator of the graph where we considered the blue and green edges of \cref{quasi_tree}.
	\subsection{Construction of the underlying Hilbert space}
	To make our Hilbert space explicit, we first consider the following vertex set
	$$\mathcal{S} := \{(k_0,\epsilon_1,k_1,\epsilon_2,k_2,...,\epsilon_{n-1},k_{n-1},\epsilon_n) ~~; n\ge 1~~; \forall t\le n,~~ \epsilon_t \in \{+,-\} \text{ and } k_{t}\in[N] \}.$$
	Let $\{ H_{\emptyset} \}\cup \{ H_s \}_{s \in \mathcal{S}}$ be independent copies of $\mathbb{C}^{N}$. With respect to the previous \cref{graphpov}, $\mathcal{S}$ is the index set for the countable and infinite copies of $[N]$ considered as $Q$ vertex set. One can think of $n$ as the future number of generations from the origin and $\epsilon_t$ as either descending from generation $t$ to generation $t+1$ taking $u$ or $u^{\ast}$. For $s:=(k_0,\epsilon_1,k_1,...,k_{n-1},\epsilon_n)\in \mathcal{S}$, we set the Hilbert space
	$$\mathrm{H}_s := H_{\emptyset} \otimes H_{(k_0,\epsilon_1)} \otimes H_{(k_0,\epsilon_1,k_1,\epsilon_2)} \otimes \cdots \otimes H_{s},$$
	and we define
	\begin{equation*}
		\mathrm{H}_{0} := \bigoplus_{s\in \mathcal{S}} \mathrm{H}_{s}.
	\end{equation*}
	We will reduce our Hilbert space by considering only the basis vector $e_1 \otimes \cdots \otimes e_n \in \mathrm{H}_{s}$ such that for all $t\le n$, $e_t$ is the $k_t$-th canonical basis vector of $H_{(k_0,\epsilon_1,...,k_{t-1},\epsilon_t)}$. From now on for $(k_0,\epsilon_1,k_{1},...,\epsilon_{n-1},k_{n-1},\epsilon_n) \in \mathcal{S}$, for $t \le n$ we consider $i_t :=(\epsilon_t,k_t)$. For any $s=(k_0,i_1,...,i_n,\epsilon_n) \in \mathcal{S}$, we denote by $e(k_0) \otimes e(i_1) \otimes\cdots \otimes e(i_n) \in \mathrm{H}_s$ the tensor product where for all $t \le n$, $e(i_t)= e(\epsilon_t,k_t)$ is the $k_t$-th canonical basis vector of $H_{(k_0,i_1,...,i_{t-1},\epsilon_t)}$.
	Explicitly we consider $\mathrm{H}_1$ the subspace of $\mathrm{H}_{0}$ which is the closure of the span of
	\begin{equation*}
		\{ e(k_0)\otimes e(i_1)\otimes \cdots \otimes e(i_n) \in \mathrm{H}_{(k_0,i_1, ...,i_{n-1}, \epsilon_n)}~; ~ s=(k_0,i_1,...,i_n,\epsilon_n) \in \mathcal{S}\}.
	\end{equation*}
	We introduce another vertex set
	$$\mathcal{V}_1:=\{0\}\cup\{(k_0,i_1,i_2,...,i_n)~;~ n\in \mathbb{N}~; ~ \forall t \le n,~ i_t :=(\epsilon_t,k_t)\in \{+,-\} \times [N]~; \text{ and } k_0 \in [N] \}.$$
	One can consider that $\mathcal{S}$ is a partition of $\mathcal{V}_1$, i.e. that for all $v := (k_0, i_1,...,i_n) \in \mathcal{V}_1$, we have $v\in s :=(k_0,i_1,i_2,...,i_{n-1},\epsilon_n) \in \mathcal{S}$. For $v := (k_0, i_1,...,i_n) \in \mathcal{V}_1$, we also denote:
	\begin{align*}
		e(v)&:= e(k_0) \otimes e(i_1) \otimes \cdots \otimes e(i_n)\\
		e(0)&:=0\\
		\epsilon(v)&:= \epsilon_n\\
		v^- &:= (k_0,i_1,...,i_{n-1}) .
	\end{align*}
	In particular, we see that $\mathrm{H}_1$ is equal to the closure of ${\Span \{ e(v)~;~ v\in \mathcal{V}_1\}}$ and this space can be seen as all possible ways to connect any $n$-th generation vertex $v= (k_0,i_1,...,i_n)\in \mathcal{V}_1$ to any descendant one $v^\prime := (k_0,i_1,...,i_n,i_{n+1})$ (note that $(v^\prime)^- = v$). So for $e=e(k_0) \otimes e(i_1) \otimes \cdots \otimes e(i_n) \in \mathrm{H}_1$ we can also denote
	$$v_{e}:= (k_0,i_1,...,i_n).$$
	Finally for $v:=(k_0,i_1,...,i_n) \in \mathcal{V}_1$ (resp. for $s:=(k_0,i_1,...,i_{n-1},\epsilon_n) \in \mathcal{S}$), we define the \textbf{length} of $v$ (resp. of $s$) by the integer $n$. We then denote it by $|v|= |s|=n$. We now construct the random Hilbert space at stake here. First consider the families of random permutations on $[N]$
	\begin{equation*}\label{radomfloors}
		(j_{+}(s))_{s \in \mathcal{S}} \text{ and } (j_{-}(s))_{s \in \mathcal{S}} \text{ iid and } \sim \mathcal{U}(\mathrm{S}_N).
	\end{equation*}
	Recall that, with respect to the construction in \cref{graphpov}, the randomness comes from the fact that at each generation, for each vertex $v$ we randomly take a descendant copy of the vertex set $[N]$ and a vertex $v^\prime$ in it to which we associate $v$. It will verify $(v^\prime)^- =v$.
	Note that for $s := (k_0,i_1,...,i_{n-1},\epsilon_n) \in \mathcal{S}$ and $ k\in [N]$, we have that $v:=(s,k)=(k_0,i_1,...,\epsilon_n,k) \in \mathcal{V}_1$.  So, keeping the same notations, from now on we denote
	
	$$j_{\epsilon}(v):= j_{\epsilon}(s)(k)$$
	where $\epsilon \in \{+,-\}$.
	We may now define the random vertex set
	\begin{equation*}
		\mathcal{V} := \{v=(k_0, i_1, ...,  i_{n}) \in \mathcal{V}_1  ~~ ; ~~ \forall \ell < n : ~ k_\ell \neq j_{-\epsilon_\ell}(k_0,i_1,...,i_{\ell-1}) \text{ or } \epsilon_{\ell+1}= \epsilon_\ell\},
	\end{equation*}
	and the random Hilbert subspace of $\mathrm{H}_1$ obtained by taking the closure of the vector space spanned by
	\begin{equation}\label{basisH}
		\Xi:=\{e=e(k) \otimes e(i_1)\otimes \cdots \otimes e(i_n) ~ ~ ; ~ v_{e} \in \mathcal{V}\}.
	\end{equation}
	We denote by $\mathrm{H}$ the Hilbert space obtained. We now explicit the Banach and operator-valued probability space structure of $\mathrm{B}(\mathrm{H}) \subset \mathcal{B}( L^2(\mathrm{H}_0))$.
	
	\begin{defi}\label{banach-op-valued-struc}
		\begin{enumerate}
			\item The space $(\mathrm{H},(\cdot,\cdot)) \subset L^2(\mathrm{H}_{0})$ is a Hilbert space where we considered the scalar product defined by
			$$\forall h,h^\prime \in \mathrm{H},~~(h,h^\prime) := \mathbb{E}[\langle h,h^\prime \rangle] ~~ \text{and}~~ \|{h}\|_2 := \sqrt{(h,h)}.$$ 
			The subordinate norm $\|\cdot\|$ on $\mathcal{B}( \mathrm{H})$ gives to $\mathcal{B}( \mathrm{H})$ a Banach algebra structure.
			\item $(\mathrm{B}(\mathrm{H}),\Phi, \ast)$ is non-commutative probability space where the trace is given by
			\begin{equation*}
				\Phi(a):=\mathbb{E}\left[\langle a e_{\alpha} , e_{\alpha} \rangle\right] 
			\end{equation*}
			for $a\in \mathrm{B}(\mathrm{H})$, where $e_{\alpha} \in H_{\emptyset}$ and $\alpha \sim \mathcal{U}([N])$ is independent from $(j_{\epsilon}(s))_{\epsilon \in \{-,+\}, s \in \mathcal{S}}$.\\
			\item The conditional expectation $\mathrm{E}$ on $\mathrm{B}(\mathrm{H})$ is the operator verifying 
			\begin{equation*}
				\langle \mathrm{E}(a) e , f \rangle = \mathds{1}( e=f) \langle ae,f \rangle
			\end{equation*}
			for all $ a \in \mathrm{B}(\mathrm{H})$ and for all $e, f \in \Xi$.
		\end{enumerate}
	\end{defi}
	We now define the operators $(q,u)$ in \cref{bigtheorem}. 
	\begin{defi}\label{def:u_q}
		We define the set $\mathcal{E}_{+} \subset \mathcal{V}^2$ as the set of edges $(v,v^{(+)})\in \mathcal{V}^2$ verifying that for all $v =(k_0,...,i_n)$, we have  $v^{(+)}=(k_0,i_1,...,i_n,(+,j_{+}(v)))$ and $k_{n}\neq  j_{-}(k_0,i_1,...,i_{n-1})$. The operator $u$ is the adjacency operator of the directed graph $\mathcal{G}_{+} = (\mathcal{V},\mathcal{E}_{+})$. It can also be defined by an explicit formula on the basis vectors spanning $\mathrm{H}$. Let $e:=e(k_0) \otimes e(i_1)\otimes \cdots e(i_n) \in \mathrm{H}$, we have
		\begin{equation*}\label{defiu}
			u(e) = \left\{
			\begin{array}{ll}
				e\otimes e(+,j_{+}(x)) & \mbox{if } \left\{\begin{array}{ll}
					&\epsilon_n=+ \\
					&\mbox{or}\\
					& k_n \neq j_{-}(k_0,i_1,...,i_{n-1})
				\end{array}
				\right.\\
				e(k_0) \otimes e(i_1)\otimes \cdots \otimes e(i_{n-1}) & \mbox{otherwise.}
			\end{array}
			\right.
		\end{equation*}
		The operator $u^\ast$ is then the adjacency operator of $\mathcal{G}_{-} = (\mathcal{V},\mathcal{E}_{-})$ where, taking the same notations for $v$, we have that $(v,v^{(-)})\in \mathcal{E}_{-}$ if $v^{(-)}=(k_0,i_1,...,i_n,(-,j_{-}(v))$ and $k_{n}\neq  j_{-}(k_0,i_1,...,i_{n-1})$.
		Using the same notation for $e \in \mathrm{H}$, $u^{*}$ can be defined symmetrically by the formula
		\begin{equation*}\label{defiu*}
			u^{*}(e) = \left\{
			\begin{array}{ll}
				e\otimes e(-,j_{-}(x)) & \mbox{if } \left\{\begin{array}{ll}
					&\epsilon_{n}=- \\
					&\mbox{or}\\
					& k_n \neq j_{+}(k_0,i_1,...,i_{n-1})
				\end{array}
				\right.\\
				e(k_0) \otimes e(i_1)\otimes \cdots \otimes e(i_{n-1}) & \mbox{otherwise.}
			\end{array}
			\right.
		\end{equation*}
		Finally we consider the isometric injection
		\begin{align*}
			\Pi &: \mathrm{M}_N(\mathbb{C}) \longrightarrow \mathrm{B}(\mathrm{H}) \\
			B &\longmapsto \left[\begin{aligned} b  : \mathrm{H} &\rightarrow \mathrm{H}\\
				e(k_0)\otimes e(i_1)\otimes \cdots \otimes e(i_n) &\mapsto e(k_0)\otimes e(i_1)\otimes \cdots \otimes B e(i_n) 
			\end{aligned}\right]
		\end{align*}
		and we set 
        \begin{equation}\label{def:q}
        {q}:= \Pi(Q_{\perp}).
        \end{equation}
         We now consider the operator valued probability space $(\mathcal{A}, \mathcal{B},\mathrm{E},\Phi,\ast)$ where $\mathcal{A}$ is the subalgebra of $\mathrm{B}(\mathrm{H})$ spanned by $u$, $u^\ast$, $\mathcal{M} := \Pi(\mathrm{M}_N(\CC))$ and $\mathcal{B}:= \mathrm{E}(\mathcal{A})$. Finally we denote $\mathcal{U}$ the sub algebra of $\mathcal{A}$ spanned by $u$ and $u^\ast$.
	\end{defi}
	
	\begin{rem}
		We could have defined both $u$ and $u^{*}:=u^{-1}$ as follows
		\begin{equation*}
			u^{\epsilon}(e) = \left\{
			\begin{array}{ll}
				e\otimes e(\epsilon,j_{\epsilon}(x)) & \mbox{if } \left\{\begin{array}{ll}
					&\epsilon_n=\epsilon\\
					&\mbox{or}\\
					& k_n \neq j_{-\epsilon}(k_0,i_1,...,i_{n-1})
				\end{array}
				\right.\\
				e(k_0) \otimes e(i_1)\otimes \cdots \otimes e(i_{n-1}) & \mbox{otherwise.}
			\end{array}
			\right.
		\end{equation*}
	\end{rem}
	For all $B \in \mathrm{M}_N (\mathbb{C})$, we have
	\begin{align*}
		\Phi(b)&= \mathbb{E}[\langle b e_\alpha,e_\alpha \rangle]=\mathbb{E}[\langle B e_\alpha,e_\alpha \rangle] = \frac{1}{N}\Tr(B).
	\end{align*}
	Applying the above to $B=P(Q,Q^\ast)$ for any $P\in \CC\langle X, X^\ast\rangle $ we obtain that $Q$ and $q$ have the same $\ast$-distribution.
	
	\subsection{Freeness with amalgamation over $\mathcal{B}$ of $u$ and $\mathcal{M}$}
	
	The freeness with amalgamation of $q$ and $u$ is a direct consequence of the freeness with amalgamation of $u$ (i.e. the algebra spanned by $u$ and $u^{\ast}$) and the algebra $\mathcal{M}$. It boils down to checking that for all $n \ge 1$ and $b_1,...,b_n \in \mathcal{M}$ such that $\mathrm{E}(b_i)=0$, if $\alpha_1,...,\alpha_{n+1} \in \mathbb{Z}$ are such that for all $i \notin \{ 1, n+1\}$ we have $\alpha_i \neq 0$ then we have
	
	$$\mathrm{E}\left[ u^{\alpha_1}b_1 \cdots u^{\alpha_n}b_n u^{\alpha_{n+1}}\right]=0.$$
	It is equivalent to check that for all $e \in \Xi$ (see \eqref{basisH}) we have
	\begin{equation}\label{freenessQu}
		\langle u^{\alpha_1}b_1 \cdots u^{\alpha_n}b_n u^{\alpha_{n+1}}e,e\rangle=0.
	\end{equation}
	This is the main result of the following lemma.
	\begin{lem}\label{chemin}
		The following assertions hold.
		\begin{enumerate}
			\item \label{point1}For all $v= (k,i_1,...,i_p)\in \mathcal{V}$ we have
			\begin{equation*}
				e(v)= u^{\alpha_{m}} b_{m} u^{\alpha_{m-1}}\cdots u^{\alpha_1} b_1 u^{\alpha_{0}} e(k)
			\end{equation*}
			for some $\alpha_0,...,\alpha_{m} \in \mathbb{Z}$ and $b_1=\Pi({B}_1),...,b_m= \Pi({B}_m) \in \mathcal{M}$ where $B_i$ are elementary matrices with zeros on the diagonal.
			\item\label{point2}
			Conversely, for all $k\in [N]$, for all $b_1=\Pi({B}_1),...,b_m= \Pi({B}_m) \in \mathcal{B}_{1}$ where $B_i$ are elementary matrices with zeros on the diagonal and for all $\alpha_0,...,\alpha_m \in \mathbb{Z}$ such that $\alpha_i \neq 0$ for $i \notin\{ 0,m\}$, there exists $v \in \mathcal{V}$ such that
			\begin{equation}
				u^{\alpha_{m}} b_{m} u^{\alpha_{m-1}}\cdots u^{\alpha_1} b_1 u^{\alpha_{0}} e(k) = \# e(v) \text{ and } ~~~ |v| = \sum_{i=1}^{m} |\alpha_i|
			\end{equation}
			with $\# \in \{0,1\}$. The vertex $v$ verifies $v= (v^{-}, (\epsilon(\alpha_m),j_{\epsilon(\alpha_m)}(v^{-})))$ for some $v^{-} \in \mathcal{V}$.
		\end{enumerate}
		In particular, $u$ and $\mathcal{M}$ are free with amalgamation over $\mathcal{B}$.
	\end{lem}
	
	Point \ref{point2} of the previous lemma says that the length of $s \in \mathcal{S}$ is the distance between $s$ and the root $\{\emptyset\}$. Besides one can notice that the operator $b= \Pi(B)$, with $B$ an elementary matrix with zeros on the diagonal is the adjacency operator of the graph with only one edge in each copie of $[N]$.\\
	To show specifically the second part of the previous lemma we will use the following. 
	\begin{lem}\label{itermediatelem}
		Let $N\in \mathbb{N}$, $v:=(k,i_1,...,i_p) \in \mathcal{V}$ and $\alpha \in \mathbb{Z}$. We set $\epsilon :=\sign(\alpha)$ and $i_p := (\epsilon_p,k_p)$. We assume that either $p=0$ (i.e. $v=(k)$), or $\epsilon_p = \epsilon$ or $k_p \neq j_{-\epsilon}(k,i_1,...,i_{p-1})$. We set the notations
		\begin{align*}
			v_{\epsilon}& := (v,(\epsilon, j_{\epsilon}(v))) \in \mathcal{V}\\
			\text{ for }  1 \le t \le |\alpha|,~ v_{\epsilon (t+1)} &:= (v_{\epsilon t},(\epsilon, j_{\epsilon}(v_{\epsilon t}))) \in \mathcal{V}.
		\end{align*}
		Then $|v_{\alpha}| = |v| + |\alpha|$ and verifies $u^{\alpha} e(v)= e(v_{\alpha})$.
	\end{lem}
	
	\begin{proof}
		We denote $e=e(v)$ and proceed by induction on $|\alpha|$. If $|\alpha|=1$ we have according to the definition of $u^{\epsilon}$ that $u^{\epsilon}(e) = e\otimes e(\epsilon, j_{\epsilon}(v))= e(v_{\epsilon})$. Assuming now $|\alpha| \ge 2$ we have
		$$u^{\alpha}(e)= u^{\alpha -\epsilon} (u^{\epsilon}(e)).$$
		We denote $\tilde{\alpha}= \alpha-\epsilon$. Since $u^{\epsilon}(e)=e \otimes e(\epsilon, j_{\epsilon}(v))$, $\sign(\alpha-\epsilon) = \epsilon$ and $|\tilde{\alpha}|= |\alpha|-1 < |\alpha|$, one can apply the induction hypothesis on $\tilde{\alpha}$ and conclude.
	\end{proof}
	We can now prove \cref{chemin}.
	\begin{proof}[Proof of \cref{chemin}]
		For the first point we proceed by induction over $|v| \ge 1$. Let $v:= (k,i_{1})= (k,(\epsilon,k_1))$. Either $k_1 = j_{\epsilon}(k)$, then $e(v)=e(k)\otimes e(\epsilon,j_{\epsilon}(k))= u^{\epsilon} (e(k))$, or $j_{\epsilon}(k) \neq k_1$ and then denoting $B := E_{j_{\epsilon}(k),k_1}$ and $b = \Pi(B)$ we have that $b u^{\epsilon} (e)= b (e(k) \otimes e(j_{\epsilon}(k),\epsilon)) = e(k) \otimes B e( \epsilon,j_{\epsilon}(k)) = e(v)$. We now consider $v= (k,i_1,...,i_{p+1})$ for some $p \ge 1$. For $t < p+1$ we consider $v_{t} := (k,i_1,...,i_t)$ and $e_{t}:= e(v_t)$. Since $e=e(v) \in \Xi$, we have $v_t \in \mathcal{V}$ for all $t < p+1$. Let $\alpha_{0},...,\alpha_{m} \in \mathbb{Z}$ and $b_1,...,b_m$ as described in \cref{chemin} and such that
		$$e_p  = u^{\alpha_{m}} b_{m} u^{\alpha_{m-1}}\cdots u^{\alpha_1} b_1 u^{\alpha_{0}} e(k).$$
		The condition $v \in \mathcal{V}$ imposes that either $k_p \neq j_{-\epsilon_p}(v_{p-1})$ or $\epsilon_p = \epsilon_{p+1}$. Therefore we have
		$$u^{\epsilon_{p+1}} (e_p)=e_p \otimes e(\epsilon_{p+1},j_{\epsilon_{p+1}}(v_p)).$$
		We replace $\alpha_{m}$ by $\alpha_{m} + \epsilon_{p+1}$. If $k_{p+1} = j_{\epsilon_{p+1}} (v_p)$ we have the result. If $k_{p+1} \neq j_{\epsilon_{p+1}} (v_p)$ then we consider $B_{m+1} := E_{k_{p+1},j_{\epsilon_{p+1}}(v_p)}$ and $b_{m+1} := \Pi(B_{m+1})$ and we have
		$$b_{m+1} u^{\alpha_{m} + \epsilon_{p+1}} b_m \cdots u^{\alpha_1}b_1u^{\alpha_0} e(k) = e(v).$$
		We now prove the second assertion and we start with the case where we impose also $\alpha_m \neq 0$. We fix $k \in [N]$ and proceed by induction, this time on $p:= \sum_{i=0}^m |\alpha_i|$. For $p=1$ we necessarily have that $m=1$ and $\alpha_1=\alpha \in \{+1,-1\}$. For $B= E_{ij}$ an elementary matrix, we have either $j\neq k$ and then $u^{\alpha}b e(k)= be(k)=0$, or $j=k$ and then $u^{\alpha}be(k) = u^{\alpha}e(i)= e(i) \otimes e(\alpha, j_{\alpha}(i))$. We now consider $p\ge 2$. If $m=1$ and $|\alpha_1|=m$ then we conclude with \cref{itermediatelem}. Otherwise we set
		$$f:= u^{\alpha_{m-1}} b_{m-1} \cdots u^{\alpha_1} b_1 u^{\alpha_0} e(k)$$
		and $\epsilon_i:= \sign(\alpha_i)$.
		Since $\alpha_m \neq 0$ then by induction we have $v \in \mathcal{V}$ such that we have $f= e(v) = e(v^{-})\otimes e(\epsilon_{m-1},j_{\epsilon_{m-1}}(v^{-}))$. We consider $E_{ij}$ the elementary matrix such that $b_m=\Pi(E_{ij})$. If $j_{\epsilon_{m-1}}(v^{-}) \neq j$ then $b_{m}f = 0$ and therefore $e=0$. Otherwise we have that $b_m f = e(v^{-})\otimes e(\epsilon_{m-1},i)$ with $i \neq j_{\epsilon_{m-1}}(v^{-})$. Applying \cref{itermediatelem} to $e=b_m f$ and $\alpha =\alpha_m$ we obtain the expected result. Now if $\alpha_m = 0$, knowing that $\alpha_{m-1} \neq0 $ we find $v \in \mathcal{V}$ of length $p = \sum |\alpha_i|$ such that we have $u^{\alpha_{m-1}}b_{m-1} \cdots u^{\alpha_0} e(k) = \epsilon e(v)$. Applying $b_m$ to $e(v)$ we either find $0$ or change the last coordinate $i_p$ of $v$.\\
		We now prove the freeness with amalgamation of $u$ and $\mathcal{M}$, i.e. that \eqref{freenessQu} holds true for all $e \in \Xi$ and for all $b_i=\Pi(B_i)$ where $B_i \in \mathrm{M}_N(\mathbb{C})$ and $\mathrm{E}(B_i)=0$. Owing to the first assertion of \cref{chemin}, we only have to check \eqref{freenessQu} holds for $e\in H_{\emptyset}$ basis vector. Finally by linearity in $b_i$ for all $i \le n$ it is sufficient to consider that $B_i$ is an elementary matrix $E_{k,\ell}$ of $\mathrm{M}_N(\mathbb{C})$ with zeros on the diagonal. Now let $k\in [N]$, $m\ge 2$, $\alpha_0,...,\alpha_m \in \mathbb{Z}$ such that $\alpha_i \neq0$ for all $i\notin \{ 0,m\} $ and $b_1,...,b_m$ such that $b_i = \Pi(B_i)$ where $B_i$ elementary matrix. Owing to the second assertion in \cref{chemin}, we have $f=u^{\alpha_m}b_m\cdots u^{\alpha_1}b_1 u^{\alpha_0} e(k)= \epsilon e(v)$ for some $v \in \mathcal{V}$ such that $|v| \ge 1$. In particular $f$ is indeed orthogonal to $e$. 
	\end{proof}
    All the ingredients are now gathered to prove \cref{lem:existandconstructofqu}.
	\begin{proof}[Proof of \cref{lem:existandconstructofqu}]
    From \cref{banach-op-valued-struc} and \eqref{def:q} from \cref{def:u_q}, for fixed non-all zero integers $\ell_0,...,\ell_{p}$ we have
    \begin{align*}
        \Tr_N(Q^{\ell_0} {Q^\ast}^{\ell_1} \cdots Q^{\ell_{p-1}}{Q^{\ast}}^{\ell_p}) - \Phi({q}^{\ell_0} {{q}^\ast}^{\ell_1}\cdots {q}^{\ell_{p-1}}{{q}^\ast}^{\ell_p}) = \frac{\delta^{\sum_t \ell_t}}{N} \underset{N \rightarrow \infty}{\rightarrow } 0.
    \end{align*}
    Besides, as mentioned above the freeness with amalgamation is a direct consequence of \eqref{freenessQu}.
    \end{proof}
	\subsection{Spectral radius of the limit operator}
	
	For $n \in \mathbb{N}$ and $T \in [n] $, we define the set
	\begin{equation*}\label{eq:defL(T,n)}
		\mathcal{L}(T,n):=\left\{\vec{n}=(n_1,...,n_T) \in \{0,...,n\}^{T}:~\sum_{t=1}^{T} n_t + T - 1 = n\right\}.
	\end{equation*} 
	We first express $\Phi[\mathrm{b}_{\xxx}^{n_1}(\mathrm{b}_{\xxx}^{*})^{n_2}]$ for any integer $n_1,n_2 \ge 0$ and $\mathrm{b}_{\xxx}$ defined by \eqref{eq:defaandp}.
	\begin{prop}\label{formuletrace}
		For all $n_1, n_2 \in \mathbb{N}$ we have
		\begin{equation*}
			\Phi[\mathrm{b}_{\sss}^{n_1}(\mathrm{b}_{\sss}^{*})^{n_2}]=\sum_{T=1}^{n_1 \land n_2}\sum_{\substack{\vec{n}\in \mathcal{L}(T,n_1)\\ ~\vec{n}^{\prime} \in \mathcal{L}(T,n_2)~}} \prod_{t=1}^{T} \Tr_{N}\left( Q^{n_t} (Q^{*})^{n_t^{\prime}}\right)
		\end{equation*}
		and
		\begin{equation*}
			\Phi[\mathrm{b}_{\ccc}^{n_1}(\mathrm{b}_{\ccc}^{*})^{n_2}]=\sum_{T=1}^{n_1 \land n_2}\sum_{\substack{\vec{n}\in \mathcal{L}(T,n_1)\\ ~\vec{n}^{\prime} \in \mathcal{L}(T,n_2)~}} (1-r)^{2(T-1)} \prod_{t=1}^{T} \Tr_{N}\left( \{rQ\}^{n_t} \{rQ^{*}\}^{n_t^{\prime}}\right).
		\end{equation*}
	\end{prop}
	For $t \ge 1$, we consider the sigma algebra
	$$\mathcal{F}_t := \sigma \left(\{\alpha\},\{ j_{\epsilon}(s)~; s\in \mathcal{S},~ |s|\le t-1,~ \epsilon \in \{+,-\} \}\right).$$
	Then \cref{formuletrace} is a direct consequence of the following lemma. 
	\begin{lem}
		Let $T \ge 1$ be an integer and $b_1,c_1,...,b_T, c_T \in \mathcal{M}$ be images of $N\times N$ complex matrices by $\Pi$. 
		\begin{enumerate}
			\item \label{item1elem} If for all $t \le T$, $b_i$ and $c_i$ are images of elementary matrices, we set
			$$e_t := b_T u b_{T-1} \cdots b_2 u b_1 e(\alpha)$$
			and we have that $e_t$ is a random vector such that there exists a unique $w_t \in \mathcal{V}$ verifying either  $|w_t|=t$ and  $\epsilon(w_t)=+$ or $w_t=0$. Also $w_t$ is $\mathcal{F}_t$-measurable.
			\item \label{item2genemat} We have
			$$\mathbb{E} [\langle b_1 u^{*} b_2 \cdots u^* b_{T-1}u^* b_T c_T u \cdots u c_1 e(\alpha), e(\alpha) \rangle] = \prod_{t=1}^{T} \Tr_N (b_t c_t) $$
			where we identify in notations the matrices $b_t$ and $c_t$ to theirs images through $\Pi$.
		\end{enumerate}
	\end{lem}
	
	\begin{proof}
		We proceed by induction on $T \ge 2$. For $T=2$ we have $be(\alpha) = e(i) \mathds{1}(\{\alpha = j\})$ for $b= E_{ij}$, and $ube(\alpha)=e(w)$ with $w := \mathds{1}(\{j=\alpha\})(i,+,j_{+}(i))$ and $b^\prime ub e(\alpha)= e(w^\prime)$ with $w^\prime:= (i, +, k)\mathds{1}(\{j=\alpha,l= j_+(i,+,j_{+} (w))\})$ for $b^\prime = E_{k,l}$. Let now $T \ge 3$. By induction we have 
		$$e_T = b_T u e(w_{T-1}).$$
		If $w_{T-1}= 0$, then we set $w_T=0$. Otherwise we have $\epsilon(w_{T-1})= +$, Therefore $ue(w_{T-1})= e(w_{T-1})\otimes e(+, j_{+}(w_{T-1}))$ and $b_T u e(w_{T-1})= e(w_T)$ where
		$$w_T := (w_{T-1},+,i) \mathds{1}_{\{j=j_{+} (w_{T-1})\}} \in \mathcal{F}_{T}$$
		and $b_T = E_{ij}$. \\
		We also proceed by induction on $T\ge 1$ for the second part of the lemma. Since each side of the formula are linear in the variables $b_t$ and $c_t$ for every $t \le T$, we only have to check the formula for elementary matrices. The formula is trivial for $T=1$. We consider now $T \ge 2$ and for $t\le T$, we consider the reel random variable
		$$I_t := \langle b_1 u^{*} b_2 \cdots b_{t-1} u^*b_tc_t u c_{t-1}u \cdots u c_1 e(\alpha), e(\alpha)\rangle$$
		and the vectors
		\begin{align*}
			e_t &:= c_t u c_{t-1} u \cdots u c_1 e(\alpha)\\
			e^\prime_t &:=b_t u b_{t-1} u \cdots u b_1 e(\alpha).
		\end{align*}
		According to the first assertion of the lemma, we find a unique $(w_{T-1},w^\prime_{T-1})\in \mathcal{V}^2$ of length $T-1$ such that $e_{T-1}=e(w_{T-1})$ and $e^\prime_{T-1}=e(w_{T-1}^\prime)$. Then rewriting $I_T$ and supposing that $w_{T-1}$ and $w_{T-1}^\prime$ are non-zero and therefore $\epsilon(w_{T-1})= \epsilon(w_{T-1}^\prime)=+$ we obtain
		\begin{align*}
			I_T &= \langle u^{*}b_T c_T u e(w_{T-1}),e(w_{T-1}^\prime) \rangle\\
			&= \langle b_T c_T  e(w_{T-1})\otimes e(+,j_{+}(w_{T-1})) ,e(w_{T-1}^\prime)\otimes e(j_{+}(w_{T-1}^{\prime}))\rangle\\
			&= (b_T c_T)_{j_{+}(w_{T-1})} \mathds{1}\left(w_{T-1}=w_{T-1}^\prime\right)
		\end{align*}
		where for $M\in \mathrm{M}_N(\mathbb{C})$ we use the notation $(M)_k := (M)_{k,k}$. Noticing that for all $w \in \mathcal{V}$ such that $|w|=T-1$ we have that $j_{+}(w)$ is independent from $\mathcal{F}_{T-1}$ we have
		$$\mathbb{E}(I_T | \mathcal{F}_{T-1})= \sum_{w\in \mathcal{V} ~; |w| =T-1} \mathbb{E}((b_{T}c_T)_{j_{+}(w)}) \mathds{1}(w_{T-1}=w_{T-1}^\prime = w)= \Tr_N (b_T c_T) \langle e(w_{T-1}),e(w_{T-1}^{\prime})\rangle.$$
		We conclude taking the expectation and applying the induction hypothesis to $\langle e(w_{T-1}),e(w_{T-1}^{\prime})\rangle$.
	\end{proof}
	
	Finally we have the following lower bound on the norm of the operator.
	\begin{lem}\label{lem:normeplusque1}
		For all $B \in \mathrm{M}_N(\mathbb{C})$ and $n \ge 1$ integer we have
		$$\norme{[u + \Pi(B)]^{n}} \ge 1.$$
	\end{lem}
	
	\begin{proof}
		Let $e \in \mathrm{H}$ such that $\norme{e}_2=1$. We set $b:=\Pi(B)$. We have
		
		$$[u + b]^{n}(e) = u^{n}(e) + \sum_{T=0}^{n} \sum_{ \sum_{t=1}^T \lambda_t = n - T +1} b^{\lambda_1}ub^{\lambda_{2}}u \cdots u b^{\lambda_{T}}$$
		where in the term $b^{\lambda_1}ub^{\lambda_{2}}u \cdots u b^{\lambda_{T}}$, the number of occurrences of $u$ is bounded by $n -1$, therefore the sum is orthogonal to $u^{n}(e)$. It remains to take the norm of in the equality above.
	\end{proof}
		\section{Path decomposition and computation on random permutation}\label{section:pathdecom}
	\subsection{Main technical statement}
	
	Having proved the first statement of \cref{bigtheorem} about the existence of the operator $(q,u) \in \mathrm{B}(\mathrm{H})$, the second statement is proved by applying the following theorem. First note that for any $ \xxx \in \{\sss, \ccc\}$ and $\ell \ge 0$ integer we have
	\begin{equation}\label{eq:bound_gelfand}
	\abs{\lambda_2(A_{\xxx})}=\{|\lambda_2(A_{\xxx}^{\ell})|\}^{1/\ell} \le \|A_{\xxx}|_{\boldsymbol{1}^\perp}^{\ell}\|^{1/\ell}
	\end{equation}
	where we used the fact that $\boldsymbol{1}$ is eigenvector for $A_{\xxx}$ and $A_{\xxx}^\ast$ and therefore $\lambda_{2}(A_{\xxx})= \lambda_1( A_{\xxx}|_{\boldsymbol{1}^{\perp}})$. To simplify further computations, for any integer $\ell \ge 0$ we introduce
	\begin{equation*}
		A_{\ell,\xxx}:=M A_{\xxx}^\ell.
	\end{equation*}
	Since $M$ is a permutation and $Q$ verifies \ref{H:1}, $A_{\ell,\xxx}$ leaves $\boldsymbol{1}$ and $\boldsymbol{1}^{\perp}$ stable and the multiplication by $M$ above does not change the operator norm. Furthermore for $N$ large enough we have
	$$1+ \epsilon \ge \exp(\frac{\epsilon}{2}),$$
	Where $\epsilon = \epsilon(N)$ is given by \cref{bigtheorem}. Therefore for $N$ large enough
	\begin{align*}
		\mathbb{P} \left( \|A_{\xxx}^\ell|_{\boldsymbol{1}^\perp}\| \ge (1+\epsilon)^\ell \rho_{n,\xxx}^\ell\right) = \mathbb{P}\left( \|A_{\ell,\xxx}|_{\mathbf{1}^\perp}\|  \ge (1+\epsilon)^{\ell}\rho_{n,\xxx}^\ell\right) \le \mathbb{P}\left(\|{A_{\ell,\xxx}|}_{\boldsymbol{1}^\perp}\| \ge e^{\frac{\epsilon\ell}{2} }\rho_{n,\xxx}^{\ell}\right).
	\end{align*}
	Adjusting the constants, the second statement of \cref{bigtheorem} becomes a corollary of the following result applied to $\ell = \frac{(1-c)\ln(N)}{7K_{\xxx} \ln(d)} $ where
	\begin{equation}\label{eq:defiK_x}
		K_{\sss}= 1 \vee (2\ln(1+ \delta + \kappa)) ~~\text{and}~~K_{\ccc}=  1 \vee \left(2\ln\left(\frac{1+ r\delta + r\kappa}{1-r}\right)\right).
	\end{equation}
	The condition \ref{H:2} implies that $K_{\xxx}^\prime$ is uniformly bounded in $N$.
	\begin{thm}\label{big2theorem}
		Let $n \ge 0$, $(M,Q)$ and $\rho_{n,\xxx}$ be as in \cref{bigtheorem}. For any $0< c < 1$, there exists $C>0$ depending on $n$, $K$ and $\alpha$, such that for any $1 \le \ell\le\frac{\ln(N)}{7\ln(d)}$ we have
		$$\mathbb{P}\left(\|A_{\xxx}^{\ell}|_{\boldsymbol{1}^\perp}\| \ge e^{C \ln(N)^{1-\alpha} }\rho_{n,\xxx}^\ell\right) \le N^{-c} + C \frac{\ell^4(4d)^{3\ell}}{N}.$$ 
	\end{thm}
	To estimate an upper bound on $\|(A_{\ell,\xxx})|_{\boldsymbol{1}^\perp}\|$, we use the high trace method, i.e. Markov inequality combined with the inequality
	
	\begin{equation*}
		\mathbb{E}\left(\|B\|^{2m}\right) = \mathbb{E}\left(\|BB^\top\|^m\right) \le \mathbb{E}\left(\Tr(BB^\top)^m\right)
	\end{equation*}
	for any matrix $B$ and integer $m\ge 1$.

	\subsection{Path decomposition}
	We will here interpret the coordinates of the matrix $A_{\ell,\xxx}$ as paths in a graph. 
	\begin{defi}
		A \textbf{path of length $k$} is a sequence $\gamma= (x_1,y_1,\ell_1,x_2,y_2,...,y_T,\ell_T,x_{T+1})$ such that 
		and for all $t \le T$, we have $x_t,y_t \in [N]$, $\ell_t \in \{0,...,k\}$, $(Q^{\ell_t})_{y_t,x_{t+1}}> 0$ and
		$$\sum_{t=1}^{T} \ell_t + T -1 = k.$$
		We denote by $\Gamma^k$ the set of paths of length $k$. For $x,y \in [N]$ we denote by $\Gamma^{k}_{xy}$ the subset of $\Gamma^{k}$ of paths verifying $x_1 = x$ and $y_T= y$.
		
	\end{defi}
	
	By convention we assume that the sum over an empty set is $0$ and the product over an empty set is $1$. With the previous definition we have 
	
	\begin{equation*}
		(A_{\ell,\sss})_{xy} = \sum_{\gamma \in \Gamma_{xy}^\ell}\prod_{t=1}^{T_\gamma}M_{x_t y_t} (Q^{\ell_t})_{y_t x_{t+1}} .
	\end{equation*}
	In the previous formula, we let appear the dependence of $T=T_\gamma$ in the choice of path $\gamma \in \Gamma_{xy}^{\ell}$. To make the expressions lighter, from now on we will omit this dependency when there is no confusion. To be able to estimate the previous product we look for paths that are highly probable and better suited for estimation.
	
	\begin{defi}\label{deficoincid} Let $\ell \ge 0$ be fixed.
		\begin{itemize}
			\item 
			We call a \textbf{coincidence} a sequence $c = (x_1, y_1,..., x_T , y_T,x_{T+1})$ such that $(x_2,...,x_{T+1})$ are pairwise distinct and such that there exists $(\ell_1,...,\ell_T)$ a $T$-tuple of integers verifying $\sum_t \ell_t \le \ell $, $\gamma=(x_1,y_1,\ell_1,x_2,...,y_T,\ell_T,x_{T+1})$ is a path and $(B_1^{\ell_1} B_{2}^{\ell_2}\cdots B^{\ell_T}_{T})_{x_1,x_{T+1}} >0$ where $B_1,...,B_T \in \{Q,Q^{\ast}\}$ are alternate terms (i.e. $B_{t+1}=B_{t}^\ast$).
			
			\item We say that a path $\gamma = (x_1,y_1, \ell_1,x_2,...,x_T, y_T,\ell_{T},x_{T+1})$ of length $\ell$ is \textbf{tangle-free} if the sequence $(x_1, y_1,x_2,...,x_T, y_T,x_{T+1})$ contains, as a subsequence, at most one coincidence. We denote by $\mathrm{F}^{k}$ (resp. $\mathrm{F}_{xy}^k$) the subset of $\Gamma^k$ (resp. $\Gamma_{xy}^k$) of path of length $k$ (resp. path of length $k$ from $x$ to $y$) that are tangle-free.
			
			\item The pair $(M,Q)$ is $\ell$-\textbf{tangle-free} if for all $k \in [\ell]$, for all $\gamma= (x_1,y_1,\ell_1,...,y_T,\ell_T,x_{T+1}) \in \Gamma^k \backslash \mathrm{F}^k$ we have 
			
			$$ \prod_{t=1}^T M_{x_t y_t} = 0.$$
			
		\end{itemize}
	\end{defi}
	When the pair $(M,Q)$ is $\ell$-tangle-free, we have the following equality
	
	\begin{equation*}
		A_{\ell,\sss} = A_{\sss}^{(\ell)}
	\end{equation*}
	where we defined $A_{\sss}^{(\ell)}$ by
	\begin{equation}\label{6}
		(A_{\sss}^{(\ell)})_{xy} := \sum_{\gamma \in \mathrm{F}^{\ell}_{xy}} \prod_{t=1}^{T} M_{x_t y_t} (Q^{\ell_t})_{y_t x_{t+1}},
	\end{equation}
	for $x,y \in [N]$. Similarly we define the matrix $\underline{A}_{\sss}^{(\ell)}$ by its coordinates
	\begin{equation}\label{7}
		(\underline{A}_{\sss}^{(\ell)})_{xy} := \sum_{\gamma \in \mathrm{F}^{\ell}_{xy}} \prod_{t=1}^{T} \underline{M}_{x_t y_t} (Q^{\ell_t})_{y_t x_{t+1}}
	\end{equation}
	where the matrix $\underline{M}$ is the centered variable
	
	\begin{equation*}
		\underline{M} := M - \frac{1}{N} \boldsymbol{1} \otimes \boldsymbol{1},
	\end{equation*}
	where for all $x,y \in [N]$, $(\boldsymbol{1} \otimes \boldsymbol{1})_{xy}=1$. We want to end up with a similar decomposition than in \cite[Lemma 1]{bord2018spectral}. To do such a decomposition we write the product terms in the expression \eqref{7} using the following identity
	\begin{equation}\label{9}
		\prod_{t=1}^{T} a_t = \prod_{t=1}^{T} b_t + \sum_{k=1}^{T} \prod_{t=1}^{k-1} b_t (a_k - b_k) \cdot \prod_{t=k+1}^{T} a_t.
	\end{equation}
	However the dependence $T= T(\gamma)$ forces us to further segment the sum \eqref{7}. Therefore we introduce some notations\label{notationcombi}. For integers $1 \le T \le \ell$ and a vector of integers $\vec{\ell} \in \mathcal{L}(T,\ell)$, we denote by $\mathrm{F}^{\vec{\ell}}$ (\textit{resp.} $\mathrm{F}^{\vec{\ell}}_{xy}$) the subset path $\gamma =  (x_1,y_1,\ell_1,x_2,...,x_T,y_T,\ell_T,x_{T+1}) \in \mathrm{F}^{\ell}$ (resp $\gamma \in \mathrm{F}^{\ell}_{xy}$) such that $(\ell_1,...,\ell_T) = \vec{\ell}$. Then the equality \eqref{6} becomes
	\begin{align}\label{10}
		(A_{\sss}^{(k)})_{xy} &= \sum_{T=1}^{\ell} \sum_{\vec{\ell} \in \mathcal{L}(T,\ell)} \sum_{\gamma \in \mathrm{F}^{\vec{\ell}}_{xy}} \prod_{t=1}^{T} M_{x_t y_t} (Q^{\ell_t})_{y_t x_{t+1}} = \sum_{T=1}^{\ell} \sum_{\vec{\ell} \in \mathcal{L}(T,\ell)} (A^{(\vec{\ell})})_{xy}
	\end{align}
	where for $\vec{\ell} \in \mathcal{L}(T,\ell)$ we set
	\begin{equation}\label{eq:defPvecl}
		(A_{\sss}^{(\vec{\ell})})_{xy} = \sum_{\gamma \in \mathrm{F}^{\vec{\ell}}_{xy}} \prod_{t=1}^{T} M_{x_t y_t} (Q^{\ell_t})_{y_t x_{t+1}} .
	\end{equation}
	We define similarly  
	\begin{equation}\label{eq:defPveclubarr}
		(\underline{A}_{\sss}^{(\vec{\ell})})_{xy} = \sum_{\gamma \in \mathrm{F}^{\vec{\ell}}_{xy}} \prod_{t=1}^{T} \underline{M}_{x_t y_t} (Q^{\ell_t})_{y_t x_{t+1}} .
	\end{equation}
	We consider the integers $1 \le T \le \ell$ and $\vec{\ell} \in \mathcal{L}(T,\ell)$ fixed. Owing to \eqref{9}
	we have
	\begin{equation*}
		(A_{\sss}^{(\vec{\ell})})_{xy} = (\underline{A}_{\sss}^{(\vec{\ell})})_{xy} + \sum_{\gamma \in \mathrm{F}^{\vec{\ell}}_{xy}}  \sum_{k=1}^{T} \left(\prod_{t=1}^{k-1} \underline{M}_{x_t y_t} (Q^{\ell_t})_{y_t x_{t+1}}\right) \frac{(Q^{\ell_k})_{y_k x_{k+1}}}{N} \prod_{t=k+1}^{T} M_{x_t y_t} (Q^{\ell_t})_{y_t x_{t+1}} .
	\end{equation*}
	The integer $T$ does not depend on the path $\gamma$ anymore and therefore we have
	\begin{equation*}
		(A_{\sss}^{(\vec{\ell})})_{xy} = (\underline{A}_{\sss}^{(\vec{\ell})})_{xy} + \sum_{k=1}^{T} \frac{1}{N}\underbrace{\sum_{\gamma \in \mathrm{F}^{\vec{\ell}}_{xy}} \prod_{t=1}^{k-1} \underline{M}_{x_t y_t} (Q^{\ell_t})_{y_t x_{t+1}} (Q^{\ell_k})_{y_k x_{k+1}} \prod_{v=k+1}^{T} M_{x_v y_v} (Q^{\ell_v})_{y_v x_{v+1}}  }_{(S^{(\vec{\ell})})_{xy}}.
	\end{equation*}
	We rewrite $(S^{(\vec{\ell})})_{xy}$ so that lower powers of $A^{(\vec{\ell})}$ appear. For all $1 \le k \le T$ we set $\vec{\ell}(k) := (\ell_1,...,\ell_{k-1})$ and $\vec{\ell}(k,T) := (\ell_{k+1},...,\ell_{T})$. We denote by $\mathrm{T}^{k, \vec{\ell}}_{xy}$ the set of paths  $\gamma := (x, y_1,...,y_T, \ell_T,y)\in \Gamma_{xy}^{\vec{\ell}}$ that verifies: 
	
	\begin{enumerate}
		\item $\gamma^{\prime} := (x,y_1,...,x_{k-1},y_{k-1},\ell_{k-1},x_k) \in \mathrm{F}^{\vec{\ell}(k)}$,
		\item $\gamma^{\prime \prime} := (x_{k+1},y_{k+1},...,x_T,y_T,\ell_T,x_{T+1})\in \mathrm{F}^{\vec{\ell}(k,T)}$,
		\item $\gamma = (\gamma^{\prime},y_k,\ell_k,\gamma^{\prime \prime})$ is tangled.
	\end{enumerate}
	To rewrite $S(\vec{\ell},x,y)$, we decompose the last summand noticing that we have
	
	$$\mathrm{F}_{xy}^{\vec{\ell}} \sqcup \mathrm{T}_{xy}^{k,\vec{\ell}} = \bigsqcup_{x_k \in [N]}\bigsqcup_{y_k \in [N]} \bigsqcup_{x_{k+1} \in [N]} \{\gamma = (\gamma^{\prime},y_k,\ell_k,\gamma^{\prime \prime}) ~~ ; ~~ \gamma^{\prime} \in \mathrm{F}^{\vec{\ell}(k)}_{x x_k},~ \gamma^{\prime \prime} \in \mathrm{F}^{\vec{\ell}(k,T)}_{x_{k+1} y}\} $$
	and introducing the rest 
	
	\begin{equation}\label{eq:defirestvecell}
		(R_{k,\sss}^{(\vec{\ell})})_{xy} := \sum_{\gamma \in \mathrm{T}^{k,\vec{\ell}}_{xy}} \prod_{t=1}^{k-1} \underline{M}_{x_t y_t} (Q^{\ell_t})_{y_t x_{t+1}} (Q^{\ell_k})_{y_k x_{k+1}} \prod_{v=k+1}^{T} M_{x_v y_v} (Q^{\ell_v})_{y_v x_{v+1}}.
	\end{equation}
	We then have 
	\begin{align*}
		(S^{(\vec{\ell})})_{xy} &= \sum_{x_{k} ,y_k,x_{k+1}= 1}^{N} \sum_{\substack{\gamma^{\prime}\in \mathrm{F}_{x x_{k}}^{\vec{\ell}(k)}\\\gamma^{\prime \prime} \in \mathrm{F}^{\vec{\ell}(k,T)}_{x_{k+1} y}}}  \prod_{t=1}^{k-1} \underline{M}_{x_t y_t} (Q^{\ell_t})_{y_t x_{t+1}} (Q^{\ell_k})_{y_k x_{k+1}} \prod_{v=k+1}^{T}M_{x_v y_v} (Q^{\ell_v})_{y_v x_{v+1}}-(R_{k,\sss}^{(\vec{\ell})})_{xy}\\
		&= \sum_{y_k=1}^{N}\sum_{x_k=1}^{N}  \left(\underline{P}^{(\vec{\ell}(k))}\right)_{x x_k} (Q^{\ell_k})_{y_k x_{k+1}} \sum_{x_{k+1}=1}^n \left(A_{\sss}^{(\vec{\ell}(k,T))}\right)_{x_{k+1} y} -(R_{k,\sss}^{(\vec{\ell})})_{xy}\\
		& = \left( A_{\sss}^{(\vec{\ell}(k))} Q^{\ell_k}(1\otimes 1)A_{\sss}^{(\vec{\ell}(k,T))}\right)_{xy}  -(R_{k,\sss}^{(\vec{\ell})})_{xy}
	\end{align*}
	and therefore we obtain
	$$A_{\sss}^{(\vec{\ell})} = \underline{A}_{\sss}^{(\vec{\ell})} + \frac{1}{N} \sum_{k=1}^{T-1}  \underline {A}_{\sss}^{(\vec{\ell}(k))} Q^{\ell_k}(1\otimes 1)A_{\sss}^{(\vec{\ell}(k,T))} - \frac{1}{N} \sum_{k=1}^{T} R_{k,\sss}^{(\vec{\ell})}.$$
	Inserting it into the expression \eqref{10} we have
	
	\begin{equation*}
		A_{\sss}^{(\ell)} = \underline{A}_{\sss}^{(\ell)} + \sum_{T=1}^{\ell} \sum_{\vec{\ell} \in \mathcal{L}(T,\ell)} \left[ \sum_{k=1}^{T}  \underline{A}_{\sss}^{(\vec{\ell}(k))} Q^{\ell_k}\boldsymbol{1} \otimes \boldsymbol{1} A_{\sss}^{(\vec{\ell}(k,T))} -\frac{1}{N} \sum_{k=1}^{T} R_{k,\sss}^{(\vec{\ell})}\right].
	\end{equation*}
	To estimate the operator norm of $A_{\sss}^{(\ell)}|_{\boldsymbol{1}^\perp}$ we first show that the terms with $\boldsymbol{1} \otimes \boldsymbol{1}$ in the previous formula have no contribution. For $1 \le T \le \ell$ and $\vec{\ell} \in \mathcal{L}(T,\ell)$, we set the notations $\ell(k) := \sum_{t=1}^{k-1} \ell_t + (k-1) -1$ and $\ell(k,T) := \sum_{t=k+1}^{T} \ell_t + T-k - 1$. We then have $\ell(k) + \ell(k,T) + \ell_k = \ell -2$ and we can partition $\mathcal{L}(\ell,T)$ as follows
	$$\mathcal{L}(\ell,T) = \bigsqcup_{\ell_k=0}^{\ell-T+1}~~~ \bigsqcup_{\ell(k)=0}^{\ell-\ell_k-2} \mathcal{L}(\ell(k),k-1)\times \{\ell_k\}\times \mathcal{L}(\ell-\ell_k-2-\ell(k),T-k)$$
	We consider the matrix
	\begin{align*}
		\phi(\ell)&:=\sum_{T=1}^{\ell}\sum_{k=1}^{T} \sum_{\vec{\ell} \in \mathcal{L}(T,\ell)}  \underline{A}_{\sss}^{(\vec{\ell}(k))} Q^{\ell_k}\boldsymbol{1} \otimes \boldsymbol{1} A_{\sss}^{(\vec{\ell}(k,T))}\\
		& = \sum_{T=1}^{\ell}\sum_{k=1}^{T} \sum_{\ell_k=0}^{\ell}\sum_{\ell(k)=0}^{\ell-\ell_k-2} \sum_{\substack{\vec{\ell}(k) \in \mathcal{L}(\ell(k),k-1)\\
				\vec{\ell}(k,T) \in \mathcal{L}(\ell-\ell_k-\ell(k)-2, T-k)}}\underline{A}_{\sss}^{(\vec{\ell}(k))} Q^{\ell_k}\boldsymbol{1} \otimes \boldsymbol{1} A_{\sss}^{(\vec{\ell}(k,T))}\\
		&= \sum_{T=1}^{\ell}\sum_{k=1}^{T} \sum_{\ell_k=0}^{\ell}\sum_{\ell(k)=0}^{\ell-\ell_k-2}\underline{A}_{\sss}^{(\ell(k))} Q^{\ell_k}\boldsymbol{1} \otimes \boldsymbol{1} A_{\sss}^{(\ell-\ell_k-\ell(k)-2)}.
	\end{align*}
	On the event $(M,Q)$ $\ell$-tangle free, for all $0 \le k \le \ell$ and for all $w \in \boldsymbol{1}^{\perp}$ we have
	$$\boldsymbol{1} \otimes \boldsymbol{1} A_{\sss}^{(k)} w = \boldsymbol{1} \otimes \boldsymbol{1} A_{k,\sss} w= 0.$$ 
	In particular on this event we have $\phi(\ell)|_{\boldsymbol{1}^{\top}} =0$.
	Besides denoting
	$$R_{k,\sss}^{(\ell)} :=\sum_{T=1}^{\ell} \sum_{\vec{\ell} \in \mathcal{L}(T,\ell)}  R_{k,\sss}^{(\vec{\ell})},$$
	we have
	\begin{equation*}\label{15}
		A_{\sss}^{(\ell)} = \underline{A}_{\sss}^{(\ell)} +\phi(\ell) - \frac{1}{N} \sum_{k=1}^{\ell} R^{(\ell)}_{k,\sss}.
	\end{equation*}
	To have consistent writings between the operators $R_{k,\sss}^{(\ell)}$ and $A_{\sss}^{(\ell)}$, for $1\le k \le \ell$ and $x,y \in [N]$ we set
	$$\mathrm{T}_{xy}^{k,\ell}:= \bigsqcup_{T=1}^{\ell} \bigsqcup_{\vec{\ell} \in \mathcal{L}(\ell,T)} \mathrm{T}_{xy}^{k,\vec{\ell}}$$
	where we consider $\mathrm{T}_{xy}^{k,\vec{\ell}} = \emptyset$ if $k> T$. With this convention and notations we have
	
	$$(R_{k,\sss}^{(\ell)})_{xy} = \sum_{\gamma \in \mathrm{T}_{xy}^{k,\ell}}\left(\prod_{t=1}^{k-1} \underline{M}_{x_t y_t} (Q^{\ell_t})_{y_t x_{t+1}}\right) (Q^{\ell_k})_{y_k x_{k+1}} \left(\prod_{t=k+1}^{T} M_{x_t y_t} (Q^{\ell_t})_{y_t x_{t+1}}\right).$$ 
	When it comes to the case $\xxx=\ccc$, all the previous reasoning holds if now one considers $A_{\mathrm{c}}= (1-r)M + rQ$ and replace respectively \eqref{eq:defPvecl} and \eqref{eq:defPveclubarr} with
	\begin{equation*}\label{eq:defPvecl2}
		(A_{\mathrm{c}}^{(\vec{\ell})})_{xy} = \sum_{\gamma \in \mathrm{F}^{\vec{\ell}}_{xy}} (1-r)^T \prod_{t=1}^{T} M_{x_t y_t} (\{rQ\}^{\ell_t})_{y_t x_{t+1}}~~\text{and}~~
		(\underline{A}_{\ccc}^{(\vec{\ell})})_{xy} = \sum_{\gamma \in \mathrm{F}^{\vec{\ell}}_{xy}} (1-r)^T \prod_{t=1}^{T} \underline{M}_{x_t y_t} (\{rQ\}^{\ell_t})_{y_t x_{t+1}}.
	\end{equation*}
	We also have a new rest term
	\begin{equation}\label{eq:defirestvecell2}
		(R_{k,\ccc}^{(\vec{\ell})})_{xy} := (1-r)^{T-1}\sum_{\gamma \in \mathrm{T}^{k,\vec{\ell}}_{xy}} \prod_{t=1}^{k-1} \underline{M}_{x_t y_t} \{rQ\}^{\ell_t})_{y_t x_{t+1}} \{rQ\}^{\ell_k})_{y_k x_{k+1}} \prod_{v=k+1}^{T} M_{x_v y_v} \{rQ\}^{\ell_v})_{y_v x_{v+1}}.
	\end{equation}
	
	We have now proved the following lemma.
	\begin{lem}\label{MQltf}
		On the event $(M,Q)$ $\ell$-tangle free we have
		
		\begin{equation*}
			\|A_{\xxx}^{\ell}|_{\boldsymbol{1}^{\perp}}\| \le \|\underline{A}_{\xxx}^{(\ell)}\|+ \frac{1}{N}  \sum_{k=1}^{\ell} \|R^{(\ell)}_{k,\xxx}\|.
		\end{equation*}
		
	\end{lem}

	\subsection{Computations on random permutations}\label{section:comput_rd_permut}
	
	In this section we check that for $M\in \mathrm{S}_N$ uniformly distributed and $Q$ deterministic verifying \ref{H:3}, the pair $(M,Q)$ is $\ell$-tangle free with high probability, provided that the integer $\ell$ is not too large before the dimension. Recall that for $Q$ a matrix we denote 
		$$d:= \underset{x \in [N]}{\max}|\{y \in [N]~;~ Q_{xy} \ne 0\}|\vee \underset{x \in [N]}{\max}|\{y \in [N]~;~ Q^\ast_{xy} \ne 0\}|.$$
	\begin{lem}\label{pr1}
		Let $Q$ be a deterministic matrix and $M \sim \mathcal U (\mathrm{S}_N)$. There exists $C> 0$ such that, for any integer $\ell\ge h$, the pair $(M,Q)$ is $\ell$-tangle free with probability at least $1 - C \frac{\ell^4(4d)^{3\ell}}{N}$.
	\end{lem}
	
	\begin{proof}
		We consider $1 \le \ell \le \frac{N}{2}$. We say that a path $\gamma:= (x_1, y_1,\ell_1,x_2 ..., x_T, \ell_{T-1} , y_T,\ell_{T},x_{T+1})$ \textit{occurs} if for all $t \le T$, we have $M_{x_t,y_t}=1$. If $(M,Q)$ is $\ell$-tangled then there exists $1 \le T \le \ell$, $\ell_1,..., \ell_T$ integers verifying $\sum_{i=1}^{T} \ell_i + T -1 = \ell$
		and a path $\gamma :=(x_1, y_1,\ell_1,x_2,y_2,...,x_T, \ell_T, y_T,\ell_T,x_{T+1})$ such that there exists $1 \le k,i < k^{\prime} \le T$ such that either
		\begin{enumerate}
			\item $(x_1,y_1,...,x_k,y_k,x_{k+1})$ and $(x_{i}, y_{i},...,x_T, y_T,x_{T+1})$ are two distinct coincidences where $x_{t}$ are pairwise distinct except possibly $x_1 = x_{k+1}$ and $x_{i}= x_{T+1}$. We denote by $I_1 (k,i)$ this event,
			\item $(x_1,y_1,...,x_T,y_T,x_{T+1})$ and $(x_k,y_k,...,x_{k^{\prime}+1})$ are two coincidences (one inside the other) where all $x_t$ are pairwise distinct except possibly $x_1 = x_{T+1}$. We denote by $I_2 (k,k^\prime)$ the corresponding event.
		\end{enumerate}
		The probability for a path $\gamma=(x_1,y_{1},\ell_{1},x_2,...,x_{T+1})$ to occur is bounded by
		$$\mathbb{P} (\gamma) = \mathbb{P} (\prod_{t=1}^{T} \mathds{1}\{\sigma(x_t) = y_{t}\}) \le \frac{1}{(N)_{T}},$$
		where $(N)_t := N(N-1)(N-2)\cdots(N-t+1)$. 
		Let $\{x_t\}_t $ be pairwise distinct elements of $[N]$, except possibly $x_1 = x_{k+1}$ and $x_{i} =x_{T+1}$ for some $1 \le k,i < T \le \ell$. We count the number of possible $\gamma$ verifying $I_{1}(k,i)$. For $\{x_{t}\}_{t \notin \{k,T+1\}}$ we have $(N)_{T-1}$ possibilities in regard of the condition of being pairwise distinct. For each $(\ell^{\prime}_{t})_{t \le T^\prime}$ such that $\sum_{t} \ell^\prime_{t} + T^\prime -1 \le \ell$, we have at most $2d^{\tiny{\sum_{t} \ell^\prime_t}} \le 2d^\ell$ possibilities for $x_{k+1}$ to have $(B_1^{\ell_{1}^\prime}\cdots B_{T^\prime}^{\ell_{T^\prime}^\prime})_{x_1x_{k+1}} >0$ with $B_t \in \{Q, Q^\ast\}$ and $B_{t+1} = B_{t}^\ast$. Besides the number of $T^\prime \le \ell$ and $(\ell_{t}^\prime)_{t \le T^\prime}$ such that $\sum_{t} \ell^\prime_{t} + T^\prime -1 \le \ell$ is given by
		
		$$|\bigsqcup_{p=1}^{\ell} \bigsqcup_{T^{\prime}=1}^{p}\{ (\ell^{\prime}_{t})_{t\le T^\prime}:~ \sum_{t=1}^{T^\prime} \ell_t^\prime + T^\prime -1 =p\}| = \sum_{p=1}^{\ell}\sum_{T^\prime=1}^{p} \binom{p-1}{T^\prime-1} \le 2^\ell.$$
		Therefore we have at most $(4d)^{2\ell}$ possibilities for $x_{k+1}$ and $x_{T+1}$. For fixed $1 \le T \le \ell$ and $\vec{\ell}=(\ell_{t})_{t\le T}\in \mathcal{L}(\ell,T)$ we have at most $d^{\sum_{t} \ell_t} = d^{\ell-T+1} \le d^{\ell}$ possibilities for the $y_t$. To chose $(\ell_{t})_{t\le T}$ we have again $2^{\ell}$ possibilities. We have
		
		$$\mathbb{P}\left(I_{k,i}\right) \le \frac{ (N)_{T-1} (4d)^{2\ell} (2d)^\ell }{(N)_{T}}\le \frac{(4d)^{3\ell}}{N}.$$
		We proceed similarly to bound the number of $\gamma$ verifying $I_{2}(k,k^\prime)$ and we obtain the exact same bound. Now summing over all $T$, $k$, $i$ and $k^\prime$ we obtain the result.
		
	\end{proof}
	
	Let $\mathbf{x}:=(x_1,...,x_k)$, $ \mathbf{y}:=(y_1,...,y_k) \in [N]^k$. We are looking for an estimate of
	$$\mathbb{E}\left(\prod_{t=1}^{k_0}\underline{M}_{x_t y_t}\prod_{t=k_0+1}^{k}M_{x_t y_t}\right)$$
	where $0 \le k_0 \le k$. To do this, we define the set of \textit{arcs} of $(\mathbf{x},\mathbf{y})$ as:
	$$\mathrm{A}_{\mathbf{x,y}}:=\{(x_t,y_t): t \in [k]\}.$$
	The \textbf{multiplicity} of $e \in \mathrm{A}_{\mathbf{x,y}}$ is given by $m(e):= \sum_{t=1}^k \mathds{1}((x_t,y_t)=e)$. An arc $e=(x,y)$ is said to be \textbf{consistent} if $\{t: (x_t,y_t)=(x,y)\}= \{t: x_t =x\}= \{t: y_t =y\}$. Otherwise, it is \textbf{inconsistent}. The following proposition is proved in \cite[Proposition 11]{bordenave2015new}.
	\begin{prop}\label{estimconsistent}
		There exists a constant $c>0$ such that for any $\mathbf{x}:=(x_1,...,x_k), \mathbf{y}:=(y_1,...,y_k) \in [N]^k$ with $2k \le \sqrt{N}$ and any $k_0 \le k$, we have
		
		$$\left|\mathbb{E}\left(\prod_{t=1}^{k_0}\underline{M}_{x_t y_t}\prod_{t=k_0+1}^{k}M_{x_t y_t}\right)\right| \le c 2^b \left(\frac{1}{N}\right)^a\left(\frac{3k}{\sqrt{N}}\right)^{a_1},$$
		where $a:=|\mathrm{A}_{\mathbf{x,y}}|$, b is the number of inconsistent arcs of $(\mathbf{x,y})$ and $a_1$ is the number of $1 \le t \le k_0$ such that $(x_t,y_t)$ is consistent and has multiplicity one.
	\end{prop}

	\section{High trace method}\label{section:high_tr_meth}
	
	\subsection{Operator norm of $\underline{A}_{\xxx}^{(\ell)}$}\label{subsectionboundofP}
	
	In this section we prove the following assertion.
	
	\begin{prop}\label{boundingnormP}
		Let $(Q,M),~ \rho_{n,\xxx}$ be as defined in \cref{bigtheorem} and for $\xxx \in \{\sss, \ccc\}$ we consider $A_{\xxx}$ defined by \eqref{eq:1origins} and \eqref{eq:2origins}. For any $c >0$ there exists $C = C(c, \alpha, n, K) >0$ such that with probability at least $1 - N^{-c}$,
		$$\| \underline{A}_{\xxx}^{(\ell)} \| \le e^{C \ln(N)^{1-\alpha}} \rho_{n,\xxx}^{\ell}$$
		for any $1 \le \ell \le \frac{\ln(N)}{7K_{\xxx}(\ln(d) \vee 2\ln(\delta \vee 1))}$ where $K_{\xxx}$ is defined by \eqref{eq:defiK_x}.
	\end{prop}
	We give the details of computations in the case where $\xxx = \sss$. For any integer $m \ge 1$ and considering as convention $x_{2m+1}:= x_1$ we have
	
	\begin{align*}
		&\|\underline{A}_{\sss}^{(\ell)} \|^{2m} = \|\underline{A}_{\sss}^{(\ell)} {\underline{A}_{\sss}^{(\ell)}}^{\top}\|^m \le \Tr \{ \left( \underline{A}_{\sss}^{(\ell)} {\underline{A}_{\sss}^{(\ell)}}^{\top}\right)^m \}\\
		& = \sum_{x_1,...,x_{2m}} \prod_{i=1}^{m} (\underline{A}_{\sss}^{(\ell)})_{x_{2i-1},x_{2i}}(\underline{A}_{\sss}^{(\ell)})_{x_{2i+1},x_{2i}}\\
		& = \sum_{x_1,...,x_{2m}} \prod_{i=1}^{m}  \bigg\{ \sum_{\gamma_{2i-1} \in \mathrm{F}^{\ell}_{x_{2i-1},x_{2i}}} \prod_{t=1}^{T_{2i-1}} \underline{M}_{x_{2i-1,t}, y_{2i-1,t}} (Q^{\ell_{2i-1,t}})_{y_{2i-1,t}x_{2i-1,t+1}}\cdot \\
		&~~~~~~~~~~~~~~~~~~~~~~~~~~~~~~~~~~~~~~~~~~~~~~~~~~~~~~~~~~~~~~~~~~~~~~~~~\sum_{\gamma_{2i} \in \mathrm{F}^{\ell}_{x_{2i+1},x_{2i}}} \prod_{t=1}^{T_{2i}} \underline{M}_{x_{2i,t}, y_{2i,t}} (Q^{\ell_{2i,t}})_{y_{2i,t}x_{2i,t+1}} \bigg\}\\
		&= \sum_{x_1,...,x_{2m}} \displaystyle\sum_{\substack{\gamma_1,..., \gamma_{2m} \\ \gamma_{2i-1} \in \mathrm{F}^{\ell}_{x_{2i-1},x_{2i}} \\ \gamma_{2i} \in \mathrm{F}^{\ell}_{x_{2i+1},x_{2i}}} } \prod_{i=1}^{m}  \prod_{t=1}^{T_{2i-1}} \underline{M}_{x_{2i-1,t}, y_{2i-1,t}} (Q^{\ell_{2i-1,t}})_{y_{2i-1,t}x_{2i-1,t+1}} \cdot \prod_{t=1}^{T_{2i}} \underline{M}_{x_{2i,t}, y_{2i,t}} (Q^{\ell_{2i,t}})_{y_{2i,t}x_{2i,t+1}} 
	\end{align*}
	where we use the notation $\gamma_i =(x_{i,1},y_{i,1},\ell_{i,1},x_{i,2}..., y_{i,T_i},\ell_{i,T_i},x_{i,T_i +1}) \in \mathrm{F}^{\ell}$. For $\gamma_{2i-1} \in \mathrm{F}^{\ell}_{x_{2i-1}, x_{2i}}$ (\textit{resp.} for $\gamma_{2i} \in \mathrm{F}^{\ell}_{x_{2i+1},x_{2i}}$) we have with the previous notation $x_{2i-1,1}=x_{2i-1}$ (\textit{resp.} $x_{2i,1}=x_{2i+1}$) and $x_{2i-1, T_{2i-1}+1}=x_{2i}$ (\textit{resp.} $x_{2i,T_{2i}+1}= x_{2i}$).\\
	For $\vec{T} = (T_1,...,T_{2m}) \in [\ell]^{2m}$ we introduce the set $\mathrm{W}_{\ell, \vec{T}}$ of $\gamma=(\gamma_1,...,\gamma_{2m})$ such that for all $1 \le i\le 2m$, $\gamma_i=(x_{i,1},y_{i,1},\ell_{i,1},x_{i,2}..., y_{i,T_i},\ell_{i,T_i},x_{i,T_i +1}) \in \mathrm{F}^{\ell} $ and
	\begin{equation}\label{eq:condilim}
		x_{2i,1} = x_{2i+1,1} \qquad \text{ and } \qquad x_{2i-1, T_{2i-1}+1} = x_{2i,T_{2i}+1},
	\end{equation}
	with the convention $x_{2m+1,1}=x_{1,1}$. We then obtain as a first bound
	\begin{equation}\label{17}
		\|\underline{A}_{\sss}^{(\ell)} \|^{2m} \le \sum_{\vec{T} \in [\ell]^{2m}}\sum_{\gamma \in \mathrm{W}_{\ell,\vec{T}}} \prod_{i=1}^{m} \prod_{t=1}^{T_i } \underline{M}_{x_{i,t} y_{i,t}} (Q^{\ell_{i,t}})_{y_{i,t} x_{i,t+1}}.
	\end{equation}
	We want to estimate the expectation of the above sum using \cref{estimconsistent} and the fact that $\gamma$ is composed by $2m$ $\ell$-tangled free paths. We decompose $\mathrm{W}_{\ell,\vec{T}}$ using the size of the support for $\gamma$. For $\gamma=(\gamma_1,...,\gamma_{2m}) \in \mathrm{W}_{\ell,\vec{T}}$, we set $X_{\gamma} := \{x_{i,t}~; i \in [2m], ~ t \in [T_i+1] \}$ and 
	$Y_{\gamma}:= \{ y_{i,t}~; i\in [2m]~, t \in [T_i] \}$. We consider the graph $K_{\gamma}$ defined as follows. The vertex set is given by $X_{\gamma}$ and for any $x, x^{\prime} \in X_{\gamma}$, $\{x,x^{\prime}\}$ is a vertex if there exists \textit{times} $(i,t)$ and $(i^{\prime},t^{\prime})$ such that $i,i^{\prime} \in [2m]$ and $(t,t^{\prime})\in [T_i] \times [T_{i^{\prime}}]$ such that
	$$x_{i,t+1} = x,~~ x_{i^{\prime},t^{\prime}+1} = x^{\prime} ~~ \text{and} ~~ y_{i,t}= y_{i^{\prime},t^{\prime}}.$$
	
	The graph $K_{\gamma}$ induces an equivalence relation on the set $X_{\gamma}$ where the equivalence classes are the connected components of the graph. For $x \in X_{\gamma}$ we denote
	$$ \cc(x) := \text{connected component of } x.$$
	We also set the set of \textit{arcs} of $\gamma$, denoted $\mathrm{A}_{\gamma}$ as the set of distinct pairs $(x_{i,t},y_{i,t})$. We then consider the partition $\{\mathrm{W}_{\ell,\vec{T}}(s,a,p)\}_{s,a,p}$ where  $\mathrm{W}_{\ell,\vec{T}}(s,a,p)$ is the set of $\gamma \in \mathrm{W}_{\ell,\vec{T}}$ such that $|X_{\gamma}| = s$, $|\mathrm{A}_{\gamma}| =a$ and $|X_{\gamma}/\sim| = |\{\cc(x)~;~ x \in X_{\gamma}\}| = s-p$. We denote $G_{\gamma}= (X_{\gamma}/ \sim, \mathrm{A}_\gamma)$ the graph where for all arc $(x,y)=(x_{i,t},y_{i,t}) \in \mathrm{A}_{\gamma}$ we put an edge $(\cc(x_{i,t}),\cc(x_{i,t+1}))$ of color $(x,y)$.
	\begin{ex}
		We consider the path 
		$$\gamma = (2,1, \ell_1, 2, 3,\ell_2, 4,6,\ell_3,7,3,\ell_4,8).$$
		We represented the corresponding graphs $G_{\gamma}$ and $K_{\gamma}$ in \cref{fig:rpzgraphGgamma}. We have in this case $X_{\gamma} := \{2,4,8,7\}$, a unique edge in $K_{\gamma}$ given by $\{4,8\}$ (dashed in \cref{fig:rpzgraphGgamma}), $(X_{\gamma}/ \sim )= \{ \cc_1, \cc_2, \cc_3\}$ and $\mathrm{A}_{\gamma} = \{ (2,1), (2,3),(4,6),(7,3)\}$. One can also refer to \cite[Figure 4,5]{Brito_Dumitriu_Harris_2022} and \cite[Figure 3,4]{bordenave2015new} for representation of other examples of this graphs.
		\begin{center}
			\begin{figure}[h]
				\begin{tikzpicture}
					\begin{scope}[every node/.style={circle,thick,draw}]
						\node (C1) at (0,0) {$\cc_1$};
						\node (C2) at (3,0) {$\cc_2$};
						\node (C3) at (6,0) {$\cc_3$};
						\node (2x) at (9,0) {2};
						\node (1y) [shape=circle,draw=blue]  at (10,0) {{\color{blue}1}}; 
						\node (4x) at (11,0) {4};
						\node (3y) [shape=circle,draw=blue]  at (12,-0.5) {{\color{blue}3}}; 
						\node (8x) at (11,-1) {8};		
						\node (7x) at (13,0) {7};
						\node (6y) [shape=circle,draw=blue]  at (14,0) {{\color{blue}6}}; 
					\end{scope}
					
					\begin{scope}[>={Stealth[black]},
						every node/.style={fill=white,circle},
						every edge/.style={draw=green,very thick}]
						\path [->] (C1) edge node {$(2,3)$} (C2);
						\path [->] (C2) edge node {$(4,6)$} (C3);
						\path [->] (C1) edge[loop above] node {$(2,1)$} (C1); 
						\path [->] (C3) edge[bend left=60] node {$(7,3)$} (C2);
						\draw[color=blue] (2x) -- (1y);
						\draw[color=blue] (4x) -- (3y);
						\draw[color=blue] (8x) -- (3y);	\draw[dashed] (8x) -- (4x);		\draw[color=blue] (7x) -- (6y);
					\end{scope}
				\end{tikzpicture}	
				\caption{On the left hand side a representation of $G_{\gamma}$ and on the right hand side a representation of $K_{\gamma}$.}
				\label{fig:rpzgraphGgamma}
			\end{figure}
		\end{center}
	\end{ex}
	Proceeding similarly in the case where $\xxx =\ccc$ and taking the expectation of \eqref{17} we have
	
	$$\mathbb{E}\|\underline{A}_{\xxx}^{(\ell)} \|^{2m} \le \sum_{s,a,p,\vec{T}} \sum_{~\gamma \in \mathrm{W}_{\ell,\vec{T}}(s,a,p)} \mu(\gamma)q_{\xxx}(\gamma),$$
	where for $\gamma \in \mathrm{W}_{\ell,\vec{T}}$ we set
	\begin{equation}\label{mu}
		\mu(\gamma) := \mathbb{E} \prod_{i=1}^{2m} \prod_{t=1}^{T_i} \underbar{M}_{x_{i,t}y_{i,t}},
	\end{equation}
	and 
	\begin{equation}\label{eq:defiqashtag}
		~ q_{\sss}(\gamma):= \prod_{i=1}^{2m} \prod_{t=1}^{T_i} (Q^{\ell_{i,t}})_{y_{i,t}x_{i,t+1}}~~\text{and}~~	q_{\ccc}(\gamma):= (1-r)^{T}\prod_{i=1}^{2m} \prod_{t=1}^{T_i} (\{rQ\}^{\ell_{i,t}})_{y_{i,t}x_{i,t+1}}.
	\end{equation}
	To estimate the above sum we need to decompose further the set $\mathrm{W}_{\ell,\vec{T}}(s,a,p)$. For $\gamma, \gamma^{\prime} \in \mathrm{W}_{\ell,\vec{T}}(s,a,p)$, we say that $\gamma \sim \gamma^{\prime}$  if there exists a pair of permutations $(\alpha, \beta) \in \mathrm{S}_N$ such that $K_{\gamma^{\prime}}$ is the image by $\alpha$ of $K_{\gamma}$ i.e. for any $(i,t)$ we have $x^{\prime}_{i,t}= \alpha(x_{i,t})$ and $y_{i,t}^{\prime}= \beta(y_{i,t})$ (where we consider $\gamma^{\prime}= (\gamma_1^{\prime},...,\gamma_{2m}^{\prime})$ and for all $i \in [2m]$, $\gamma_i^{\prime} = (x_{i,1}^{\prime}, y_{i,1}^{\prime},\ell_{i,1}^{\prime},x_{i,2}^\prime,...,x_{i,T_i}^{\prime}, y_{i,T_i}^{\prime},\ell_{i,T_i}^{\prime},x_{i,T_i +1}^{\prime})$). We denote by $\mathcal{W}_{\ell,\vec{T}} (s,a,p)$ the set of equivalent classes. For all $\gamma^{\prime} \sim \gamma$, we have $\mu(\gamma)= \mu(\gamma^{\prime})$.

	where for $\gamma \in \mathrm{W}_{\ell,\vec{T}}$ we set $T = \sum_{i=1}^{2m} T_i$. Finally we have
	\begin{equation}\label{eq:htmethod}
		\mathbb{E}\|\underline{A}_{\xxx}^{(\ell)} \|^{2m} \le \sum_{s,a,p,\vec{T}} |\mathcal{W}_{\ell,\vec{T}}(s,a,p)| \underset{\gamma \in \mathrm{W}_{\ell,\vec{T}}(s,a,p)}{\max} \left\{ |\mu(\gamma)| \sum_{\substack{\gamma^{\prime} \in \mathrm{W}_{\ell,\vec{T}}(s,a,p)\\ \gamma^{\prime} \sim \gamma}} q_{\xxx}(\gamma^{\prime}) \right\}.
	\end{equation}
	First we bound the cardinality of $\mathcal{W}_{\ell,\vec{T}} (s,a,p)$.
	
	\begin{lem}\label{boundingWlTsap}
		We have
		$$|\mathcal{W}_{\ell,\vec{T}}(s,a,p)| \le  (\ell a (s-p)^2s^2)^{2m(g+5)}.$$
	\end{lem}
	The following section is dedicated to the proof of \cref{boundingWlTsap}.
	\subsection{Encoding equivalent classes}\label{section:encodingeqclass}
	In all this section we consider  the integers $1 \le T \le \ell$, $m \ge 1$ and $1 \le s,a,p \le 2m\ell$ fixed. To bound $|\mathcal{W}_{ \ell,\vec{T}} (s,a,p)|$ we need an efficient way to encode in a unique way an equivalent class $[\gamma]\in \mathcal{W}_{\ell,\vec{T}} (s,a,p)$. To proceed we start by introducing the order of apparition of the elements while we explore chronologically $\gamma$. This will allow us to code $[\gamma]$ by coding what we will call its \textit{canonical path}. To do so we consider $\boldsymbol{T}:= \{(i,t)~; ~ i\in [2m],~t \in [T_i]\}$ the set of times, that we consider ordered canonically, i.e. for $i< i^{\prime}$ we have for any $t \in [T_i]$ and $ t^\prime \in [T_{i^{\prime}}]$ $(i,t) \prec (i^\prime,t^{\prime})$ and for $1\le t<t^{\prime} \le T_{i}$ we have $(i,t) \prec (i,t^\prime)$. For $(i,t) \neq (1,1)$ we denote by $(i,t)^{-}$ the greatest element of $\boldsymbol{T}$ strictly smaller than $(i,t)$. For a time $(i,t) \in \boldsymbol{T}$ we set $\gamma_{i,t} :=(x_{i,t},{y}_{i,t},{x}_{i,t+1})$. For all $x \in X_{\gamma}$ (\textit{resp}. $y\in Y_{\gamma}$) we denote by $\bar{x}$ (\textit{resp}. $\bar{y}$) the order of apparition of $x$ (\textit{resp.} of $y$) in $\gamma$. We consider $\bar{\gamma}:= (\bar{x}_{1,1},\bar{y}_{1,1},\ell_{1,1},\bar{x}_{1,2},...,\bar{y}_{i,t},\ell_{i,t},\bar{x}_{i,t+1},...,\bar{y}_{2m,T_{2m}},\ell_{2m,T_{2m}},\bar{x}_{2m,T_{2m}+1})$ the canonical path and knowing $\bar{\gamma}$ is equivalent to know $[\gamma]$. In fact, we do not have necessarily   $\bar{\gamma} \in [\gamma]$, for instance  $\bar{\gamma}_{i}= (\bar{x}_{i,1},\bar{y}_{i,1},\ell_{i,1},\bar{x}_{i,2},...)$ can be tangled. However, if $K_{\bar{\gamma}}$ is constructed as we did for $K_{\gamma}$ then for all $\gamma^{\prime} \in [\gamma]$, the graph $K_{\gamma^{\prime}}$ is isomorphic to $K_{\bar{\gamma}}$.\\
	We define a graph which genius gives a first condition on the integers $s,a$ and $p$ for $\mathrm{W}_{\ell,\vec{T}}$ not to be empty. For $x\in X_{\gamma}$, we introduce  the \textit{color} of the connected component $\cc(x)$ defined by  $$\bar{c}(\bar{x}):= \min\{\bar{x}~;~ x \in \cc(x^{\prime})\}.$$
	For $(i,t)$ a time we denote by $\vec{x}_{i,t+1}:= (\bar{c}(\bar{x}_{i,t+1}),\bar{x}_{i,t+1})$ the \textit{connected component mark} of $x_{i,t+1}$. For example $\bar{x}_{1,1}= \bar{y}_{1,1}=1$ and $\vec{x}_{1,1}=(1,1)$. If $x_{1,2} \neq x_{1,1}$ then $\bar{x}_{1,2}=2$. One can notice that $x, x^{\prime}\in X_{\gamma}$ being in the same connected component might be an information posterior to the discovery of both $x$ and $x^\prime$. Therefore $x_{1,1}$ and $x_{1,2}$ might be in the same connected component and in that case we have $\vec{x}_{1,2}=(1,2)$. Otherwise we have $\vec{x}_{1,2} = (2,2)$. Finally we set $\bar{\gamma}_{i,t}:= (\vec{x}_{i,t},\bar{y}_{i,t},\vec{x}_{i,t+1})$ and we have that knowing $(\bar{\gamma}_{i,t})_{(i,t) \in \boldsymbol{T}}$ is equivalent to knowing the equivalence class of $\gamma$.\\
	The vertex set of the graph at stake here is $V_{\gamma}:=[s-p]\simeq X_{\gamma}/\sim$ thought as the set of connected components of $K_{\gamma}$ ordered by their order of apparition. For all $i\in [2m]$ and all $t \in [T_i]$ we define the edge $e_{i,t}:=(\bar{\cc}(x_{i,t}),\bar{\cc}(x_{i,t+1}))$ where $\bar{\cc}(x) \in [s-p]$ for $x\in X_{\gamma}$ is the order of apparition of $\cc(x)$ in $\gamma$. We also associate the color $(\bar{x}_{i,t},\bar{y}_{i,t})$ to the edge $e_{i,t}$. Finally we denote by $E_{\gamma}$ the set of colored edges $e_{i,t}$ for all $i\in [2m]$ and $t\in [T_i]$ and $G_{\gamma} = (V_{\gamma},E_{\gamma})$ the corresponding graph.
	\begin{lem}\label{lem:geniuspositive} 
		If $\mathrm{W}_{\ell,\vec{T}}(s,a,p) \neq \emptyset$ then $g:= a+p-s+1 \ge 0$.
	\end{lem}
	\begin{proof}
		In regard of the construction of $G_{\gamma}$ we have $|E_{\gamma}|= |A_{\gamma}|=a$. Besides, owing to \eqref{eq:condilim}, $G_{\gamma}$ is weakly connected and therefore if it is not empty then its genius $g$ is non-negative i.e.
		$$g:=|E_\gamma|-|V_{\gamma}|+1 = a - s+p +1 \ge 0.$$
	\end{proof}
	We give a second condition on the fixed integers of the section for $\mathrm{W}_{\ell,\vec{T}}$ to be not empty. This condition is given by following lemma.
	\begin{lem}\label{lem:conditionm+glep}
		If $\mathrm{W}_{\ell,\vec{T}}(s,a,p)\neq \emptyset$ then $2g + 2m \le p$.
	\end{lem}
	The rough idea behind the proof of the previous lemma can be described as follows. On the one hand we have that $p = \sum_{k} (|\cc_{k}|-1)$ where $\cc_k$ is the $k$-th discovered connected component of $K_{\gamma}$. Therefore it can be seen as the number of vertices that are discovered after we discovered the underlying connected component of $K_{\gamma}$. On the other hand, the number of excess edges $g$ and  $m$ the number of times where we impose the condition \eqref{eq:condilim} will  give control of the number of such vertices. To proceed we define an underlying covering tree $\Gamma \subset G_{\gamma}$ as the result of a chronologically growing forest $(\Gamma_{i,t})_{(i,t) \in \boldsymbol{T}}$ that will allow us to keep track of the times where vertices of $K_{\gamma}$ are discovered.\\
	We consider the sequence $(G_{i,t})_{(i,t)\in \boldsymbol{T}}$ of graphs such that for $(i,t) \in \boldsymbol{T}$, $G_{i,t}$ is the graph spanned by $e_{(j,s)}$ for all $ (j,s) \in \boldsymbol{T}$ such that $(j,s) \preceq (i,t)$. They are non-decreasing sub-graphs of $G_{\gamma}$ and we have $G_{2m,T_{2m}} = G_{\gamma}$. We then inductively define the sequence of non-decreasing forests of $G_{i,t}$ as follows. By convention $\Gamma_{1,0}$ is the graph with the vertex set $\{1\}$ and no edges. Then for a time $(i,t) \in \boldsymbol{T}$ we construct $\Gamma_{i,t}$ from $\Gamma_{(i,t)^{-}}$ as follows. We say that $(i,t)$ is a \textit{first time} if adding $e_{i,t}$ to the graph $\Gamma_{(i,t)^{-}}$ does not create any cycle. In this case $\Gamma_{i,t}$ is defined as the graph spanned by $\Gamma_{(i,t)^{-}}$ and the edge $e_{i,t}$.  If $(i,t)$ is not a first time we set $\Gamma_{i,t}= \Gamma_{(i,t)^{-}}$. We set $\Gamma:= \Gamma_{2m,T_{2m}}$, covering tree of $G_{\gamma}$. We call an \textit{excess edge} all edges $e \in G_{\gamma} \backslash \Gamma$ and we say that $(i,t)$ is an \textit{important time} if $e_{i,t}$ is an excess edge. We now give a sequence of $3$-tuple indexed on $\boldsymbol{T}_{\mathrm{imp}}$ the set of important time such that, knowing this sequence allows to recover $\bar{\gamma}$.\\
	In regard of the relation \eqref{eq:condilim} $\Gamma_{i,t}$ has at most two connected component. More precisely
	\begin{itemize}[label=]
		\item For $i$ odd $\Gamma_{i,t}$ is a tree.
		\item For $i$ even $\Gamma_{i,t}$ has at most two connected component and $\Gamma_{i,T_{i}}$ is connected.
	\end{itemize}
	For $i$ even we call \textit{merging time} the time $(i,t_i)$ at which $\Gamma_{i,t_i}$ becomes connected. Possibly we have $t_i = 1$ and if we have $t_i \ge 2$ then $(i,t_i)$ is also a first time. With these definitions we can now prove the previous lemma.
	\begin{proof}[Proof of \cref{lem:conditionm+glep}]
		For all $1 \le k \le s-p$, the first seen element of the component $\cc_k \in X_{\gamma}/\sim$ is discovered at a first time with a tree edge and the others are either discovered by an excess edge seen for the first time or at a merging time. Each excess edge or tree edge discovered at a merging time can discover two vertices that had not been seen before. Therefore we have $p$ bounded by twice the number of merging times (at most $m$) plus twice the number of excess edges.
	\end{proof}
	Exploring the path $\gamma$ through time $(i,t) \in \boldsymbol{T}$ is a successive repetition of the pattern (1)(2)(3) given by
	\begin{enumerate}[label=(\arabic*)]
		\item (possibly empty) a sequence of first time which are not merging times;
		\item (necessarily) an important time or a merging time;
		\item (possibly empty) a path using the colored edges of the tree defined so far.
	\end{enumerate}
	In (3) we only need the in-coming and out-going vertices to know what happened in the time in between.
	We may now build our first encoding by marking the important and merging times with information that will be sufficient to recover $(\bar{\gamma}_{i,t})_{(i,t) \in \boldsymbol{T}}$ entirely. For $(i,t)$ a time, we assume known $(\bar{\gamma}_{j,s})_{(j,s)\prec (i,t)}$ and we assume that so far we have seen $u$ vertices of $X_{\gamma}$ and $v$ of $Y_{\gamma}$. If $(i,t)$ is a first time and not a merging time then we can recover $\bar{\gamma}_{(i,t)}$. Indeed necessarily
	\begin{itemize}
		\item if $t\ge 2$ then we have $\vec{x}_{i,t} = \vec{x}_{i,(t-1)+1}$, $\vec{x}_{i,t+1}=(u+1,u+1)$ and $\bar{y}_{i,t}= v+1$,
		\item if $t=1$ and $i$ odd then we still have $\vec{x}_{i,t+1}= \vec{x}_{i,2}= (u+1,u+1)$ and $\bar{y}_{i,t}= \bar{y}_{i,2}=v+1$ but now we have $\vec{x}_{i,1}=\vec{x}_{i-1,1}$,
		\item finally if $t=1$ and $i$ even, $(i,1)$ not being a merging time implies here that both connected component seen at the time $(i,1)$ are new and therefore we have $\vec{x}_{i,1}=(u+1,u+1)$, $\bar{y}_{i,1}=v+1$ and $\bar{x}_{i,2}= (u+2,u+2)$. 
	\end{itemize}
	Notice that at a first time that is not a merging time, the $x \in X_{\gamma}$ discovered gives the color of its connected component. If $(i,t)$ is an important time, we first mark it with the vector $(\bar{y}_{i,t},\vec{x}_{i,t+1},\vec{x}_{i,\tau})$ where $(i,\tau)$ is the next time we step outside of the tree $\Gamma_{i,t} = \Gamma_{(i,t)^{-}}$. If for the remaining time of the era $i$ the path $\gamma$ stays in $\Gamma_{i,t}$ then we set $\tau = T_{i}+1$. If we have $\tau \neq T_i+1$ then $(i,\tau)$ is the next important, merging or first time. If $(i,t)$ is a merging time we also mark it $(\bar{y}_{i,t},\vec{x}_{i,t+1},\vec{x}_{i,\tau})$ where $(i,\tau)$ is once again the next important, merging or first time. For these two marked times, in the time in between $(i,t)$ and $(i,\tau)$ we only go through seen vertices of $K_{\gamma}$ and seen vertices of $\Gamma_{i,t}$ and therefore one can recover $\{\bar{\gamma}_{i,s}:~t \le s \le \tau\}$. Finally for $i\in [2m]$ odd, we call \textit{first outgoing time} $(i,t)$ the first time occurring that is either a first or important time and mark it with $(\vec{x}_{1,1},\bar{y}_{i,t},\vec{x}_{i,t+1},\vec{x}_{i,\tau})$ where $(i,\tau)$ is the next important time. If a first outgoing time is also a first time, in the time in between $(i,t)$ and $(i,\tau)$ we only see first times. If for $t\le s \le T_i$ we have only first times we set $\tau= T_i+1$. Therefore knowing the first outgoing, important and merging times and their marks we can recover all $\bar{\gamma}$. 
	\begin{rem}
		We can describe the recovering of $\bar{\gamma}$ as follows. One has to segment the recovering with each eras $i$. If the time $(1,1)$ is not marked then it a first time, therefore we can set that $\bar{x}_{1,1}=1$, $\bar{c}(1)=\bar{c}(\bar{x}_{1,1})=1$, $\bar{y}_{1,1}= 1$, $\bar{x}_{1,2}=2$ and finally $\bar{c}(2)=\bar{c}(\bar{x}_{1,2})=2$. Therefore we know that $\bar{\gamma}_{1,1}=((1,1),1,(2,2))$. If now $(1,1)$ is an important time, we still have $\bar{x}_{1,1}=1$, $\bar{c}(1)=1$ and $\bar{y}_{1,1}=1$ but then we have $\bar{c}(2)=1$ and we know the value of $\bar{x}_{1,2}$ from the mark on $(1,1)$. Since $\Gamma_{1,1}=\Gamma_{1,0}$ has no edge, necessarily the next time we step outside the tree is the next time, i.e the time $(1,2)$ is the next first or important time. We start by explaining, for an era $j = 2i$ how to recover $\bar{\gamma}_{j,s}$ for $1 \le s \le t$ where $(j,t)$ is the first marked time but not first outgoing time of the era $j$ considering that $\bar{\gamma}_{j^{\prime},s^\prime}$ for $ (1,1) \prec (j^\prime,s^\prime) \prec (j,1)$ are known. Let $(2i,t)$ be the first marked time and let $(\bar{y}_{2i,t}, \vec{x}_{2i,t+1},\vec{x}_{2i,\tau})$ be its mark. If $\bar{y}_{2i,t}$ (or $\bar{c}(\bar{x}_{2i,t+1})$) has never been seen during the preceding eras then it means $(2i,t)$ is an important time and in the time in between $(2i,1)$ and $(2i,t)^-$ we only have first times. We can recover then $\bar{\gamma}_{2i,s}$ for $1 \le s \le t$. If now $\bar{y}_{2i,t}$ or $\bar{c}(\bar{x}_{2i,t+1})$ has been seen before, then $(2i,t)$ is the merging time and once again we know that we only have first times between $(2i,1)$ and $(2i,t)^-$. Now we consider an era $j=2i+1$, $(2i+1, t)$ its first outgoing time and $(\vec{x}_{2i+1,1}, \bar{y}_{2i+1,t},\vec{x}_{2i+1,t},\vec{x}_{2i+1,\tau})$ its mark. Then there exists only one path between $\cc(x_{2i+1,1})$ and $\cc(x_{2i+1,\tau})$ and even though $\tau$ is not indexed in the mark it is given by the length of this path. Now if there exists another marked time $(2i+1,t^\prime)$ in the same era $2i+1$, it is necessarily an important time. If we assume $\tau \neq t^\prime$, we have $\bar{\gamma}_{2i+1,s}$ for $1 \le s \le \tau$ given by the only path between $\cc(x_{2i+1,1})$ and $\cc(x_{2i+1,\tau})$. We then recover $\bar{\gamma}_{2i+1,s}$ for $\tau+1 \le s \le t-1$ from the fact that $\{(2i+1,s):~ \tau+1 \le s \le t^\prime-1\}$ are first times. If $\tau = t^\prime$ then $\bar{\gamma}_{2i+1,s}$ for $1 \le s \le \tau = t^\prime$ is given by the only path between $\cc(x_{2i+1,1})$ and $\cc(x_{2i+1,t^\prime})$. If there is no other marked time in the era $2i$ then we only have first times (and $\tau = T_{2i+1}+1$).
		We now consider $(i,t)\prec (i,t^\prime)$ two consecutive marked times of a same era $i$ and we suppose $(\bar{\gamma}_{j,s})_{(j,s)\prec (i,t)}$ is known. Let $(\bar{y}_{i,t}, \vec{x}_{i,t+1}, \vec{x}_{i,\tau})$ be the mark on $(i,t)$ and $(\bar{y}_{i,t^\prime}, \vec{x}_{i,t^\prime+1}, \vec{x}_{i,\tau^\prime})$ the one on $(i,t^\prime)$. We have only one path (possibly empty) possible between $\cc(x_{i,t+1})$ and $\cc(x_{i,\tau})$. The length of this path and the path itself gives us the value of $\tau$ and $\bar{\gamma}_{i,s}$ for $t \le s \le \tau$. If $\tau = t^\prime$ then we have recovered $\bar{\gamma}_{i,s}$ for $ t\le s \le t^\prime$. Otherwise all $(i,s)$ for $ \tau \le s \le t^\prime$ are necessarily first times.
	\end{rem}
	With the previous encoding, for each important time or merging time we have at most $a(s-p)^2s^{2}$ ways to mark it. We have at most $m$ merging times. Finally we have at most $m$ first outgoing time that have each $a(s-p)^3s^{3}$ ways to be marked. We have at most $2m\ell$ ways to position the important times. However with no further details on our path, upper bounding roughly the number of important times $|\boldsymbol{T}_{\mathrm{imp}}|$ by $2m\ell - s+p$ is too big to be implemented in \eqref{eq:htmethod}. Therefore we are going to take into account the tangle-freeness of our paths $\gamma_i$ to make some important times superfluous. Meaning that there will be more specific important times that -with the good modified encoding, will allow us to not take into account the other ones.\\
	To proceed we consider each era $i$ separately. We use the fact that $\gamma_i$ being tangle-free, if we explore a cycle during the era $i$ we know it to be the only one. For $i\in [2m]$ let $(i,t_0)$ be the smallest time such that $\cc(x_{i,t_0+1}) \in \{\cc(x_{i,1}),\cc(x_{i,2}),...,\cc(x_{i,t_0})\}$, if it exists. Let $1 \le \sigma\le t_0$ be such that $\cc(x_{i,t_0+1})=\cc(x_{i,\sigma})$. We denote by $\mathrm{C}_{i}:= (\bar{\cc}(x_{i,\sigma}),\bar{\cc}(x_{i,\sigma+1}),..., \bar{\cc}(x_{i,t_0+1}))$ and $\mathrm{C}_i$ is the only cycle visited by $\gamma_i$. We define the \textit{short cycling time} as the greatest important time $(i,t)$ smaller then $(i,t_{0})$ in $(\boldsymbol{T},\prec)$ and $(i,\hat{t})$ the smallest time $(i,\hat{t}) \succeq (i,\sigma)$ verifying that $\cc(x_{i,\hat{t}+1}) \notin \mathrm{C}_i$. As convention we take $\hat{t}=T_i +1$ if $\gamma_i$ stays in $\mathrm{C}_i$ for the remaining time of the era $i$. If $\hat{t} \ge t_{0}+2$ then the cycle $\mathrm{C}_i$ is visited more then once, even though it may not finish the cycle after the first visit. We replace the mark on $(i,t)$ with $(\bar{y}_{i,t},\vec{x}_{i,t+1},\sigma,\hat{t},\vec{x}_{i,\tau})$ where now $(i,\tau)$ is the next important time after $(i,\hat{t})$. We take as convention that $\tau = T_{i}+1$ if $\gamma_i$ stays in $\Gamma_{(i,t)}$ for the remaining time of the era $i$. During the era $i$ we have only one short cycling time. The important times $(i,t^{\prime})$ that verify $1\le t^{\prime} \le t$ or $ \tau \le t^{\prime} \le T_{i}$ are called \textit{long cycling time}. The important times in the time in between $(i,t)$ and $(i,\tau)$ are called \textit{superfluous} and are not marked anymore. We can now prove \cref{boundingWlTsap}.
	\begin{proof}[Proof of \cref{boundingWlTsap}]
		The previous encoding also allows us to recover $\bar{\gamma}$ from the marks on the long cycling and short cycling times (see proof of \cite[Lemma 3]{bord2018spectral}). If an excess edge is seen twice in the era $i$ outside of the time interval $(i,t)^{+}$ and $(i,\tau)$ it would mean that we discover another cycle in $\gamma_i$, which is impossible. Also the excess edge seen at the short cycling time $(i,t)$ cannot be seen by a long cycling time either, for the same reason. We have the following for each kind of marked times. We have at most $g$ long cycling times for each era $i\in [2m]$. For each long cycling time we have at most $a(s-p)^2 s^2$ ways to mark it. We also have possibly $m$ merging times, and again at most $a(s-p)^2 s^2$ to mark one of them. Similarly we have at most $m$ first outgoing time. For each first outgoing time we have at most $a(s-p)^3 s^3$ ways to mark it. Finally at each era we have at most one short cycling time with at most $a(s-p)^2s^2 \ell^{2}$ ways to mark it. In each era, for any short cycling, long cycling or first outgoing time, we have at most $T_i \le \ell$ ways of positioning the time in the era. Therefore we  have
		\begin{align*}	|\mathcal{W}_{\ell,\vec{T}}(s,a,p)|&\le [\ell a(s-p)^2 s^2 ]^{2mg} [\ell a (s-p)^2 s^2 \ell^2]^{2m} [\ell a (s-p)^2s^2]^{m}[\ell a(s-p)^3 s^3]^m\\
			&\le [\ell a (s-p)^2s^2]^{2m(g+5)}.
		\end{align*}
		
	\end{proof}
	
	\subsection{Contributions of the matrix $Q$ for a fixed path and proof of \cref{boundingnormP}}\label{section:contribQ}
	In this section we fix the integers $\ell, m \ge 1$ such that 
	\begin{equation}\label{eq:fix_of_lm}
		1 \le \ell \le  \frac{\ln(N)}{6(\ln(d)\vee 2\ln(\delta\vee 1))} \text{ and } 1 \le m \le \ln(N).
	\end{equation}
	We also fix the integers $0 \le s,a,p \le 2m\ell$ and consider $[\gamma] \in \mathcal{W}_{\ell,\vec{T}}(s,a,p)$ for some $\vec{T} \in [\ell]^{2m}$ and we denote by $a_1$ the number of inconsistent edges $e \in \mathrm{A}_{\gamma}$ of multiplicity $1$ (see \cref{estimconsistent}).
	
	\begin{lem}\label{boundingsumq(gamma)} 
		There exists $c,c(n)\ge1$ such that
		$$ \sum_{\gamma^{\prime} \sim \gamma} q_{\xxx} (\gamma^{\prime})  \le c(n)N^{s-p}(cm\ell)^{a_1 + 3(m+g) +p} \rho_{n,\xxx}^{2m\ell}. $$
	\end{lem}
	To prove the previous lemma we need the following definition and lemmas.
	\begin{defi}\label{def:graphbipart}
		Let $\mathcal{G}=(\mathcal{V},\mathcal{E})$ be a directed graph with $\mathcal{E}$ being a multi-set, and connected if we do not take into account the orientation (i.e. \textup{weakly connected}). We also consider that $\mathcal{G}$ is bi-partite in the following sense. There exists $\mathcal{X}\bigsqcup\mathcal{Y}$ partition of $\mathcal{V}$ such that for all $e=(u,v) \in \mathcal{E}$, we have $u\in \mathcal{X}$ and $v\in \mathcal{Y}$. Let $\mathcal{T} \subset \mathcal{G}$ be (any) underlying covering tree of $\mathcal{G}$ and $\mathcal{R}=\mathcal{G}\backslash\mathcal{T}$. 
		\begin{itemize}
			\item We say that $(e, e^\prime)\in \mathcal{E}^2$ are a \textup{double edge} if both $e$ and $e^\prime$ goes from $u \in \mathcal{X}$ to $v\in \mathcal{Y}$ and $e \in \mathcal{T}$.
			\item For $e, e^\prime \in \mathcal{R}$ we denote $e\sim e^{\prime}$ when  both $e$ and $e^\prime$ goes from $u\in \mathcal{X}$ to $v \in \mathcal{Y}$.
		\end{itemize}
	\end{defi}
	\begin{lem}\label{lem:boundingpdtZ(G)}
		Let $\mathcal{G}= (\mathcal{X}\sqcup\mathcal{Y},\mathcal{E})$ be a bi-partite connected  and directed graph in the sense of \cref{def:graphbipart} and $\mathcal{T}$ an underlying tree. If there is a double edge in $\mathcal G$ we denote it $(e,e^\prime) \in \mathcal{T} \times (\mathcal{E}\backslash \mathcal{T})$. Let $(\ell(f))_{f \in \mathcal{E}}$ be a family of integers. We define 
		
		$$Z(\mathcal{G}):= \sum_{\substack{y_{u} \in [N]\\ u \in \mathcal{\mathcal{Y}}}}\sum_{\substack{x_{v} \in [N]\\ v \in \mathcal{\mathcal{X}}}} \prod_{\substack{f \in \mathcal{E} \\ f=(u,v)}} (Q^{\ell(f)})_{y_u x_v}~~ \text{and} ~~L_{\mathcal T} := \sum_{f \in \mathcal{T} \backslash\{e\}} \ell (f),$$
		where $\{e\} = \emptyset$ if there is no double edge. If we have $\mathcal{G} = \mathcal{T} \sqcup \{e^\prime\}$ and $(e,e^\prime)$ a double edge, then we have
		$$Z(\mathcal{G}) \le N \Tr_{N} \left(Q^{\ell(e)}{Q^{\ast}}^{\ell(e^\prime)} \right) \delta^{L_{\mathcal T}}.$$
		If $\mathcal{G} = \mathcal{T}$ we have
		$$Z(\mathcal T) \le N \delta^{L_{\mathcal T}}.$$
	\end{lem}
	
	\begin{proof}[Proof of \cref{lem:boundingpdtZ(G)}]
		To ease notations we denote by $\mathcal{T}$ the underlying tree but also the set of edges of this same underlying tree. We proceed by induction on $|\mathcal{T}| = |\mathcal X|+ |\mathcal Y| -1$, where the equality is due to the the fact that $\mathcal{T}$ is a covering tree on a set of cardinal $|\mathcal X|+ |\mathcal Y|$. In the case where $|\mathcal{T}| =1$ and $\mathcal{G} $ has a double edge we have
		\begin{align*}
			Z(\mathcal{G}) &=  \sum_{x,y} (Q^\ell(e))_{yx}(Q^{\ell(e^\prime)})_{yx} = N \Tr_{N} \left( Q^{\ell(e)} {Q^\ast}^{\ell(e^\prime)} \right).
		\end{align*}
		In the case where there is no double edge and $|\mathcal{T}| =1 $, then necessarily we have only one edge and therefore $$Z(\mathcal{G}) =  \sum_{x,y} (Q^{\ell(f)})_{yx} \le N \delta^{\ell(f)}.$$
		We now assume $|\mathcal{T}| \ge 2$. We consider $\{e, e^\prime\}$ the possibly empty set of double edge and $e_0=(u_0,v_0)$ a leaf of $\mathcal{T}$ that is not $e$. Such a leaf exists since for any tree with more than two edges there exists at least two different leaves. Since we are considering an oriented graph that is weakly connected, denoting $\mathrm{d}_{\mathcal{T}}^{\mathrm{o}}(\cdot)$ the weak degree in $\mathcal{T}$, one needs to distinguish whether we have $\mathrm{d}_{\mathcal{T}}^{\mathrm{o}}(u_{0})=1$ or $\mathrm{d}_{\mathcal{T}}^{\mathrm{o}}(v_{0})=1$. We consider the first case and we have
		\begin{align*}
			Z(\mathcal{G}) &=  \sum_{\substack{(y_{u},x_{v})  \in [N]^2\\ (u, v) \in  \mathcal{Y}\backslash \{u_0\}\times \mathcal{X}} } \prod_{f \in \mathcal{T} \backslash \{e_{0} \}} (Q^{\ell(f)})_{y_u x_v} \sum_{y_{u_0}  \in [N]} (Q^{\ell (e_{0})})_{y_{u_0}x_{v_0}}.
		\end{align*}
		We then bound the sum over $y_{u_{0}} \in [N]$ by $\delta^{\ell(e_{0})}$ and apply the induction to $\mathcal G^\prime = (\mathcal X \sqcup (\mathcal Y \backslash \{u_0\}), \mathcal E \backslash \{e_{0} \})$. In the case where we have $\mathrm{d}^{\mathrm{o}}_{\mathcal{T}}(v_{0}) =1$ we similarly obtain
		\begin{align*}
			Z(\mathcal{G}) &=  \sum_{\substack{(y_{u},x_{v})  \in [N]^2\\ (u, v) \in  \mathcal{Y}\times \mathcal{X}\backslash \{v_0\}}} \prod_{f \in \mathcal{T} \backslash \{e_{0} \}} (Q^{\ell(f)})_{y_u x_v} \sum_{x_{v_0}  \in [N]} (Q^{\ell (e_{0})})_{y_{u_0}x_{v_0}}.
		\end{align*}
		We then bound the sum over $x_{v_{0}} \in [N]$ by $\delta^{\ell(e_{0})}$ and apply the induction to $\mathcal G^\prime = ((\mathcal X\backslash \{v_0\} )\sqcup \mathcal Y , \mathcal E \backslash \{e_{0} \})$. 
	\end{proof}
	For $\gamma \in \mathrm{W}_{\ell,\vec{T}}(s,a,p)$, we rewrite $q(\gamma)$ as follows
	$$q_{\sss}(\gamma)= \prod_{i=1}^{2m}\prod_{t=1}^{T_i} (Q^{\ell_{i,t}})_{y_{i,t},x_{i,t+1}}=\prod_{\cc \in \chi} \prod_{\substack{f\in \mathcal{E}_{\cc}\\ f=(u,v)}} (Q^{\ell(f)})_{x_u,x_v}$$
	where $\chi := X_{\gamma}/\sim$ are the connected components of $K_{\gamma}$ and for $\cc \in \chi$ we considered the graph $\mathcal{G}_{\cc}:=(\mathcal{V}_{\cc},\mathcal{E}_{\cc})$ defined by $\mathcal{V}_{\cc}:= \{\bar{x}; ~ x\in \cc\} \sqcup \{\bar{y};~ \exists (i,t) \in \boldsymbol{T},~x \in \cc;~(y_{i,t},x_{i,t+1})= (y,x)\}=: \mathcal{X}_{\cc} \sqcup \mathcal{Y}_{\cc}$ and
	$\mathcal{E}_{\cc}$ is the multi-set defined by $\mathcal{E}_{\cc}:= \{(u,v)\in \mathcal{V}_{\cc}^2;~~ \exists (i,t) \in \boldsymbol{T};~~ (\bar{y}_{i,t},\bar{x}_{i,t+1})= (u,v)\}$. The notations are the one introduced in the proof of \cref{boundingWlTsap} and the graphs $(\mathcal{G}_{\cc})_{\cc \in \chi}$ are directed, weakly connected and bi-partite in the sense of \cref{def:graphbipart}. In the writing of $q(\gamma)$ above we have considered that for $u= \bar{x}$ we set $x_{u}= x$. We consider $\mathcal{T}_{\cc}$ the covering tree obtained by constructing the tree while we explore $\gamma$ chronologically and we set $\mathcal{R}_{\cc} := \mathcal{G}_{\cc} \backslash \mathcal{T}_{\cc}$. We introduce $\chi_{\mathrm{d}} \subset \chi$ the set of connected component $\cc \in \chi$ such that there exists $(e(\mathrm{cc}),e^\prime(\mathrm{cc}))$ a double edge in $\mathcal{G}_{\cc}$. We consider the graphs
	$$\mathcal{G} = (\mathcal{X} \sqcup \mathcal{Y}, \mathcal{E}) :=\bigsqcup_{\cc\in \chi}\mathcal{G}_{\cc},~~ \mathcal{F}= \bigsqcup_{\cc\in \chi} \mathcal{T}_{\cc}  \backslash\{e(\cc)\} ~~ \mathcal{R}:= \bigsqcup_{\cc\in \chi} \mathcal{R}_{\cc} \backslash \{e^\prime(\cc)\}, ~\text{and}~\mathcal{D} := \bigsqcup_{\cc\in \chi} \{e(\cc),e^{\prime}(\cc)\}$$
	where for $\cc\notin \chi_{\mathrm{d}}$ we have $\{e(\cc)\}=\{e^\prime(\cc)\}= \{\emptyset\}$. 	We set 
	\begin{equation}\label{eq:defiF}
		F:= |\mathcal{F}| ~\text{and}~ R:= |\mathcal{R}|.
	\end{equation}
	In order to prove \cref{boundingmu(gamma)}, we provide an upper bound for $F$.
	\begin{lem}\label{lem:upperboundF}
		Let $m,\ell \ge 1$, $1 \le s,a,p \le 2m\ell$ and $\vec{T}\in [\ell]^{2m}$ be fixed. For all $\gamma \in W_{\ell,\vec{T}}(s,a,p)$, considering the corresponding $F$ given by \eqref{eq:defiF} we have
		\begin{equation*}
			F \le 3(m+g) -2 +a_1 + p.
		\end{equation*}
	\end{lem}
	We use the previous lemmas to prove \cref{boundingsumq(gamma)}.
	\begin{proof}[Proof of \cref{boundingsumq(gamma)}]

		We give the detail of the upper bound for $q_{\sss}$ defined by \eqref{eq:defiqashtag} and we precise along the proof where the reasoning needs to be adapted for $q_{\ccc}$. \\
         For $\gamma^\prime=(x^\prime_{1,1},y^{\prime}_{1,1},\ell_{1,1}^{\prime},x^{\prime}_{1,2},..., x^{\prime}_{2m, T_{2m}+1}) \sim \gamma$ and considering the weighted graph defined above we set 
        \begin{equation*}
        L_{\mathcal R}(\gamma^\prime) = \sum_{f \in \mathcal R} \ell^{\prime}(f).
        \end{equation*}

        We first bound our sum as follows
        \begin{align*}
            \sum_{\gamma^\prime \sim \gamma} q_{\sss} \le \sum_{\gamma^\prime \sim \gamma} \kappa^{L_{\mathcal R}(\gamma^{\prime})} \prod_{\cc \in \chi } \prod_{f \in \mathcal{T}_{\cc} \sqcup \{e^{\prime}(\cc)\}} (Q^{\ell^{\prime}(f)})_{x^{\prime}_u x^\prime_{v}}.
        \end{align*}
        For a fixed vector $\vec{\ell} = (\ell_{11},..,\ell_{it},...,\ell_{2m,T_{2m}})$ we have 
\begin{align}
    \sum_{\substack{\gamma^{\prime} \sim \gamma\\ \vec{\ell}(\gamma^\prime) = \vec{\ell} } } q_{\sss}(\gamma^\prime) &\le \kappa^{L_{\mathcal R}(\gamma^{\prime}) }\sum_{\substack{y_{u},x_{v} \in[N]\\ (u,v) \in \mathcal{Y} \times \mathcal X}} \prod_{f \in \mathcal{T}_{\cc}\sqcup \{e^\prime(\cc)\}} (Q^{\ell (f)})_{x_u x_v}\notag \\
    &\le N^{s-p} \delta^{L_{\mathcal F}} \kappa^{L_{\mathcal R}}\prod_{\cc \in \chi_{\mathrm{d}}} \Tr_N \left( Q^{\ell(e(\cc))}{Q^{\ast}}^{\ell(e^\prime(\cc))}\right)\label{eq:improv}
\end{align}

where we applied \cref{lem:boundingpdtZ(G)}. For a finite set $\mathcal S$ and an integer $I$ we introduce
		\begin{equation*}
			\mathcal{L}^\prime (\mathcal S, I) := \left\{ (\ell(s))_{s\in \mathcal S} \in [\ell]^{\mathcal S}:~ \sum_{s \in \mathcal S} \ell(s) =I\right\}.
		\end{equation*}
        We set  $T:= \sum_{i=1}^{2m} T_i$ and $M = 2m (\ell + 1) -T $. We bound $|\mathcal{L} (L_{\mathcal F} , \mathcal F) |$ by $\ell^{F}$ and
        summing over all vector $\vec{\ell}$ possible we obtain
        \begin{align}\label{eq:before_last_ineq}
            \sum_{\gamma^{\prime} \sim \gamma } q_{\sss} \le N^{s-p}\ell^F \kappa  \sum_{L_{\mathrm{d}}+ L_{\mathcal F} + L_{\mathcal R}=M} (2\kappa)^{L_{\mathcal R} + R -1} \delta^{L_{\mathcal F}}\sum_{\substack{L^{1}_\mathrm{d}+L^{2}_\mathrm{d}=L_{\mathrm d} \\ (\ell_{\cc}^i)_{\cc \in \chi_{\mathrm d}} \in \mathcal L(L_{\mathrm d}^i, \chi_{\mathrm d})}} \prod_{\cc \in \chi_{\mathrm d}} \Tr_N ( Q^{\ell^1(\cc)}{Q^{\ast}}^{\ell^2(\cc)} )
        \end{align}
where we used 
\begin{align}\label{eq:appear2kappa}
    \sum_{(\ell(f))_{f \in \mathcal R}\in \mathcal L (L_{\mathcal R}, \mathcal R)} \kappa^{L_{\mathcal R}} = \binom{R-1}{L_{\mathcal R} + R- 1} \kappa^{L_{\mathcal R}} \le (2\kappa)^{L_{\mathcal R} + R -1} \kappa. 
\end{align}
		For $0 \le L_{\mathrm{d}}^{1} +   L_{\mathrm{d}}^2 = L_{\mathrm{d}}$ fixed, denoting $\mathrm{x}_{\mathrm d} = | \chi_{\mathrm d}|$ we have
		\begin{align}
			\sum_{ (\ell^i_{\cc})_{\cc} \in \mathcal{L}^\prime( \chi_{\mathrm d}, L_{\mathrm{d}}^i) } \prod_{\cc \in \chi_\mathrm{d}} \Tr_N\left( Q^{\ell^1_{\cc}} {Q^\ast}^{\ell^2_{\cc}}\right) &\le  \sum_{ (\ell^i_{\cc})_{\cc} \in \mathcal{L}( \chi_{\mathrm d}, L_{\mathrm{d}}^i+\mathrm{x}_{\mathrm{d}}) } \prod_{\cc \in \chi_\mathrm{d}} \Tr_N\left( Q^{\ell^1_{\cc}} {Q^\ast}^{\ell^2_{\cc}}\right) \label{eq:sum_tracesQ}\\
			&\le \sum_{T=0}^{L_{\mathrm{d}}^1 \land L_{\mathrm{d}}^2 + \mathrm{x}_{\mathrm{d}}}\sum_{ (\ell^i_{t})_{t} \in \mathcal{L}( T, L_{\mathrm{d}}^i+\mathrm{x}_{\mathrm{d}}) } \prod_{t=1}^T \Tr_N\left( Q^{\ell^1_{t}} {Q^\ast}^{\ell_{t}^2}\right). \notag
		\end{align}
		Recall the definition of $\mathrm{b}_{\xxx}$ given by \eqref{eq:defaandp}. We define the rest
		\begin{align*}
			\mathrm{R}_{\sss}:=\sum_{T=0}^{L_{\mathrm{d}}^1 \land L_{\mathrm{d}}^2 + \mathrm{x}_{\mathrm{d}}}\sum_{ (\ell^i_{t})_{t} \in \mathcal{L}( T, L_{\mathrm{d}}^i+\mathrm{x}_{\mathrm{d}}) } \prod_{t=1}^T \Tr_N\left( Q^{\ell^1_{t}} {Q^\ast}^{\ell_{t}^2}\right) -  \Phi({\mathrm{b}}_{\sss}^{L^1_{\mathrm{d}}+\mathrm{x}_{\mathrm d}} {{\mathrm{b}}^\ast_{\sss}}^{L^2_{\mathrm{d}}+\mathrm{x}_{\mathrm d}}) 
		\end{align*}
		and we have
		\begin{align*}
			\mathrm{R}_{\sss} &= \sum_{T=0}^{L_{\mathrm{d}}^1 \land L_{\mathrm{d}}^2 + \mathrm{x}_{\mathrm{d}}}\sum_{ (\ell^i_{t})_{t} \in \mathcal{L}( T, L_{\mathrm{d}}^i+\mathrm{x}_{\mathrm{d}}) } \sum_{k=1}^T \sum_{\substack{V \subset T\\ \abs{V} =k}} \frac{\delta^{\sum_{t\in V} \ell^1_{t} + \ell^2_t}}{N^k}\prod_{t \notin V} \Tr_N\left( Q_{\perp}^{\ell^1_t} {Q_{\perp}^\ast}^{\ell_t^2}\right)\\
			&=\sum_{k=1}^{L_{\mathrm{d}}^1 \land L_{\mathrm{d}}^2 + \mathrm{x}_{\mathrm{d}}} \sum_{T=k}^{L_{\mathrm{d}}^1 \land L_{\mathrm{d}}^2 + \mathrm{x}_{\mathrm{d}}}\sum_{L^i_{\mathrm d, k} \le L^i_{\mathrm d} }\binom{T}{k}|\mathcal{L}^\prime(L^1_{\mathrm{d},k},k)||\mathcal{L}^\prime(L^2_{\mathrm{d},k},k)|\frac{\delta^{2\ell k}}{N^k}\sum_{\ast } \prod_{t=k+1}^T \Tr_N\left( Q_{\perp}^{\ell^1_t} {Q_{\perp}^\ast}^{\ell_t^2}\right)
		\end{align*}
		where the last summand is over all $(\ell_{t}^i)_t \in \mathcal{L}(T-k, L^i_{\mathrm d} +\mathrm{x}_{\mathrm{d}} - L_{\mathrm d, k}^i-k)$ for $i\in \{ 1,2\}$. The last sum above verifies 
		\begin{align*}
			\abs{\sum_{\ast } \prod_{t=k+1}^T \Tr_N\left( Q_{\perp}^{\ell^1_t} {Q_{\perp}^\ast}^{\ell_t^2}\right)} &\le \sum_{\ast } \prod_{t=k+1}^T \abs{\Tr_N\left( Q_{\perp}^{\ell^1_t} {Q_{\perp}^\ast}^{\ell_t^1}\right)\Tr_N\left( Q_{\perp}^{\ell^2_t} {Q_{\perp}^\ast}^{\ell_t^2}\right)}^{1/2}\\
			&\le \sum_{\ast } \prod_{t=k+1}^T \| Q_{\perp}^{\ell^1_t} \|\| Q_{\perp}^{\ell^2_t}\|\le (2 \|Q_{\perp}\|)^{L_{\mathrm{d}}^1+L_{\mathrm{d}}^2+ 2\mathrm{x}_{\mathrm{d}}-L_{\mathrm{d},k}^1-L_{\mathrm{d},k}^2}
		\end{align*}
		where for the last two inequalities we applied \cref{formuletrace} and used that for $ b,h \in \mathcal{A}$
		$$\Phi(bh^{*}) \le \sqrt{\tau(bb^{*}) \tau(hh^{*})}, ~~ \text{and} ~~ \tau(b^{\ast}b)= \mathbb{E}[\|b^{\ast}b(e_{0})\|_2^2]\le\|b\|^2.$$
		The integer $n$ being fixed, for all integer $L=\mathrm{q} n + \mathrm{r} \ge 1$ with $0 \le \mathrm{r} \le n-1$, We set $c(n) := (\kappa\vee 1) (1+\delta^{\prime})^{n}$ and we have
		\begin{equation}\label{eq:use_sub_norm}
			{\mathrm{N}}_{L,\sss}^L= \norme{{\mathrm{b}}_{\sss}^L} \le \norme{{\mathrm{b}}_{\sss}^{n}}^{\mathrm{q}} \norme{{\mathrm{b}}_{\sss}}^{\mathrm{r}} \le {\mathrm{N}}_{n,\sss}^{L-\mathrm{r}}(1+ |\lambda_2(Q)|)^{n}\le c(n)  {\mathrm{N}}_{n,\sss}^L 
		\end{equation}
		where we used that the operator norm is sub-multiplicative.
		Besides we have $| \mathcal{L}^{\prime}(L,k)| \le \ell^{k}$ and therefore we obtain 
		\begin{align*}
			\mathrm{R}_{\sss} &\le 2 m \ell c(n)^2\sum_{k=1 }^{L_{\mathrm{d}}^1 \land L_{\mathrm{d}}^2 + \mathrm{x}_{\mathrm{d}}} \sum_{T=k}^{L_{\mathrm{d}}^1 \land L_{\mathrm{d}}^2 + \mathrm{x}_{\mathrm{d}}} \binom{T}{k}\left( \frac{\ell^2\delta^{2\ell}}{N} \right)^k \sum_{L_{\mathrm d,k} \le L_{\mathrm d}} {{\mathrm{N}}_{n,\sss}}^{L_{\mathrm d} +2 \mathrm{x}_{\mathrm{d},\sss} - L_{\mathrm{d},k}}\\
			&\le c(2m\ell)^5c(n)^{2} \frac{\rho_{n,\sss}^{L_{\mathrm d}}}{N}.
		\end{align*}
       We used the condition \eqref{eq:fix_of_lm} for the second inequality. Implementing the previous inequality in \eqref{eq:sum_tracesQ} we obtain 
		\begin{align}
			\sum_{ (\ell^i_{\cc})_{\cc} \in \mathcal{L}^\prime( \chi_{\mathrm d}, L_{\mathrm{d}}^i) } \prod_{\cc \in \chi_\mathrm{d}} \Tr_N\left( Q^{\ell^1_{\cc}} {Q^\ast}^{\ell^2_{\cc}}\right) &\le |\Phi({\mathrm{b}}_{\sss}^{L^1_{\mathrm{d}}+\mathrm{x}_{\mathrm d}} {{\mathrm{b}}^\ast_{\sss}}^{L^2_{\mathrm{d}}+\mathrm{x}_{\mathrm d}})| + \abs{ \mathrm{R}_{\sss}} \notag\\
			& \le c(n) (2m\ell)^{5}\rho_{n,\sss}^{L_{\mathrm d}} \label{eq:bound_sum_traceLid}
		\end{align}
		where we proceed as we did in \eqref{eq:use_sub_norm} and obtained $|\Phi({\mathrm{b}}_{\sss}^{L^1_{\mathrm{d}}+\mathrm{x}_{\mathrm d}} {{\mathrm{b}}^\ast_{\sss}}^{L^2_{\mathrm{d}}+\mathrm{x}_{\mathrm d}})|  \le c(n)(\mathrm{N}_{n,\sss}\vee 1)^{L_{\mathrm{d}}}$. 
		For the convex combination case we introduce 
		\begin{align*}
			\mathrm {S}_{\ccc}:=(1-r)^{2(\mathrm{x}_{\mathrm{d}}-1)}\sum_{(\ell_{\cc},\ell^\prime_{\cc})_{\cc} } \prod_{\cc \in \chi_\mathrm{d}} \Tr_N\left(\{r Q\}^{\ell_{\cc}} \{rQ^\ast\}^{\ell_{\cc}^\prime}\right) 
			\end{align*}
			and we have
			\begin{align*}
			\mathrm {S}_{\ccc} &\le \sum_{T=0}^{L_{\mathrm{d}}^1 \land L_{\mathrm{d}}^2 + \mathrm{x}_{\mathrm{d}}}\sum_{ (\ell^i_{\cc})_{\cc} \in \mathcal{L}^\prime( \chi_{\mathrm d}, L_{\mathrm{d}}^i) } (1-r)^{2(T-1)}\prod_{t=1}^T \Tr_N\left( \{rQ\}^{\ell_{t}^1} \{rQ^\ast\}^{\ell_{t}^2}\right)\\
			& \le |\Phi( {\mathrm{b}}_{\ccc}^{L_{\mathrm d}^1 + \mathrm{x}_{\mathrm{d}}} {{\mathrm{b}}_{\ccc}^{\ast}}^{L_{\mathrm d}^2+ \mathrm{x}_{\mathrm{d}}})| + |\mathrm{R}_{\ccc}| \le  c(n) (2m\ell)^{5}\rho_{n,\ccc}^{L_{\mathrm d}} 
		\end{align*}
		where $\mathrm R_{\ccc}$ is defined analogously than for $\xxx = \sss$.
		Back to the case $\xxx= \sss$, implementing the inequality \eqref{eq:bound_sum_traceLid} into \eqref{eq:before_last_ineq} we obtain
		\begin{align*}
			\sum_{\gamma^\prime \sim \gamma} q_{\sss}(\gamma^\prime) &
			\le N^{s-p} \ell^{F} (2m\ell)^6 c(n)^2\sum_{L_{\mathcal R}+ L_{\mathcal F} + L_{d} = 2m\ell}  (2\kappa) ^{L_{\mathcal R}}  \delta^{L_{\mathcal F}}  \rho_{n,\sss}^{L_{\mathrm d}}\\
			&\le  N^{s-p} \ell^{F} (2m\ell)^8 c(n)^2 \rho_{n,\sss}^{2m \ell}
		\end{align*}
		where we used the fact that owing to \cref{lem:normeplusque1} we have $(2\kappa) \vee \delta \vee \mathrm{N}_{n,\sss} \ge 1$ and the fact that $M+ R + \mathrm{x}_{\mathrm{d}} \le 2m(\ell +1)$. It remains to apply \cref{lem:upperboundF} and adapt the constant $c(n)>0$.

	\end{proof}
	We now give the proof of \cref{lem:upperboundF}.
	\begin{proof}[Proof of \cref{lem:upperboundF}]
		We have
		$$F:= \sum_{\cc \in \chi_{\mathrm{d}}} (|\mathcal{X}_{\cc}| + |\mathcal{Y}_{\cc}| -2) +\sum_{\cc \notin \chi_{\mathrm{d}}} (|\mathcal{X}_{\cc}| + |\mathcal{Y}_{\cc}| -1).$$
		We introduce $\chi_{\mathrm{d}}^\prime$ the set of $\cc \in \chi_{\mathrm{d}}$ such that the connected component verifies $|\mathcal{X}_{\cc}| =1$ and there is only one incoming edge $e \in \mathrm{A}_{\gamma}$ of multiplicity $m(e) \ge 2$. For all $\cc \in \chi_{\mathrm{d}}^\prime$, we have $\mathcal{T}_{\cc} \backslash \{e(\cc)\}= \emptyset$, therefore:
		\begin{align*}
			F\le \sum_{\cc \notin \chi_{\mathrm{d}}^\prime} (|\mathcal{X}_{\cc}| + |\mathcal{Y}_{\cc}| -1) \le p + \sum_{\cc  \notin \chi_{\mathrm{d}}^\prime} |\mathcal{Y}_{\cc}|
		\end{align*}
		where we used that $\sum_{\cc \in \chi} (|\mathcal{X}_{\cc}|-1) =p$. Now we re-write $Y:=\sum_{\cc  \notin \chi_{\mathrm{d}}^\prime} |\mathcal{Y}_{\cc}|$ noticing that we have $\mathcal{Y} := \bigcup_{\cc \notin \chi_{\mathrm{d}}^\prime} \mathcal{Y}_{\cc} = \mathcal{Y}_1 \sqcup \mathcal{Y}_2 \sqcup \mathcal{Y}_3$ where $\mathcal{Y}_1$ is the set of $y \in \mathcal{Y}$ seen only one time, i.e. there exist a unique $(i,t) \in \boldsymbol{T}$ such that $y_{i,t}=y$, $\mathcal{Y}_2$ is the set of $y\in \mathcal{Y}$ such that there exists an edge $(x,y) \in \mathrm{A}_{\gamma}$ of multiplicity at least two and such that $|\mathcal{X}_{\cc}| \ge 2$ and finally $\mathcal{Y}_3$ is the set of $y\in \mathcal{Y}$ seen by at least two different edges of $\mathrm{A}_{\gamma}$. Recall that the edges $e \in \mathrm{A}_{\gamma}$ are not multiple, but have a multiplicity instead. Precisely, recall that for any $e=(x,y) \in \mathrm{A}_{\gamma}$ there exists $(i,t) \in \boldsymbol{T}$ such that we have
		\begin{equation*}
			(x_{i,t},y_{i,t}) = (x,y) \text{ and } (\cc^\prime,\cc ) = (\cc (x_{i,t}) ,\cc(x_{i,t+1}))
		\end{equation*}
		and the multiplicity of an edge $e\in \mathrm{A}_{\gamma}$ is the number of $(i,t)$ verifying the previous equation. Then the set $\mathcal{Y}_1$ is in bijection with the corresponding incoming edges and these edges are inconsistent and of multiplicity one (see \cref{estimconsistent}). Therefore we have $|\mathcal{Y}_1| \le a_1$. If $y$ is seen by different edges, at least one of them is either discovered at a merging time or is an excess edge and therefore $|\mathcal{Y}_3| \le g+m$. Finally the number of $y$ seen several times by the same edge and such that $|\mathcal{X}_{\cc}| \ge 2$ is bounded by the number of corresponding edges. We denote by $\mathrm{E}_{22}$ the number of edges $e :=(x,y) \in \mathrm{A}_{\gamma}$ linking $\cc^\prime$ to $\cc$ for some $\cc^\prime, \cc \in \chi$ and such that $m(e) \ge 2$ and $|\mathcal{X}_{\cc}| \ge 2$.  We proceed as in the proof of \cite[Lemma 5]{bord2018spectral} to bound $\mathrm{E}_{22}$. We call \textit{in-degree} of a connected component $\cc \in \chi$ the number of edges $e\in \mathrm{A}_{\gamma}$ that links some $\cc^\prime \in \chi$ to $\cc$. Denoting $c_{k}$ the number of $\cc\in \chi$ with in-degree equal to $k$ and $c_{\ge 2} := \sum_{k\ge 2} c_k$ we have:
		$$c_0 + c_1 + c_{\ge2}= s-p ~~ \text{and} ~~ c_1 + 2 c_{\ge 2} \le \sum_{k} k c_k =a.$$
		Subtracting to the right-hand side, twice the left hand side,
		$$c_1 \ge 2(s - p) - a - 2c_0 \ge a - 2g -2m + 2.$$
		The bound $c_0 \le 2m$ follows from the observation that only a component $\cc \in \chi$ such that $\cc = \cc(x_{j,1})$ for some $1 \le j\le 2m$ can be of in-degree $0$. Besides, denoting $\mathrm{E}_{12}$ the number of edges of $\mathrm{A}_{\gamma}$ that are of multiplicity at least $2$ and such that the outgoing connected component $\cc \in \chi$ verifies $|\mathcal{X}_{\cc} |= 1$, we observe that $c_1 \le a_1 + \mathrm{E}_{12}$. Indeed we can associate each connected component $\cc$ of in-coming degree equal to one to its only in-coming edge. In particular we have
		$$|\mathcal{Y}_3| \le \mathrm{E}_{22} =a - a_1 - \mathrm{E}_{12} \le a-c_1 \le 2g +2m -2 ~~ \text{and therefore}~~ F \le 3(m+g) -2 +a_1 + p. $$
\end{proof}
\begin{rem}\label{rem:conj}   
One strategy to obtain \cref{conj:better} this conjecture would be to get an analog of \cite[Lemma 4]{bord2018spectral} in the following sense. We would need to improve the combinatorial arguments at stake at stake in \cref{pr1} and \cref{lem:upperboundF} in order to obtain that when summing over all $\gamma^{\prime} \sim \gamma$ for some $\vec{\ell}$ fixed tangle-free in \eqref{eq:improv}, we obtain 
\begin{equation*}
  \sum_{\substack{\gamma^{\prime} \sim \gamma\\ \vec{\ell}(\gamma^\prime) = \vec{\ell} } } q_{\sss}(\gamma^\prime) \le N^{s-p} \delta^{L_{\max}} \kappa^{L_{\mathcal R}}\prod_{\cc \in \chi_{\mathrm{d}}} \Tr_N \left( Q^{\ell(e(\cc))}{Q^{\ast}}^{\ell(e^\prime(\cc))}\right)
\end{equation*}
for some $L_{\max} \le Cm F$ with $C>0$ a universal constant. 
\end{rem}

    The last lemma of this section will estimate the contribution $\mu(\gamma)$ defined in \eqref{mu} using \cref{estimconsistent}.	
	\begin{lem}\label{boundingmu(gamma)}
		There is a constant $c>0$ such that, if $\gamma \in \mathrm{W}_{\ell,\vec{T}}(s,a,p)$, $g=a-s+p+1$ and $a_1$ is the number of arcs in $\mathrm{A}_{\gamma}$ visited exactly once by $\gamma$ then
		$$\left|\mu(\gamma)\right| \le c^{m+g}N^{-a} \left( \frac{6 \ell m}{\sqrt{N}} \right)^{(a_1-4g-2m+2p)_+}.$$
		Moreover we have $a_1 \ge 2(a-(\ell+1) m)$.
	\end{lem}

	\begin{proof}
		Let $\mathrm{A}_1 \subset \mathrm{A}_{\gamma}$ be the set of $e = (x,y)$ which are visited exactly once in $\gamma$, i.e. we have
		$$\sum_{i=1}^{2m}\sum_{t=1}^{T_i} \mathds{1}(e_{i,t}=e) = 1.$$
		Let $\mathrm{A}_{\ast}$ be the set of inconsistent arcs (see definitions above \cref{estimconsistent}) and we denote $b$ its cardinal. We consider $\mathrm{A}_{1}^\prime \subset \mathrm{A}_1$ the subset of consistent arcs of multiplicity one. We have
		$$|\mathrm{A}_{1}^\prime| + |\mathrm{A}_{\ast}| \ge |\mathrm{A}_1|$$
		We set $a_{1}^\prime := |\mathrm{A}_{1}^\prime|$ and $a_{\ge 2} := |\mathrm{A}_{\gamma} \backslash \mathrm{A}_{1}|$. We have
		$$a_1 + a_{\ge 2} = a ~~\text{ and } ~~ a_1 + 2a_{\ge 2} \le \sum_{k} k a_k = T \le 2 m (\ell+1).$$
		Therefore:
		
		$$a_1 \ge 2 (a-m(\ell+1)).$$
		It gives the second claim of the lemma. We now count the inconsistent edges using the vocabulary of the the proof of $\cref{boundingWlTsap}$. Exploring $\gamma$ chronologically, seeing a new inconsistent edge means that we discover an edge with either one of its ending vertices that had been seen before. That can happen either at the beginning of a sequence of times that goes out of the existing forest so far, at a merging time or at a first visit of an excess edge. Every time $(i,t)$ where we start a sequence of times going out of the existing forest and such that we discover a new inconsistent edge is preceded by a first visit of an excess edge. Every such step can create $2$ inconsistent edge: the one we see at this time and the one seen before with at least one of its ending point in common. Therefore we have
		$$b\le 4g +2m.$$
		However the bound on $b$ can be slightly improved. Indeed each time we discover a new element of an already seen connected component it means that either we have taken an excess edge or we are at a merging time and the corresponding edge $(x,y)$ is such that $x$ has never been seen before and therefore no possibly inconsistent outgoing edge was discovered before. Also either  the $y$ in this edge has never been seen before, and therefore no new inconsistent edge here, or it has been seen before and we are in presence of an excess edge. In conclusion each time we discover a new element of a connected component $\cc$, we have counted $2$ new inconsistent edges that cannot be. Therefore we have in fact
		$$b\le 4g +2m -2p$$ 
		and
		$$a^{\prime}_{1} \ge a_1 - 4g -2m +2p.$$
		It remains to apply \cref{estimconsistent}.
	\end{proof}
	We have all the ingredients to now prove \cref{boundingnormP}.
	\begin{proof}[Proof of \cref{boundingnormP}]
		We prove the case $\xxx= \sss$ and all the computations are exactly the same for $\xxx=\sss$. We set
		$$m = \lceil \ln(N)^\alpha\rceil .$$
		We have
		$$\mathbb{P}\left(\| \underline{A}_{\sss}^{(\ell)} \| \ge e^{C \ln(N)^{1-\alpha}} \rho_{n,\sss}^{\ell}\right)\le\mathbb{P}\left(\| \underline{A}_{\sss}^{(\ell)} \|^{2m} \ge N^{C}\rho_{n,\sss}^{2m\ell}\right).$$
		Applying Markov inequality to the previous probability and owing to \eqref{eq:htmethod}, proving our proposition boils down to proving that there exists a constant $\mathfrak{c}=\mathfrak{c}(n)>0$ such that
		\begin{equation}\label{defi:S}
			S:= \sum_{s,a,p,\vec{T}} \left|\mathcal{W}_{\ell,\vec{T}}(s,a,p) \right|\underset{\gamma \in \mathrm{W}_{\ell,\vec{T}}(s,a,p)}{\max}\left\{ |\mu(\gamma)| \sum_{\gamma^\prime} q(\gamma^{\prime})\right\} \le  N^{\mathfrak{c}} \rho_{n,\sss}^{2m\ell}.
		\end{equation}
		Let $\gamma \in \mathrm{W}_{\ell, \vec{T}} (s,a,p)$ and $a_1$ be the number of $(x,y) \in \mathrm{A}_{\gamma}$ of multiplicity one. Applying \cref{boundingsumq(gamma)} and \cref{boundingmu(gamma)} we have
		$$|\mu(\gamma)| \sum_{\gamma^{\prime}\sim \gamma} q_{\sss}(\gamma^{\prime}) \le N^{s-p} (c m \ell)^{a_1 +3(m+g) +p} c(n) \rho_{n,\sss}^{2m(\ell+1)}c^{m+g} N^{-a} \left(\frac{6\ell m }{\sqrt{N}} \right)^{(a_1 -4g -2m +2p)_{+}}. $$
		In regard of the conditions on $Q$ in \cref{bigtheorem} and the definition of $\rho_n$, $\ell$ and $m$, we find $ c> 1$ such that $\rho_n \le c \ell m$. Since $cm \ell \ge 1$ we have $(cm \ell)^{a_1+p} \le (cm\ell)^{(a_1-4g - 2m +2p)_{+}}(2m\ell)^{4g+2m}$ and using that $a_{1} \ge 2(a-m(\ell+1))$ we deduce:
		$$|\mu(\gamma)| \sum_{\gamma^{\prime}\sim \gamma} q_{\sss}(\gamma^{\prime}) \le c(n) (cm \ell)^{8g+8m}N^{-g+1} \rho_{n,\sss}^{2m\ell} \left(\frac{(cm\ell)^4}{N}\right)^{(a - (\ell+2)m -2g +p)_+},$$
		where $c>0$ is adapted if necessary. To ease notation we set
		$$\epsilon := \frac{(cm \ell)^{4}}{N}.$$
		Considering any $\ell=\mathcal{O}(\ln(N))$, we have $\epsilon = o(1)$. We apply \cref{boundingWlTsap} to $S$ defined by \eqref{defi:S} and sum over all $\vec{T}$ possible. We have at most $\ell^{2m}$ possibilities for $\vec{T}$ and considering the fact that for $c>1$ large enough we have $a,s$ and $\ell$ bounded by $c m \ell$ we obtain
		\begin{align*}
			S^\prime&:= \frac{S}{\rho_{n,\sss}^{2m\ell}} \le c(n) N \sum_{\substack{s,a,p \\ g(s,a,p) \ge 0, ~ p \le 2(m+g)}} \ell^{2m} (2 m\ell)^{12m(g+5)} (c m \ell)^{8g+8m} N^{-g} \epsilon^{(a - (\ell+2)m -2g +p)_+}\\
			& \le cN(cm \ell)^{80m}  \sum_{\substack{s,g,p \\ g(s,a,p) \ge 0, ~ p \le 2(m+g)}} (cm\ell)^{20mg} N^{-g} \epsilon^{(s - \ell^\prime m -g)_+},
		\end{align*}
		where we performed the change of variable $g=a-s+p+1$ and where we consider $\ell^\prime := \ell + 2 + 1/m$. Now summing over $ p \le 2 (m+g)$ we have in the sum an additional term $2(m+g)$ which we bound with $cm\ell$. Considering $\mathfrak c>0$ such that $ c(n)(cm \ell)^{81m} \le N^{\mathfrak c}$:
		$$S^{\prime} \le N^{\mathfrak c}  \sum_{s,g} \left( \frac{L}{N} \right)^g  \epsilon^{(s - \ell^\prime m -g)_+} $$
		where we set $L:= (cm\ell)^{20m}.$ Decomposing the sum on the right hand side above we have
		$$S^\prime \le S_1 + S_2 + S_3$$
		where $S_1$ is the sum over $\{ 1 \le s \le \ell^{\prime}m,~ g\ge 0\}$, $S_2$ is the sum over $\{\ell^\prime m < s, ~ 0\le g \le s-\ell^\prime m\}$ and $S_3$ the sum over $\{\ell^\prime m < s,~ g > s- \ell^\prime m\}$. In the sum $S_1$ we have $(s - \ell^\prime m -g)_+ =0$ and for our choice of $m$ and $\ell$, up to adapting the constant $c>1$ and taking $N$ large enough we have
		$$\frac{L}{N} = \frac{e^{c\ln(\ln(N))\ln(N)^\alpha}}{N} \le \frac{1}{2}.$$
		and therefore
		$$S_{1} \le  N^{\mathfrak c} \sum_{s=1}^{\ell^\prime m } 2 \le N^{\mathfrak{c}^\prime}.$$
		Similarly, using that $\frac{L}{N \epsilon} \ge 2$ for $N$ large enough and adapting the constants we have
		\begin{align*}
			S_2 &= N^{\mathfrak c} \sum_{s= \ell^\prime m +1 }^{+ \infty} \sum_{g=0}^{s- \ell^\prime m} \left(\frac{L}{N \epsilon} \right)^g \epsilon^{s-\ell^\prime m }\\
			& \le N^{\mathfrak{c}^\prime} \sum_{s= \ell^\prime m +1 }^{+ \infty}\left(\frac{L}{N} \right)^{s-\ell^{\prime}m} \le  N^{\mathfrak{c}^\prime} \sum_{k= 0 }^{+ \infty}\left(\frac{L}{N} \right)^k.
		\end{align*}
		Once again considering $N$ large enough the geometric sum converges and updating the constant we have the wanted bound. Finally for $N$ large enough we have
		\begin{align*}
			S_{3}&= N^{\mathfrak{c}} \sum_{s=\ell^{\prime}m+1}^{+ \infty} \sum_{g= s- \ell^\prime m +1}^{+ \infty} \left(\frac{L}{N} \right)^g \le N^{\mathfrak{c}} \sum_{s=\ell^{\prime}m+1}^{+ \infty} 2\left(\frac{L}{N} \right)^{s- \ell^\prime m +1} \le N^{\mathfrak{c}^\prime} \sum_{k=0}^{+\infty} \left(\frac{L}{N} \right)^{k}.
		\end{align*}
		It remains to consider $N$ large enough and $\mathfrak{c}^{\prime \prime} >0$ large enough so \eqref{defi:S} holds. 
	\end{proof}
	
	\subsection{Operator norm of $R^{(\ell)}_{k,\xxx}$}\label{section:op_norm_rest}
	We will now adapt what have been doing to bound the operator norm $A_{\xxx}^{(\ell)}$ to upper bound the operator norm of $R^{(\ell)}_{k,\xxx}$ defined by \eqref{eq:defirestvecell} for $\xxx =\sss$ and \eqref{eq:defirestvecell2} for $\xxx=\ccc$. We set
	\begin{equation*}
		\Delta_{\sss} := 1+ \kappa + \delta ~~\text{and}~~\Delta_{\ccc} := 1+ r\kappa + r\delta
	\end{equation*}
	where $\kappa$ is defined by \eqref{eq:kappa} and $\delta>0$ comes from \ref{H:1}.
	
	\begin{prop}\label{boundingnormR}	Let $(Q,M),~ n,~ \rho_n$ and $A$ be as defined in \cref{bigtheorem}. For any $c >0$, $0 < \alpha < 1$ and $1 \le k \le \ell \le  \frac{\ln(N)}{7\ln(d)}$, there exists $C = C(c, \alpha, n, \delta) >0$ such that with probability at least $1 - N^{-c}$,
		$$\|R_{k,\xxx}^{(\ell)}\| \le e^{C \ln(N)^{1-\alpha}}\Delta_{\xxx}^\ell.$$
	\end{prop} 
	
	In all this section the notations are the one introduced in \cref{subsectionboundofP}. We also give the proof of the previous proposition only in the case $\xxx =\sss$. We again use the inequality
	\begin{equation}\label{TracedeRk}
		\|R^{(\ell)}_{k,\sss}\|^{2m} \le \Tr\left\{ \left(R_{k,\sss}^{(\ell)}{R_{k,\sss}^{(\ell)}}^{\top}\right)^m\right\}=:\Tr[k,\ell,m].
	\end{equation}
	Let $\vec{T} \in [\ell]^{2m}$ be fixed. We denote by $\mathrm{W}^{\prime}_{\ell, \vec{T}}$ the set of $\gamma:= ( \gamma_1,...,\gamma_{2m})$ such that for all $i \in [2m]$
	\begin{itemize}
		\item $T(\gamma_i):= T_i$ and $\gamma_i=(x_{i,1}, y_{i,1}, \ell_{i,1}, x_{i,2},...,x_{i,T_{i}}, y_{i,T_{i}},\ell_{i,T_i},x_{i,T_i+1}) \in \mathrm{T}^{k,\ell}$,
		\item the condition \eqref{eq:condilim} is verified.
	\end{itemize}
	With this notations, expending the trace in \eqref{TracedeRk} gives
	\begin{equation}\label{traceRkl}
		\Tr[k, \ell,m]= \sum_{\vec{T} \in [\ell]^{2m}} \sum_{\gamma\in \mathrm{W}^{\prime}_{\ell,\vec{T}}} \prod_{i=1}^{2m}\prod_{t=1}^{k-1}\underline{M}_{x_{i,t}y_{i,t}} (Q^{\ell_{i,t}})_{y_{i,t}x_{i,t+1}} \cdot (Q^{\ell_{i,k}})_{y_{i,k} x_{i,k+1}} \cdot \prod_{t=k+1}^{T_i}M_{x_{i,t}y_{i,t}}(Q^{\ell_{i,t}})_{y_{i,t}x_{i,t+1}}.
	\end{equation}
	In regard of the definition of $\mathrm{T}^{k,\ell}$, if $\vec{T}\in [\ell]^{2m}$ and there exists $i$ such that $T_{i}<k$ then $\mathrm{W}^\prime_{\ell,\vec{T}} = \emptyset$. For $i\in [2m]$ we set
	$$\gamma_{i}^\prime = (x_{i,1},y_{i,1},\ell_{i,1},x_{i,2},...,x_{i,k-1},y_{i,k-1},\ell_{i,k-1},x_{i,k})$$
	and
	$$\gamma_{i}^{\prime \prime} = (x_{i,k+1},y_{i,k+1},\ell_{i,k+1},...,x_{i,T_i},y_{i,T_i},\ell_{i,T_i},x_{i,T_i+1}).$$
	As in \cref{subsectionboundofP} we define the graph $K_{\gamma}$ for a fixed $\gamma \in \mathrm{W}^{\prime}_{\ell,\vec{T}}$. We consider $X_{\gamma}:= \{ x_{i,t}:i\in [2m], t \in [T_i +1]\}$ and $Y_{\gamma} := \{y_{i,t}: i\in [2m], t\in [T_i]\}$. As before we set $K_{\gamma}= (X_{\gamma}, E(K_{\gamma}))$ and $\{x, x^{\prime}\}$ is a vertex if there exists $(i,t), (j,s) \in \boldsymbol{T}$  such that 
	\begin{equation*}
		x_{i,t+1}=x,~x_{j,s+1} = x^\prime ~~\text{ and } ~~ y_{i,t}=y_{j,s}.
	\end{equation*}
	For $x \in X_{\gamma}$ we consider again $\cc(x)$ its connected component in $K_{\gamma}$ and $\chi_{\gamma}=X_{\gamma}/\sim$ the set of connected component of $K_{\gamma}$. The \textbf{arcs} of $\gamma$ are the pairs $(x_{i,t},y_{i,t})$ for $t \neq k$ and we denote by $\mathrm{A}^{\prime}_{\gamma}$ the set of arcs of $\gamma$. For $s,a,p$ integers, we denote by $\mathrm{W}^\prime_{\ell,\vec{T}}(s,a,p)$ the set of $\gamma \in \mathrm{W}^\prime_{\ell,\vec{T}}$ verifying $|X_{\gamma}|=s$, $|X_{\gamma}/\sim|= s-p$ and $|\mathrm{A}_{\gamma}^\prime|=a$. Adapting the computations for the case $\xxx=\ccc$ and taking the expectation of \eqref{traceRkl} we have
	$$\mathbb{E} \|R_{k,\xxx}^{(\ell)}\|^{2m} \le \mathbb{E}(\Tr[k,\ell,m]) = \sum_{\vec{T},s,a,p}~~ \sum_{\gamma \in \mathrm{W}^{\prime}_{\ell,\vec{T}}(s,a,p)} \mu^{\prime}(\gamma) q_{\xxx}(\gamma)$$
	where for $\gamma \in \mathrm{W}^{\prime}_{\ell,\vec{T}}(s,a,p)$ we define
	\begin{equation}\label{mu'}
		\mu^\prime (\gamma):=\mathbb{E}\prod_{i=1}^{2m}\prod_{t=1}^{k-1} \underline{M}_{x_{i,t}y_{i,t}}\prod_{t=k+1}^{T_i}M_{x_{i,t}y_{i,t}} 
	\end{equation}
	and 
	\begin{equation}
		q_{\sss}(\gamma):= \prod_{i=1}^{2m}\prod_{t=1}^{T_i} (Q^{\ell_{i,t}})_{y_{i,t}x_{i,t+1}}	~~\text{ and } ~~ q_{\ccc}(\gamma):= (1-r)^{T-1} \prod_{i=1}^{2m}\prod_{t=1}^{T_i} (\{rQ\}^{\ell_{i,t}})_{y_{i,t}x_{i,t+1}}
	\end{equation}
	where for $\vec{T}=(T_1,...,T_{2m}) \in [\ell]^{2m}$ we set $T = \sum_{i} T_i$.
	We decompose $\mathrm{W}^{\prime}_{\ell,\vec{T}}(s,a,p)$ into equivalent classes. We say that $\gamma \sim \gamma^{\prime}$ if there exists $\alpha, \beta $ permutations of $[N]$ such that for all $i \in [2m]$ and all $t \in [T_i+1]$ we have $\alpha(x_{i,t})= x^{\prime}_{i,t}$ and $\beta(y_{i,t})= y^{\prime}_{i,t}$, where $\gamma^\prime := (\gamma_1^\prime,...,\gamma_{2m}^{\prime})$ and $\gamma_{i}^\prime:= (x_{i,1}^\prime,y_{i,1}^\prime,\ell_{i,1}^\prime,....,x^\prime_{i,T_i},y^\prime_{i,T_{i}},\ell_{i,T_i}^\prime,x^\prime_{i,T_i+1})$ for $i \in [2m]$. We denote by $\mathcal{W}^\prime_{\ell,\vec{T}}(s,a,p):= \mathrm{W}^{\prime}_{\ell,\vec{T}}(s,a,p)/\sim$ the set of equivalent classes. We have
	\begin{equation}\label{defSR_kl}
		\mathbb{E}\|R_{k,\xxx}^{(\ell)}\|^{2m} \le \sum_{\vec{T},s,a,p} |\mathcal{W}^\prime_{\ell,\vec{T}}(s,a,p)| \underset{\gamma\in \mathrm{W}^{\prime}_{\ell,\vec{T}}(s,a,p)}{\max} \left( |\mu^\prime(\gamma)| \sum_{\gamma^\prime \sim \gamma} q_{\xxx}(\gamma^\prime) \right).
	\end{equation}
	
	As before we first bound the cardinality of $\mathcal{W}^\prime_{\ell,\vec{T}}(s,a,p)$, independently of the choice of $\vec{T}$. 
	
	\begin{lem}\label{boundWlTR_k}
		If $g^{\prime}:= a+p-s < 0$ or $2g^\prime + 10m > p$ then $\mathrm{W}^{\prime}_{\ell,\vec{T}}(s,a,p)$ is empty. Otherwise we have:
		$$|\mathcal{W}^\prime_{\ell,\vec{T}}(s,a,p)| \le  (\ell (a+2m)^2 (s-p)^2 s^2)^{4m(g^\prime+ 7)}.$$
	\end{lem}
	
	\begin{proof}
		We repeat the construction given in \cref{section:encodingeqclass}. Let $\ell\ge 1$, $\vec{T} \in [\ell]^{2m}$ and  $\gamma \in \mathrm{W}^{\ell,\vec{T}}(s,a,p)$ be fixed. We set $\gamma_{i,t} := ( x_{i,t},y_{i,t},x_{i,t+1})$ for $i \in [2m]$ and $t \in [T_i]$. We explore $(\gamma_{i,t})_{(i,t)\in \boldsymbol{T}}$ chronologically where $\boldsymbol{T}:= \{(i,t): i\in [2m], t\in[T_i]\backslash\{k\}\}$ is ordered as in \cref{section:encodingeqclass}. Recall in particular that for $i \le i^\prime$, for all $t\in [T_i]\backslash\{k\}$ and $t^\prime \in [T_{i^\prime}]\backslash\{k\}$ we have $(i,t) \preceq (i^\prime,t^\prime)$, and for all $t \le t^\prime \in [T_i]\backslash\{k\}$ we have $(i,t) \preceq (i,t^\prime)$.\\
		
		For $y \in Y_{\gamma}$, we denote by $\bar{y}$ the order of apparition of $y$ in the ordered sequence $(y_{i,t})_{(i,t) \in \boldsymbol{T}}$. The notation $\bar{x}$ and $\bar{\cc}(x)$ are defined similarly in regard of the apparition order of $x \in X_{\gamma}$ or $\cc(x) \in K_{\gamma}/\sim $. Finally for $x^\prime \in X_{\gamma}$, we set the \textit{connected component color} of $x$ as the function $\bar{c}(\bar{x})$ that return the first element ever seen of the connected component of $x$, i.e.
		$$\bar{c}(\bar{x}):= \min\{\bar{x}^{\prime} \le \bar{x}: x^{\prime} \in \cc(x) \}.$$
		Finally we define the \textit{connected component mark} of $x \in X_{\gamma}$ by $\vec{x}:= (\bar{c}(\bar{x}),\bar{x})$ and we set $\bar{\gamma}_{i,t}:=(\vec{x}_{i,t},\bar{y}_{i,t},\vec{x}_{i,t+1})$ for all $(i,t) \in \boldsymbol{T}$. By construction, knowing $(\bar{\gamma}_{i,t})_{(i,t)\in \boldsymbol{T}}$ is equivalent to knowing the equivalent class of $\gamma$. We will now encode the sequence $(\bar{\gamma}_{i,t})_{(i,t) \in \boldsymbol{T}}$.\\
		
		We set $V_{\gamma}:= [s-p]$ and consider the colored directed graph $G_{\gamma}^\prime :=(V_{\gamma}, E_{\gamma}^\prime)$ defined as follows. For each time $(i,t) \in \boldsymbol{T}$ such that $t \neq k$, we put a directed edge $e_{i,t}:= (\bar{\cc}(x_{i,t}),\bar{\cc}(x_{i,t+1}))\in E_{\gamma}^{\prime}$ and colorize it with $(\bar{x}_{i,t},\bar{y}_{i,t})$. By definition we will have $|\mathrm{A}_{\gamma}^\prime|=|E_{\gamma}^\prime| =a$. We denote by $\bar{G}^\prime_{\gamma}$ the associate undirected graph. We observe that each connected component of $\bar{G}^\prime_{\gamma}$ has at least one cycle. Indeed let $\mathcal{C} \subset \bar{G}^\prime_{\gamma}$ be a connected component and $i \in [2m]$ such that $\Ima(\gamma_i) \cap \mathcal{C} \neq \emptyset$. Then the path $\gamma_{i}$ being tangled we have at least two cycles in the image of $\gamma_i$. Also for all $\epsilon \in \{\prime, \prime \prime\}$, by definition $\Ima(\gamma^\epsilon_{i})$ is contained in one connected component of $\bar{G}^\prime_{\gamma}$. Therefore we have only two possibilities. Either $\Ima(\gamma_{i}^\prime)\cap \Ima(\gamma_{i}^{\prime \prime}) \neq \emptyset$ and therefore $\Ima(\gamma_{i}) \subset \mathcal{C}$, then both cycles (and the others) of $\gamma_i$ are contained in $\mathcal{C}$. One could argue that one of the edges of one of the cycle might be $(x_{i,k},y_{i,k})$, which is not an element of $E_{\gamma}^\prime$. In this case, the fact that we have two cycles in $\gamma_i$ implies that removing one edge leaves at least one cycle. Otherwise, $\gamma_{i}$ being tangled and the images of $\gamma_{i}^\epsilon$ being distinct implies that both $\gamma_{i}^\prime$ and $\gamma_{i}^{\prime\prime}$ have a cycle and therefore at least one is in $\mathcal{C}$. We deduce from this fact that in each connected component we have at least as much edges than we have vertices and therefore, summing over all connected component
		$$0 \le g^\prime = |E_\gamma^\prime|-|V_{\gamma}| = a - s + p.$$
		This is the first claim of the lemma.\\
		
		For every time $(i,t)\in \boldsymbol{T}$ we define $G_{i,t}^\prime$ as the subgraph of $G_{\gamma}^\prime$ spanned by the edges $e_{j,s}$ for all $(j,s) \preceq (i,t)$. We have $G_{2m,T_{2m}}^\prime = G_{\gamma}^\prime$ and we again define a spanning forest $\Gamma_{i,t}$ of $G_{i,t}^\prime$ inductively. We set $\Gamma_{1,0} =(\{1\},\emptyset)$. We say that $(i,t) \in \boldsymbol{T}$ is a \textit{first time} if adding $e_{i,t}$ to $\Gamma_{(i,t)^{-}}$ does not create a weak cycle. If $(i,t) \in \boldsymbol{T}$ is a first time then $\Gamma_{i,t}:= \Span\{\Gamma_{(i,t)^-},e_{i,t}\}$, otherwise $\Gamma_{i,t} = \Gamma_{(i,t)^-}$. Finally we set $\Gamma:= \Gamma_{2m,T_{2m}}$.\\
		For each even era $i \in [2m]$, we define the \textit{first merging time} $(i,t^\prime_{i})$ as the smallest time $(i,t)$ with $1 \le t \le k-1$, if it exists, such that $\Gamma_{i,t}$ and $\Gamma_{(i,1)^-}$ have the same number of connected component. If this time does not exists exist we take as convention $t^\prime_{i} = k$. Similarly, for each $i$, we define the \textit{second merging time} $(i,t_{i}^{\second})$ as the smallest time $(i,t)$ with $k \le t \le T_{i}$ such that $\Gamma_{i,t}$ has the same number of connected components than $\Gamma_{(i,k)^-}$. If this time does not exists we set $t^{\second}_{i}= T_{i}+1$. In regard of \eqref{eq:condilim}, we have $t^{\second}_{i} \le T_{i}$ for $i$ even. Also notice that if $t_{i}^\epsilon \ge 2$ then the merging time is a first time.\\
		The edges of $G_{\gamma}^\prime \backslash \Gamma$ are called \textit{excess edges}. The times $(i,t) \in \boldsymbol{T}$ such that $e_{i,t}$ is an excess edge are called \textit{important time}. Denoting $c_{\gamma}$ the number of connected component of $G_{\gamma}^\prime$, the total number excess edges is
		$$|G_{\gamma}^\prime \backslash \Gamma| = |E_{\gamma}^\prime| - |V_{\gamma}| + c_{\gamma} = g^\prime + c_{\gamma}.$$
		To recover the previous equality, we sum the number of excess edges in each connected component. Owing to \eqref{eq:condilim}, there is at most one new connected component appearing for each era and therefore we have $1 \le c_{\gamma}\le 2m$. Besides there is at least one cycle in each connected component therefore there are at most $g^\prime +1 $ excess edges in each connected component of $G_{\gamma}^\prime$.\\
		
		By construction the path  $\gamma_{i}^\prime$ and $\gamma^{\second}_{i}$ can be decomposed by a successive repetition of the pattern (1)(2)(3) given by
		\begin{enumerate}[label=(\arabic*)]
			\item  a sequence of first times (possibly empty)
			
			\item an important time or a merging time
			
			\item a path using the colored edges of the forest defined so far (possibly empty).
		\end{enumerate}
		We encode the times as follows. If $(i,t) \in \boldsymbol{T}$ is an important time we mark $(i,t)$ by the vector $(\bar{y}_{i,t}, \vec{x}_{i,t+1}, \vec{x}_{i,\tau})$ where $(i,\tau)$ is the next step outside $\Gamma_{i,t}$ (by convention if $\gamma_{i}$ stays in the forest for the remaining time of the era $i$ we set $\tau = T_{i}+1$). By construction $(i,\tau)$ is the next first, important or merging time. Similarly if $(i,t)$ is a first merging time, we mark the time $(i,t)$ by the \textit{first merging time mark} $(\bar{y}_{i,t},\vec{x}_{i,t+1},\vec{x}_{i,t})$ where $(i,\tau)$ is the next step outside $\Gamma_{i,t}$. For $(i,t)$ a second merging time we mark it with the \textit{second merging time mark} $(\bar{y}_{i,k},\bar{y}_{i,t},\vec{x}_{i,t+1},\vec{x}_{i,\tau})$ where $(i,\tau)$ is first time outside $\Gamma_{i,t}$. The positions of the important and merging times and their marks encodes $(\bar{\gamma}_{i,t})_{(i,t) \in \boldsymbol{T}}$. Besides denoting $\cc_{k}$ the set of vertices of the $k$-th connected component of $K_{\gamma}$ we have that $p= \sum_{k=1}^{s-p} (|\cc_{k}|-1)$. Hence $p$ is equal to the number of times we discover a new element $x_{i,t}$ in a connected component that we already have seen. The number of such times is upper bounded by twice the number of excess edges plus the number of merging times
		$$p \le 2(g^\prime + N_{\gamma} + 3m ) \le 2g^\prime +10 m.$$
		It proves the second statement of the lemma. \\
		
		As in \cref{section:encodingeqclass}, the number of important times encoded is to important for our purposes, therefore we improve the way we encode the paths taking into account the tangle-freeness of $\gamma^\epsilon_{i}$. We partition the encoding era wise, and partition the important times in each era into three categories: \textit{short cycling times}, \textit{long cycling times} and \textit{superfluous times}.
		Let $i \in [2m]$ and $\epsilon \in \{\prime, \second\}$ be fixed. If it exists, let $(i,t_0)$ be the smallest time such that $\cc(x_{i,t_0 +1}) \in \{\cc(x_{i,\ast}), \cc(x_{i,\ast+1}),..., \cc(x_{i,t_0})\}$ where $\ast=1$ if $\epsilon = \prime$ and $\ast= k+1$ otherwise. Let $ \ast \le \sigma \le t_0$ be such that $\cc(x_{i,t_0+1}) = \cc(x_{i,\sigma})$. By assumption, $\mathrm{C}_{i} = (\cc(x_{i,\sigma}),...,\cc(x_{i,t_0+1}))$ will be the only cycle of $G_{\gamma}^\prime$ seen by $\gamma_{i}^\epsilon$. The last important time $(i,t) \preceq (i,t_{0})$ is called the \textit{short cycling time}. We denote by $(i,\hat{t})$ the smallest time $(i,\hat{t}) \succeq (i,\sigma)$ such that $\cc(x_{i,\hat{t}+1}) \notin \mathrm{C}_i$. If $\gamma_{i}^\epsilon$ stays in $\Gamma_{i,t}$ for its remaining time, by convention we set $\hat{t} = k $ (\textit{resp}. $\hat{t}= T_{i}+1$) for $\epsilon = \prime$ (\textit{resp}. for $\epsilon = \prime\prime$). We replace the mark of the short cycling time by $(\bar{y}_{i,t},\vec{x}_{i,t+1},\sigma,\hat{t},\vec{x}_{i,\tau})$ where once again $(i,\tau)$ is such that $ \tau \ge t$ and $(i,\tau)$ is the next step outside of $\Gamma_{i,t}$. The time $(i,\tau)$ is the next first, important or merging time after $(i,\hat{t})$. If $\gamma_{i}^\epsilon$ stays in $\Gamma_{i,t}$ we set $\tau=k$ (\textit{resp}. $\tau=T_{i}+1$) for $\epsilon = \prime$ (\textit{resp}. for $\epsilon= \second$). Important times $(i,t^\prime)$ with $1 \le t^\prime \le t$ or $\tau \le t^\prime \le k$ (or in the case $\epsilon = \second$, $k+1 \le t^\prime \le t$ or $\tau \le t^\prime \le T_{i}$) are called \textit{long cycling times}. The other important times with $t < t^\prime < \tau$ are called \textit{superfluous}.\\
		We have our second encoding and we can recover the sequence $(\bar{\gamma}_{i,t})_{(i,t)\in \boldsymbol{T}}$ from the positions of the merging times, the long cycling times, the short cycling times and their respective marks. For each $i \in [2m]$ and for each $\epsilon \in \{\prime,\second\}$, there is at most one short cycling time, one merging time, and as argued in the proof of \cref{boundingWlTsap} $g^\prime + 1$ long cycling times. For each time we have $T_{i} \le \ell$ or $k \le \ell$ ways to position them. Therefore we have a total of $\ell^{4m(g^\prime +3)}$ ways to positions our marked times. We have $|Y_{\gamma}| \le a +2m =: a^\prime$, where the term $2m$ coming from the element $y_{i,k}$ for $i \in[2m]$. We have at most ${a^\prime} (s-p)^2 s^2$ ways to mark a long cycling time, at most $a^\prime (s-p)^2 s^{2} \ell^{2}$ ways to mark a short cycling time and ${a^\prime}^2 (s-p)^2s^2$ ways to mark a merging time. We deduce from that
		\begin{align*}
			|\mathcal{W}^\prime_{\ell, \vec{T}} (s,a,p)| & \le \ell^{4m(g^\prime +3)}({a^\prime} (s-p)^2 s^2)^{4m(g^\prime +1)} (a^\prime(s-p)^2 s^2 \ell^2)^{4m} ({a^\prime}^2 (s-p)^2s^2 )^{4m}\\
			&\le (\ell (a+2m)^2 (s-p)^2 s^2)^{4m(g^\prime+ 7)}.
		\end{align*}
		
	\end{proof}
	
	\begin{lem}\label{boundingsumqprime(gamma)}
		The exists $c >0$ such that for all $\gamma \in \mathrm{W}_{\ell,\vec{T}}^\prime(s,a,p)$,
		$$\sum_{\gamma^\prime \sim \gamma} q_{\xxx}(\gamma^\prime) \le N^{s-p} cm\ell\Delta_{\xxx}^{2m(\ell+1)}.$$
	\end{lem}
	For $\gamma \in \mathrm{W}_{\ell,\vec{T}}^\prime(s,a,p)$ we consider for each $\cc \in \chi_{\gamma}$ the connected and bi-partite graph in the sense \cref{def:graphbipart} $\mathcal{G}_{\cc}= (\mathcal{X}_{\cc} \bigsqcup \mathcal{Y}_{\cc},\mathcal{E}_{\cc})$ and its underlying covering tree $\mathcal{T}_{\cc}$ constructed as we did in \cref{section:contribQ}. We consider $\mathcal{G}$ the union graph as in \cref{section:contribQ} and
	\begin{equation*}
		\mathcal{F}= \bigsqcup_{\cc \in \chi} \mathcal{T}_{\cc},~~\text{and}~~ \mathcal{R}= \bigsqcup_{\cc \in \chi} \mathcal{R}_{\cc}.
	\end{equation*}
	
	\begin{proof}
		We give the detail of the proof for the case $\xxx = \sss$. We start from         \begin{align}\label{eq:1stboundonsumqgamma}	\sum_{\gamma^{\prime} \sim \gamma} q_{\sss}(\gamma^{\prime})&= \sum_{\gamma^{\prime} \sim \gamma}\prod_{\cc \in \chi} \prod_{\substack{f\in \mathcal{E}_{\cc}\\ f=(u,v)}} (Q^{\ell(f)})_{x^{\prime}_u,x^{\prime}_v} \le \sum_{\substack{y_{u},x_{v} \in[N]\\ (u,v) \in \mathcal{Y} \times \mathcal X}} \sum_{(\ell(f))_{f \in \mathcal{E}} \in \mathcal{L}^\prime(\mathcal{E},M)} \prod_{\cc \in \chi} \prod_{f \in \mathcal{E}_{\cc}} (Q^{\ell(f)})_{y_u x_v} 
		\end{align}
		Using the same notations than in the proof of \cref{boundingsumq(gamma)}, for a fixed choice of $\{ (y_u, x_{v}):~ (u,v) \in \mathcal{Y} \times \mathcal X \}$, we rewrite and bound the second sum in \eqref{eq:1stboundonsumqgamma} as follows
		\begin{align*}
			\sum_{\ast} \prod_{\cc \in \chi}\prod_{f \in \mathcal{E}_{\cc}} (Q^{\ell(f)})_{y_u x_v} &= \sum_{L=0}^M \sum_{\substack{ \ast \ast}}\prod_{\cc \in \chi}\prod _{f \in \mathcal{T}_{\cc}}  (Q^{\ell(f)})_{y_u x_v} \sum_{\ast \ast \ast} \prod_{\cc \in \chi}\prod _{f \in \mathcal{R}_{\cc} }  (Q^{\ell(f)})_{y_u x_v}
		\end{align*}	
		where the summand $\ast$ is over all $ (\ell(f))_{f} \in \mathcal{L}^\prime( \mathcal E , M)$, the summand $\ast \ast$ is over $(\ell(f))_{f} \in \mathcal{L}^\prime( \mathcal{F}, L)$ and finally the summand $\ast \ast \ast $ is over $(\ell (f))_{f} \in \mathcal{L}^\prime(\mathcal{R} , M-L)$. For convenience we take $(\kappa+1)^{M-L+R}$ as another upper bound in \eqref{eq:appear2kappa}.
		We have
		\begin{align*}
			\sum_{\gamma^\prime \sim \gamma} q(\gamma^\prime) & \le \sum_{L=0}^M (\kappa+1) ^{M-L+R} \sum_{\ast \ast } \prod_{\cc \in \chi} \sum_{\substack{y_{u},x_{v} \in[N]\\ (u,v) \in \mathcal{Y}_{\cc} \times \mathcal{X}_{\cc}}}  \prod_{f \in \{e^\prime\} \sqcup \mathcal{T}_{\cc}} (Q^{\ell(f)})_{y_u x_v}\\
			&\le \sum_{L=0}^M (\kappa+1) ^{M-L+R} \sum_{\ast \ast } N^{s-p} \delta^{L} \le N^{s-p} cm\ell (\kappa + \delta + 1)^{2m(\ell+1)}
		\end{align*}
		where we applied \cref{lem:boundingpdtZ(G)} between the first and second line and bounded $M \le c m \ell$ and $\abs{\mathcal{L}^\prime(\mathcal{F},L)} \le \binom{L+F-1}{F-1}$. We also used the fact that $M+R+F =2m(\ell+1)$.
	\end{proof}
	\begin{lem}\label{boundingmuprime(gamma)}
		There exists a constant $c >0 $ such that, if $\gamma \in \mathrm{W}_{\ell,\vec{T}}^\prime (s,a,p)$, denoting $g^\prime = a+p-s$ we have
		$$|\mu^\prime(\gamma)| \le c^{m+g^\prime} N^{-a}.$$
	\end{lem}
	
	\begin{proof}
		We use the terminology used in the proof of \cref{boundWlTR_k} above. Let $A_{\ast}$ be the set of inconsistent arcs of $A_{\gamma}^\prime$ (as defined above in \cref{estimconsistent}). As argued in the proof of \cref{boundingmu(gamma)},$|A_{\ast}|$ is bounded by four times the number of excess edges plus twice the number of merging times. There are at most $g^\prime +2m$ excess edges and $3m$ merging times, hence:
		$$|A_{\ast}| \le 4(g^\prime+ 2 m) + 6m.$$
		Applying \cref{estimconsistent} we obtain the wanted result. 
	\end{proof}
	We can now prove \cref{boundingnormR}.
	\begin{proof}[Proof of \cref{boundingnormR}]
		We set
		$$m = \lceil  \ln(N)^\alpha\rceil.$$
		From \eqref{defSR_kl} and Markov inequality, it suffices to prove that there exists $\mathfrak c> 0$ such that
		\begin{equation}
			S=\sum_{\vec{T},s,a,p} |\mathcal{W}^\prime_{\ell,\vec{T}}(s,a,p)| \underset{\gamma\in \mathrm{W}^{\prime}_{\ell,\vec{T}}(s,a,p)}{\max} \left( |\mu^\prime(\gamma)| \sum_{\gamma^\prime \sim \gamma} q_{\xxx}(\gamma^\prime) \right) \le N^{\mathfrak c}\Delta_{\xxx}^{2m\ell}.
		\end{equation}
		We fix $\ell\ge 0$ and $\vec{T} \in [\ell]^{2m}$ and $\gamma \in \mathrm{W}^\prime_{\ell, \vec{T}}(s,a,p)$. We set $g^\prime = g^\prime(s,a,p) = a-s+p$. Owing to \cref{boundingmuprime(gamma)} and \cref{boundingsumqprime(gamma)} we have
		$$|\mu^{\prime}(\gamma)| \sum_{\gamma^{\prime} \sim \gamma} q(\gamma^{\prime}) \le c^{m+g^{\prime}}cm\ell\Delta_{\xxx}^{2m(\ell+1)}N^{-g^{\prime}}.$$
		Besides we find $c> 0$ such that $\ell$, $a+2m$, $s$ and $(s-p)$ are bounded by $cm\ell$. Then owing to \cref{boundWlTR_k}
		$$|\mathcal{W}^{\prime}_{\ell,\vec{T}}(s,a,p)| \le (cm\ell)^{28m (g^\prime+7)}.$$
		Therefore summing over all $\vec{T}$ possible, i.e. at most $\ell^{2m} \le (cm\ell)^{2m}$ we have
		\begin{align*}
			S & \le \sum_{\substack{s,a,p; \\ g^\prime(s,a,p) \ge 0,~ p \le 2g^\prime +10}} \Delta_{\xxx}^{2m(\ell+1)} (cm\ell)^{28m(g^\prime +8)+1} N^{-g^\prime} c^{m+g^\prime}\\
			& \le (c\ell m )^{227m}\Delta_{\xxx}^{2m(\ell+1)} \sum_{g^\prime} (c\ell m)^{28mg^\prime} N^{-g^\prime}
		\end{align*}
		where at the last line we performed the change of variable $a \to g^\prime = a +p -s$ and we summed over the variable $s \le c m \ell$ and $p \le 2g^{\prime} + 10m \le cm\ell$. Up to updating the constant $c >0$ we have
		$$S \le N^{\mathfrak c} \Delta_{\xxx}^{2m\ell} \sum_{g^{\prime}} \left(\frac{L}{N}\right) ^{g^{\prime}}$$
		with $L = (c\ell m )^{28m} \le N/2$ for $N$ large enough.
	\end{proof}
	
	\section{Proof of \cref{big2theorem}}\label{section:proofbigth}
	All ingredients are now gathered to prove \cref{big2theorem}. Let $0< \alpha,r <1$, $0< c <c^\prime <1$ and $n \ge 0$ be fixed. We can assume for all that follows that $\ell \le \frac{\ln(N)}{\ln(d)}$ since otherwise the bound of \cref{pr1} would be larger than $1$. Notice that for $N$ large enough, $c\frac{\ell^4 (4d)^{3 \ell}}{N} \le c\frac{(4d)^{6\ell}}{N}$. Therefore denoting $\Omega$ the event $(M,Q)$ is $\ell$-tangle free and applying \cref{pr1} and \cref{MQltf}, for any $C>0$ we have
	
	\begin{align*}
		\mathbb{P} \left( \|A^\ell_{\xxx|_{\boldsymbol{1}^\perp}}\| \ge e^{C \ln(N)^{1-\alpha}} \rho_{n,\xxx}^\ell\right) &\le  \mathbb{P} \left( \|A^\ell_{\xxx|_{\boldsymbol{1}^\perp}}\| \ge e^{C \ln(N)^{1-\alpha}} \rho_{n,\xxx}^\ell \cap \Omega \right) + \frac{(4d)^{6\ell}}{N}\\
		&\le \mathbb{P} \left( J \ge e^{C \ln(N)^{1-\alpha}} \rho_{n,\xxx}^\ell\right) + \frac{(4d)^{6\ell}}{N},
	\end{align*}
	where 
	$$ J := \| \underline{A}_{\xxx}^{(\ell)}\| + \frac{1}{N}\sum_{k=1}^{\ell} \|R_{k,\xxx}^{(\ell)}\|.$$
	Owing to \cref{boundingnormP} and \cref{boundingnormR} we find $C=C(n) >0$ such that with probability at least $1 - 2N^{-c^\prime}$
	\begin{align}
		J &\le e^{C  \ln(N)^{1-\alpha}} \rho_{n,\xxx}^{\ell} + \frac{1}{N} \sum_{k=1}^{\ell} e^{C \ln(N)^{1-\alpha}}\Delta_{\xxx}^{\ell}
		\le e^{C  \ln(N)^{1-\alpha}} \rho_{n,\xxx}^\ell\left(1 + \frac{ \ell}{N}\left(\frac{\Delta_{\xxx}}{\rho_{n,\xxx}}\right)^{\ell}\right) \label{eq:normalization}\\
		&\le e^{C  \ln(N)^{1-\alpha}} \rho_{n,\xxx}^\ell\left(1 + \frac{ \ell}{\sqrt{N}}\right)\notag
	\end{align}
	where in the second line we used that owing to \cref{lem:normeplusque1} we have $ \rho_{n,\sss} \ge 1$ and $\rho_{n,\ccc} \ge 1-r$ and we considered $\ell  \ge \frac{(1-c^\prime)\ln(N)}{7 K_{\xxx}\ln(d)}$. Adjusting the constant $C$ and for $N$ large enough we have the wanted result.
	
	\subsection*{Acknowledgments}
	I would like to thank Charles Bordenave for the numerous discussions that led to the main ideas of this article and for his constant supervision of the following work. I would also like to thank Camille Male for his interest in my work and for answering my questions about his work, in particular about the freeness of amalgamation. I thank Guillaume Aubrun for his supervision and feedback, even though the pandemic context made it logistically difficult. Finally, I would like to thank Adam Arras, Alexandre Gaudilli{\`e}re, Nordine Moumeni, and in general all the people who, besides my supervisor, discussed my work or theirs with me and enriched my knowledge, understanding, or confidence in the fields of random matrices and graph theory.
	\bibliographystyle{alpha}
	\bibliography{biblio.bib}
\end{document}